\documentclass{amsart}
\usepackage{amsmath,amssymb}
\title[Nakayama functors for coalgebras]{Nakayama functors for coalgebras and their applications to Frobenius tensor categories}
\author[T. Shibata]{Taiki Shibata}
\address[T. Shibata]{Department of Applied Mathematics,
  Okayama University of Science \\
  1-1 Ridai-cho, Kita-ku Okayama-shi, Okayama 700-0005, Japan.}
\email{shibata@ous.ac.jp}
\author[K. Shimizu]{Kenichi Shimizu}
\address[K. Shimizu]{Department of Mathematical Sciences,
  Shibaura Institute of Technology \\
  307 Fukasaku, Minuma-ku, Saitama-shi, Saitama 337-8570, Japan.}
\email{kshimizu@shibaura-it.ac.jp}


\date{}

\usepackage{mathtools}

\usepackage{bbm}
\usepackage[setpagesize=false]{hyperref}
\usepackage{tikz}
\usepackage{tikz-cd}
\usetikzlibrary{calc}
\usepackage{accents}
\usepackage{amsthm}
\numberwithin{equation}{section}
\theoremstyle{plain}
\newtheorem{C}{}[section] 
\newtheorem{lemma}[C]{Lemma}

\newtheorem{theorem}[C]{Theorem}
\newtheorem{corollary}[C]{Corollary}
\theoremstyle{definition}
\newtheorem{definition}[C]{Definition}
\theoremstyle{remark}
\newtheorem{remark}[C]{Remark}
\newtheorem{example}[C]{Example}
\newcommand{\id}{\mathrm{id}}
\newcommand{\op}{\mathrm{op}}
\newcommand{\cop}{\mathrm{cop}}
\newcommand{\radj}{\mathrm{ra}}
\newcommand{\rradj}{\mathrm{rra}}

\newcommand{\ladj}{\mathrm{la}}
\newcommand{\lladj}{\mathrm{lla}}

\newcommand{\bfk}{\Bbbk}
\newcommand{\Obj}{\mathrm{Obj}}
\newcommand{\Hom}{\mathrm{Hom}}
\newcommand{\End}{\mathrm{End}}

\newcommand{\Img}{\mathrm{Im}}
\newcommand{\socle}{\mathop{\mathrm{{soc}}}}
\newcommand{\coHom}{\mathrm{coHom}}

\newcommand{\rat}{\mathrm{rat}}
\newcommand{\Rat}{\mathrm{Rat}}
\newcommand{\Inj}{\mathrm{Inj}}
\newcommand{\Proj}{\mathrm{Proj}}

\newcommand{\unitobj}{\mathbf{1}}
\newcommand{\rev}{\mathrm{rev}}
\newcommand{\eval}{\mathrm{ev}}
\newcommand{\coev}{\mathrm{coev}}

\newcommand{\copow}{\otimes}

\renewcommand{\L}{\mathrm{l}}
\newcommand{\R}{\mathrm{r}}

\newcommand{\Nak}{\mathbb{N}}
\newcommand{\Nakl}{\Nak^{\L}}
\newcommand{\Nakr}{\Nak^{\R}}

\newcommand{\Mod}{\mathfrak{M}}
\newcommand{\qfcomod}{\mathfrak{M}_{\mathtt{qf}}}
\newcommand{\fdmod}{\Mod_{\fd}}
\newcommand{\Vect}{\mathbf{Vec}}
\newcommand{\Sets}{\mathbf{Set}}

\newcommand{\sfInj}{\mathfrak{I}_{\mathtt{sf}}}
\newcommand{\fdInj}{\mathfrak{I}_{\mathtt{fd}}}
\newcommand{\fdPro}{\mathfrak{P}_{\mathtt{fd}}}

\newcommand{\Ind}{\mathrm{Ind}}

\newcommand{\fd}{\mathtt{fd}} 
\newcommand{\fg}{\mathtt{fg}} 
\newcommand{\fp}{\mathtt{fp}} 
\newcommand{\qf}{\mathtt{qf}} 

\newcommand{\src}{\mathfrak{s}}
\newcommand{\tgt}{\mathfrak{t}}
\newcommand{\Span}{\mathrm{span}}

\newcommand{\triv}{\mathrm{triv}}
\newcommand{\modobj}{\mathtt{g}}
\newcommand{\RadIso}{\mathtt{r}}
\newcommand{\DriIso}{\mathtt{u}}
\newcommand{\PivIso}{\mathtt{p}}
\newcommand{\trace}{\mathsf{tr}}
\newcommand{\ptrace}{\mathsf{ptr}}

\newcommand{\YD}{\mathfrak{YD}}

\newcommand{\actl}{\mathbin{\scriptstyle\succ}}
\newcommand{\actr}{\mathbin{\scriptstyle\prec}}

\begin{document}

\maketitle

\begin{abstract}  We introduce Nakayama functors for coalgebras and investigate their basic properties. These functors are expressed by certain (co)ends as in the finite case discussed by Fuchs, Schaumann, and Schweigert.
  This observation allows us to define Nakayama functors for Frobenius tensor categories in an intrinsic way. As applications, we establish the categorical Radford $S^4$-formula for Frobenius tensor categories and obtain some related results. These are generalizations of works of Etingof, Nikshych, and Ostrik on finite tensor categories and some known facts on co-Frobenius Hopf algebras.
\end{abstract}

\tableofcontents

\section{Introduction}

Throughout this paper, $\bfk$ denotes the base field.
Let $A$ be a finite-dimensional algebra, and let ${}_A\fdmod$ be the category of finite-dimensional left modules over $A$.
Given a vector space $X$, we denote its dual space by $X^*$.
In representation theory, the endofunctor on ${}_A\fdmod$ defined by $M \mapsto A^* \otimes_A M$ is called the Nakayama functor. Fuchs, Schaumann and Schweigert \cite{MR4042867} pointed out that this functor is expressed as the coend
\begin{equation}
  \label{eq:intro-rex-Nakayama-algebra}
  A^* \otimes_A M = \int^{X \in {}_A\fdmod} \Hom_A(M, X)^* \otimes_{\bfk} X
\end{equation}
for $M \in {}_A \fdmod$. A finite abelian category \cite{MR3242743} is a linear category that is equivalent to ${}_A \fdmod$ for some finite-dimensional algebra $A$. The coend formula \eqref{eq:intro-rex-Nakayama-algebra} of the Nakayama functor allows us to define the right exact Nakayama functor $\Nakr_{\mathcal{M}}: \mathcal{M} \to \mathcal{M}$ for any finite abelian category $\mathcal{M}$ without referencing an algebra $A$ such that $\mathcal{M} \approx {}_A\fdmod$ so that $\Nakr_{\mathcal{M}} \cong A^* \otimes_A (-)$ if $\mathcal{M} = {}_A \fdmod$. As discussed in \cite{MR4042867}, such an abstract treatment of the Nakayama functor is very useful in the theory of finite tensor categories and their module categories. For example, basic properties of the Nakayama functor imply the categorical Radford $S^4$-formula for finite tensor categories \cite{MR2097289}, a Frobenius property of tensor functors between finite tensor categories \cite{MR3569179}, and a criterion for a finite tensor category to be symmetric Frobenius \cite{MR4042867}. Also, it was recently revealed that the Nakayama functor plays a crucial role in the modified trace theory for finite tensor categories \cite{2021arXiv210313702S,2021arXiv210315772S}.

Motivated by these results, we desire to define the Nakayama functor for some wider classes of linear categories as a universal object like~\eqref{eq:intro-rex-Nakayama-algebra} to extend the above results on finite tensor categories.
A locally finite abelian category \cite{MR3242743} is a linear category that is equivalent to $\fdmod^C$ for some coalgebra $C$, where $\fdmod^C$ denotes the category of finite-dimensional right $C$-comodules.
In this paper, we show that there are endofunctors on a particular class of locally finite abelian categories expressed by (co)ends similar to \eqref{eq:intro-rex-Nakayama-algebra}.
As applications, we extend several known results for finite tensor categories to {\em Frobenius tensor categories}, that is, a tensor category of which every simple object has
an injective hull \cite{MR3410615}. For example, we show that the categorical Radford $S^4$-formula for finite tensor categories and some related results given in \cite{MR2097289} also hold for Frobenius tensor categories. Our results also cover some facts on co-Frobenius Hopf algebras given in \cite{MR2236104,MR2278058,MR2554186}.

The authors' opinion is that, in general, Nakayama functors cannot be defined for locally finite abelian categories by a (co)end formula like \eqref{eq:intro-rex-Nakayama-algebra}. To explain this, let $C$ be a coalgebra.
We recall that $C^*$ has a canonical structure of an algebra and every left (right) $C$-comodule is naturally a right (left) $C^*$-module.
Given a left $C^*$-module $X$, we denote its rational part by $X^{\rat}$ (see Subsection \ref{subsec:notations}).
In this paper, we investigate two endofunctors $\Nakl_C$ and $\Nakr_C$ on the category $\Mod^C$ of all right $C$-comodules defined by
\begin{equation}
  \label{eq:intro-Nakayama-def}
  \Nakl_C(M) = \Hom^C(C, M)^{\rat}
  \quad \text{and} \quad
  \Nakr_C(M) = C \otimes_{C^*} M
\end{equation}
for $M \in \Mod^C$, where $\Hom^C$ is the Hom functor for $\Mod^C$ (see Subsection \ref{subsec:Nakayama-coalg-def} for detail).
The Nakayama functor of Chin and Simson \cite{MR2684139} is actually a restriction of the functor $\Nakl_C$ (Remark \ref{rem:Chin-Simson}).
By a Tannaka theoretic argument, we show that the functors $\Nakl_C$ and $\Nakr_C$ are expressed as (co)ends as
\begin{align}
  \label{eq:intro-Nakayama-coalgebra-1}
  \Nakl_C(M) & = \int_{X \in \fdmod^C} \Hom^C(X, M) \otimes_{\bfk} X, \\
  \label{eq:intro-Nakayama-coalgebra-2}
  \Nakr_C(M) & = \int^{X \in \fdmod^C} \coHom^C(X, M) \otimes_{\bfk} X    
\end{align}
for $M \in \Mod^C$, where $\coHom^C$ is the `coHom' functor introduced by Takeuchi \cite{MR472967} (see Theorem \ref{thm:Nakayama-(co)end} for the precise meaning of these formulas).
The reason why \eqref{eq:intro-Nakayama-coalgebra-1} and \eqref{eq:intro-Nakayama-coalgebra-2} cannot be used to define Nakayama functors for locally finite abelian categories is that, in general, the full subcategory $\fdmod^C$ is closed under neither $\Nakl_C$ nor $\Nakr_C$ (see Subsection~\ref{subsec:counterexamples}).

Based on pioneering works on coalgebras and comodules \cite{MR472967,MR1717358,MR1786197,MR1904645,MR1998048,MR2078404,MR2253657,MR2684139,MR3125851,MR3150709}, we investigate behaviors of the functors \eqref{eq:intro-Nakayama-def} in the case where the coalgebra $C$ is semiperfect, co-Frobenius or quasi-co-Frobenius (QcF) (see Definitions~\ref{def:semiperfect}, \ref{def:co-Frob} or \ref{def:QcF}, respectively).
One of important observations is that $\fdmod^C$ is closed under both $\Nakl_C$ and $\Nakr_C$ if $C$ is semiperfect (Theorem~\ref{thm:Nakayama-semiperfect}). Thus Nakayama functors are defined for a locally finite abelian category that is equivalent to $\fdmod^C$ for some semiperfect coalgebra $C$. We also show that $\Nakl_C$ and $\Nakr_C$ are equivalences if and only if $C$ is QcF, and that these functors are isomorphic to the identity functor if and only if $C$ is a symmetric coalgebra (Theorems~\ref{thm:Nakayama-QcF} and \ref{thm:Nakayama-symmetric}).

Since any Frobenius tensor category is equivalent to $\fdmod^C$ for some QcF coalgebra $C$ (Theorem~\ref{thm:Frob-tensor-cat-def}), one can apply our results to such tensor categories.
As an application, we obtain the categorical Radford $S^4$-formula for Frobenius tensor categories mentioned in the above.
Other applications are introduced during explaining the organization of this paper in the below.

\subsection*{Organization of this paper}

This paper is organized as follows: In Section~\ref{sec:preliminaries}, we fix notations used throughout in this paper. We also recall some basic results on semiperfect coalgebras, the coHom functor, and (co)ends.

In Section~\ref{sec:Naka-for-coalgebras}, for a coalgebra $C$, we introduce two endofunctors $\Nakl_C$ and $\Nakr_C$ on the category $\Mod^C$. We first show that $\Nakr_C$ is left adjoint to $\Nakl_C$ (Lemma~\ref{lem:Nakayama-adj}). Thus, in particular, $\Nakl_C$ and $\Nakr_C$ are left and right exact, respectively. Next, we show that $\Nakl_C$ and $\Nakr_C$ are expressed by (co)ends as in \eqref{eq:intro-Nakayama-coalgebra-1} and \eqref{eq:intro-Nakayama-coalgebra-2} (Theorem~\ref{thm:Nakayama-(co)end}). Their universal properties imply that $\Nakl_C$ and $\Nakr_C$ are Morita-Takeuchi invariant in the following sense: If $\Phi: \Mod^C \to \Mod^D$ is an equivalence of linear categories, where $C$ and $D$ are coalgebras, then there are isomorphisms $\Phi \circ \Nakl_C \cong \Nakl_D \circ \Phi$ and $\Phi \circ \Nakr_C \cong \Nakr_D \circ \Phi$ of functors (Lemma~\ref{lem:Nakayama-Mori-Take-equiv}).

In the most general setting, $\Nakl_C$ and $\Nakr_C$ do not seem to benefit the study of coalgebras. There is, even, an example of a coalgebra $C$ such that both $\Nakl_C$ and $\Nakr_C$ are zero functors (Example~\ref{ex:counter-example-3}). However, if we assume that $C$ is semiperfect, then $\Nakl_C$ and $\Nakr_C$ behave well as in the case of finite-dimensional algebras. For example, for a semiperfect coalgebra $C$, we show that $\Nakl_C$ and $\Nakr_C$ preserve $\fdmod^C$ and they induce equivalences between the full subcategory of projective objects of $\fdmod^C$ and the full subcategory of injective objects of $\fdmod^C$ (Theorem~\ref{thm:Nakayama-semiperfect}).

For a co-Frobenius coalgebra $C$, we show that $\Nakl_C$ and $\Nakr_C$ are isomorphic to the functor twisting the coaction by $\nu^{-1}$ and $\nu$, respectively, where $\nu$ is the Nakayama automorphism of $C$ (Lemma~\ref{lem:Nakayama-co-Fro}). By this result and the Morita-Takeuchi invariance, we see that $\Nakl_C$ and $\Nakr_C$ are equivalences precisely if $C$ is QcF (Theorem~\ref{thm:Nakayama-QcF}), and that $\Nakl_C$ and $\Nakr_C$ are identity functors precisely if $C$ is a symmetric coalgebra (Theorem~\ref{thm:Nakayama-symmetric}).
Now suppose that $C$ is a QcF coalgebra. Given a simple right $C$-comodule $S$, we denote by $E(S)$ and $P(S)$ the injective hull and the projective cover of $S$, respectively. Then we have \begin{equation}
  \label{eq:intro-inj-hull-top}
  \socle P(S) \cong \Nakl_C(S)
  \quad \text{and} \quad
  \mathop{\mathrm{top}} E(S) \cong \Nakr_C(S)
\end{equation}
as in the case of finite-dimensional algebras (Theorem~\ref{thm:Nakayama-permutation}).

In Section~\ref{sec:Naka-for-locally-finite}, we reformulate results of Section~\ref{sec:Naka-for-coalgebras} in the setting of locally finite abelian categories. Given a category $\mathcal{A}$, we denote its ind-completion by $\Ind(\mathcal{A})$.
We note that a locally finite abelian category is equivalent to $\fdmod^C$ for some coalgebra $C$ and $\Ind(\fdmod^C)$ can be identified with $\Mod^C$. Now let $\mathcal{A}$ be a locally finite abelian category. In view of \eqref{eq:intro-Nakayama-coalgebra-1} and \eqref{eq:intro-Nakayama-coalgebra-2}, we can define two endofunctors $\Nakl_{\Ind(\mathcal{A})}$ and $\Nakr_{\Ind(\mathcal{A})}$ on $\Ind(\mathcal{A})$ by
\begin{align*}
  \Nakl_{\Ind(\mathcal{A})}(M)
  & = \int_{X \in \mathcal{A}} \Hom_{\Ind(\mathcal{A})}(X, M) \copow X, \\
  \Nakr_{\Ind(\mathcal{A})}(M)
  & = \int^{X \in \mathcal{A}} \coHom_{\Ind(\mathcal{A})}(X, M) \copow X
\end{align*}
for $M \in \Ind(\mathcal{A})$, where $\copow$ is the copower and $\coHom_{\Ind(\mathcal{A})}$ is the coHom functor for $\Ind(\mathcal{A})$ (see Section~\ref{sec:preliminaries}). By Theorem~\ref{thm:Nakayama-semiperfect}, we see that $\Nakl_{\Ind(\mathcal{A})}$ and $\Nakr_{\Ind(\mathcal{A})}$ restrict to endofunctors on $\mathcal{A}$ if $\mathcal{A}$ has enough injectives and projectives (Theorem~\ref{thm:Nakayama-compact-restriction}).
We also show the following generalization of \cite[Theorem 3.18]{MR4042867}:
Let $\mathcal{A}$ and $\mathcal{B}$ be locally finite abelian categories.
Then, for every linear functor $F: \Ind(\mathcal{A}) \to \Ind(\mathcal{B})$ such that $F$ preserves limits and the double right adjoint $F^{\rradj}$ of $F$ exists, there is an isomorphism $F \circ \Nakl_{\Ind(\mathcal{A})} \cong \Nakl_{\Ind(\mathcal{B})} \circ F^{\rradj}$ (Theorem~\ref{thm:Nakayama-and-adjunctions}). A similar result holds also for $\Nakr$.

In Section~\ref{sec:Frobenius-tensor}, we give applications of our results to Frobenius tensor categories.
If $H$ is a Hopf algebra, then $\fdmod^H$ is a left rigid monoidal category, however, it is not right rigid in general.
This means that $\fdmod^H$ may not be a tensor category in the sense of \cite{MR3242743}.
Mentioning this problem, in this section, we show the following criterion by utilizing Nakayama functors: Let $\mathcal{C}$ be a locally finite abelian category equipped with a structure of a monoidal category such that the tensor product $\otimes$ of $\mathcal{C}$ is bilinear and exact, and the unit object $\unitobj \in \mathcal{C}$ is simple. Suppose that $\mathcal{C}$ has a non-zero object that is either injective or projective. Then $\mathcal{C}$ is left rigid if and only if $\mathcal{C}$ is right rigid (Theorem~\ref{thm:one-sided-rigidity}).
This result is a generalization of the fact that the antipode of a co-Frobenius Hopf algebra is bijective.

The discussion for proving the above result also gives some equivalent conditions for a tensor category to be Frobenius (Theorem~\ref{thm:Frob-tensor-cat-def}). Now let $\mathcal{C}$ be a Frobenius tensor category with left duality $(-)^{\vee}$. In this case, the functor $\Nakr := \Nakr_{\Ind(\mathcal{C})}$ restricts to an endofunctor on $\mathcal{C}$. By the basically same way as \cite{MR4042867}, we obtain natural isomorphisms
\begin{equation}
  \label{eq:intro-Nakayama-iso}
  \Nakr(W \otimes X) \cong \Nakr(W) \otimes X^{\vee\vee}
  \quad \text{and} \quad
  X \otimes \Nakr(Y) \cong \Nakr(X^{\vee\vee} \otimes Y)
\end{equation}
for $W, X, Y \in \mathcal{C}$. Hence, letting $\modobj_{\mathcal{C}} := \Nakr(\unitobj)$, we have
\begin{equation}
  \label{eq:intro-Nakayama-iso-2}
\Nakr(X) \cong \modobj_{\mathcal{C}} \otimes X^{\vee\vee}
\end{equation}
for $X \in \mathcal{C}$.
It turns out that the object $\modobj_{\mathcal{C}}$ is a categorical analogue of the distinguished grouplike element of a co-Frobenius Hopf algebra (see Section~\ref{sec:hopf-algebras}).
By \eqref{eq:intro-inj-hull-top} and \eqref{eq:intro-Nakayama-iso-2}, we obtain a formula of the top of the injective hull of a simple object of $\mathcal{C}$ (Corollary~\ref{cor:ACE15-Rem-2-10}).

By \eqref{eq:intro-Nakayama-iso}, we also obtain a natural isomorphism
\begin{equation}
  \label{eq:intro-Radford-iso}
  \RadIso_X: X \otimes \modobj_{\mathcal{C}} \to\modobj_{\mathcal{C}} \otimes X^{\vee\vee\vee\vee}
\end{equation}
for $X \in \mathcal{C}$, which we call the {\em Radford isomorphism} (Definition~\ref{def:Radford-iso}). We give explicit formulas for this isomorphism in the two cases where $\mathcal{C}$ is braided and $\mathcal{C}$ is semisimple (Theorems~\ref{thm:Radford-iso-braiding} and \ref{thm:Radford-iso-semisimple}).
These results are, in fact, generalizations of the categorical Radford $S^4$-formula and related results given in \cite{MR2097289}.
We also define a spherical tensor category in the spirit of \cite{MR4254952} using the isomorphism~\eqref{eq:intro-Radford-iso} and show that a braided spherical tensor category is a ribbon category (Theorem~\ref{thm:spherical-ribbon}).

In Section~\ref{sec:hopf-algebras}, we explain how our results are applied to Hopf algebras. In particular, we derive several known equivalent conditions for a Hopf algebra to be co-Frobenius as a corollary of our results on tensor categories (Theorem~\ref{thm:Hopf-semiperfect}).
Let $H$ be a co-Frobenius Hopf algebra.
By a generalization of the fundamental theorem for Hopf modules \cite{MR1629389}, we also show that the space of left cointegrals on $H$ becomes an object of the category ${}^{\tilde{H}}_H \YD$ of `twisted' Yetter-Drinfeld modules over $H$.
From this observation, we obtain, for example, an explicit formula of the Nakayama automorphism of $H$ (Lemma~\ref{lem:co-Fb-Hopf-Nakayama-2}).
The `twisted' Drinfeld center $\mathfrak{Z}_{0,4}$ is introduced as a category-theoretical counterpart of ${}^{\tilde{H}}_H \YD$.
We also show that the Radford $S^4$-formula for $H$ given in \cite{MR2278058} and the Radford isomorphism \eqref{eq:intro-Radford-iso} for $\fdmod^H$ are equivalent in the sense that they give rise to isomorphic objects in $\mathfrak{Z}_{0,4}$.

\subsection*{Acknowledgment}

The authors thank the anonymous referee for the careful reading of manuscript.
The first author (T.S.) is supported by JSPS KAKENHI Grant Numbers JP19K14517 and JP22K13905.
The second author (K.S.) is supported by JSPS KAKENHI Grant Number JP20K03520.

\subsection*{Competing interests}

The authors declare no conflict of interest.

\section{Preliminaries}
\label{sec:preliminaries}

\subsection{Notations}
\label{subsec:notations}

Our basic reference on the category theory is Mac Lane \cite{MR1712872}.
Given a category $\mathcal{C}$, we denote by $\mathcal{C}^{\op}$ its opposite category.
The class of objects of $\mathcal{C}$ is written as $\Obj(\mathcal{C})$, however, we also write $X \in \mathcal{C}$ to mean that $X$ is an object of the category $\mathcal{C}$.
An object $X \in \mathcal{C}$ is written as $X^{\op}$ when it is regarded as an object of $\mathcal{C}^{\op}$.
For a functor $F$, we denote by $F^{\ladj}$ and $F^{\radj}$ a left and a right adjoint of $F$, respectively, if they exist.

Throughout this paper, we work over the field $\bfk$.
The category of all vector spaces (over $\bfk$) is denoted by $\Vect$.
By a (co)algebra, we always mean a (co)associative (co)unital (co)algebra over the field $\bfk$.
Given an algebra $A$, we denote by ${}_A\Mod$ and $\Mod_A$ the category of left and right $A$-modules, respectively.
Similarly, given a coalgebra $C$, we denote by ${}^C\Mod$ and $\Mod^C$ the category of left and right $C$-comodules, respectively.
The Hom functors for ${}_A\Mod$, $\Mod_A$, ${}^C\Mod$ and $\Mod^C$ are written as ${}_A \Hom$, $\Hom_A$, ${}^C\Hom$ and $\Hom^C$, respectively. We use the subscript $\fd$ to mean the full subcategory of finite-dimensional objects. For example, $\Mod^C_{\fd}$ is the category of finite-dimensional right $C$-comodules.

The comultiplication and the counit of a coalgebra $C$ is usually denoted by $\Delta : C \to C \otimes_{\bfk} C$ and $\varepsilon : C \to \bfk$, respectively. The Sweedler notation, such as
\begin{equation*}
  \Delta(c) = c_{(1)} \otimes c_{(2)}
  \quad \text{and} \quad
  c_{(1)} \otimes \Delta(c_{(2)}) = c_{(1)} \otimes c_{(2)} \otimes c_{(3)} = \Delta(c_{(1)}) \otimes c_{(2)},
\end{equation*}
will be used to express the comultiplication of $c \in C$. We denote by $C^{\cop}$ the coalgebra whose underlying space is $C$ and the comultiplication is given by $c \mapsto c_{(2)} \otimes c_{(1)}$.

The coaction of $C$ on a right $C$-comodule $M$ is usually denoted by $\delta_M : M \to M \otimes_\bfk C$ and expressed as $m \mapsto m_{(0)} \otimes m_{(1)}$. If $M$ is a bicomodule, the left and the right coactions on $M$ are denoted by $\delta^{\L}_{M}$ and $\delta^{\R}_{M}$, respectively.

Given a vector space $V$, we denote by $V^* = \Hom_{\bfk}(V, \bfk)$ the dual space of $V$.
For $f \in V^*$ and $v \in V$, $f(v)$ is often written as $\langle f, v \rangle$.
The dual space $C^*$ of a coalgebra $C$ is an algebra with respect to the convolution product defined by $\langle f * g, c \rangle = \langle f, c_{(1)} \rangle \langle g, c_{(2)} \rangle$ for $f, g \in C^*$ and $c \in C$.
Every right $C$-comodule $M$ is naturally a left $C^*$-module by the action defined by $c^* \rightharpoonup m = m_{(0)} \langle c^*, m_{(1)} \rangle$ for $c^* \in C^*$ and $m \in M$, and this construction extends to a fully faithful functor
\begin{equation}
  \label{eq:comod-C-to-C*-mod}
  \Mod^{C} \to {}_{C^*}\Mod,
  \quad (M, \delta_M) \mapsto (M, \rightharpoonup).
\end{equation}
A left $C^*$-module is said to be {\em rational} if it belongs to the image of the functor \eqref{eq:comod-C-to-C*-mod}.
Every left $C^*$-module $M$ has the largest rational submodule $M^{\rat}$, called the {\em rational part} of $M$.
By definition, we may and do view $M^{\rat}$ as a right $C$-comodule.
The assignment $M \mapsto M^{\rat}$ extends the functor
\begin{equation}
  \label{eq:rational-functor}
  (-)^{\rat} : {}_{C^*}\Mod \to \Mod^{C},
  \quad M \mapsto M^{\rat},
\end{equation}
which is right adjoint to~\eqref{eq:comod-C-to-C*-mod}.

A left $C$-comodule is regarded as a right $C^*$-module in a similar way.
The rational part of a right $C^*$-module $M$ is also defined in an analogous way and denoted by the same symbol $M^{\rat}$.
This notation is ambiguous when $M$ is a $C^*$-bimodule.
Thus, in that case, we denote by $({}_{C^*}M)^{\rat}$ and $(M_{C^*})^{\rat}$ the rational part of $M$ as a left and a right $C^*$-module, respectively.

\subsection{Basics on semiperfect coalgebras}

For the theory of coalgebras and comodules, we refer the reader to \cite{MR1786197}.

\begin{definition}
\label{def:semiperfect}
A coalgebra $C$ is said to be {\em left} ({\em right}) {\em semiperfect} \cite[Definition 3.2.4]{MR1786197} if every finite-dimensional left (right) $C$-comodule has a projective cover in $\Mod^C$. A coalgebra $C$ is said to be {\em semiperfect} if it is both left and right semiperfect.
\end{definition}

We recall some results on semiperfect coalgebras, which are especially important in this paper.
Let, in general, $A$ be an algebra, and let $M$ be a left $A$-module. Then $M^*$ is a right $A$-module by the action $\langle m^* \cdot a, m \rangle = \langle m^*, a m \rangle$ ($m^* \in M^*$, $m \in M$, $a \in A$). For a while, we let $C$ be an arbitrary coalgebra. First of all, we note:

\begin{lemma}[{\cite[Lemma 2.2.12]{MR1786197}}]
  \label{lem:DNR-Lem-2-2-12}
  If $M$ is a finite-dimensional rational left $C^*$-module, then $M^*$ is a rational right $C^*$-module.
\end{lemma}

In other words, if $M$ is a finite-dimensional right $C$-comodule, then $M^*$ is a left $C$-comodule.
The left coaction of $C$ on $M^*$, which we denote by $m^* \mapsto m^*_{(-1)} \otimes m^*_{(0)}$, is characterized by the equation
\begin{equation}
  \label{eq:dual-coaction}
  \langle m^*_{(-1)}, m \rangle m^*_{(0)} = \langle m^*, m_{(0)} \rangle m_{(1)}
  \quad (m^* \in M^*, m \in M).
\end{equation}

Given an abelian category $\mathcal{A}$, we denote by $\Proj(\mathcal{A})$ and $\Inj(\mathcal{A})$ the full subcategory of projective and injective objects of $\mathcal{A}$, respectively.
By Lemma~\ref{lem:DNR-Lem-2-2-12}, the assignment $M \mapsto M^*$ gives rise to an anti-equivalence between the categories $\fdmod^C$ and ${}^C\fdmod$.
Thus the assignment $M \mapsto M^*$ also gives equivalences
\begin{equation*}
  \Inj(\fdmod^C) \approx \Proj({}^C\fdmod)^{\op}
  \quad \text{and} \quad
  \Proj(\fdmod^C) \approx \Inj({}^C\fdmod)^{\op}.
\end{equation*}
We warn that, for an algebra $A$, a projective object of ${}_A\fdmod$ may not be a projective object of ${}_A \Mod$. For the case of coalgebras, the situation is easy.
To explain this, we first remark:

\begin{lemma}[{\cite[Corollary 2.4.20]{MR1786197}}]
  \label{lem:DNR-Cor-2-4-20}
  An object $X \in \fdmod^C$ is an injective (projective) object of $\Mod^C$ if and only if $X^*$ is a projective (injective) object of ${}^C\Mod$.
\end{lemma}

By assembling some results on coalgebras in \cite{MR1786197}, one obtains:

\begin{lemma}
  \label{lem:fd-injective}
  For a coalgebra $C$, we have
  \begin{equation*}
    \Inj(\fdmod^C) = \Inj(\Mod^C) \cap \Obj(\fdmod^C)
    \quad \text{and} \quad
    \Proj(\fdmod^C) = \Proj(\Mod^C) \cap \Obj(\fdmod^C).
  \end{equation*}
\end{lemma}
\begin{proof} 
  The first equation is a rephrasing of (i) $\Leftrightarrow$ (ii) of \cite[Theorem 2.4.17]{MR1786197}.
  We prove the second one.
  Suppose $P \in \Proj(\fdmod^C)$.
  Then $P^*$ is an injective object of ${}^C\fdmod$, and hence it is also an injective object of ${}^C\Mod$ by the first equation applied to $C^{\cop}$.
  Thus, by Lemma~\ref{lem:DNR-Cor-2-4-20}, $P$ is a projective object of $\Mod^C$. The ``$\subset$'' part is proved. The converse inclusion is clear. The proof is done.
\end{proof}

The following criteria will be used throughout this paper:

\begin{lemma}[{\cite[Theorem 3.2.3]{MR1786197}}]
  \label{lem:DNR-thm-323}
  For a coalgebra $C$, the following are equivalent:
  \begin{enumerate}
  \item $C$ is right semiperfect.
  \item $\Mod^C$ has enough projective objects.
  \item Every simple left $C$-comodule has a finite-dimensional injective hull in ${}^C\Mod$.
  \item $({}_{C^*}C^{*})^{\rat}$ is dense in $C^*$ with respect to the finite topology.
  \end{enumerate}
\end{lemma}

Suppose that $C$ is right semiperfect.
Given $V \in {}^C\Mod_{\fd}$, we denote its injective hull in ${}^C\Mod$ by $E(V)$.
Lemma~\ref{lem:DNR-thm-323} implies that $E(V)$ is finite-dimensional.
As noted in the proof of \cite[Theorem 3.2.3]{MR1786197}, a projective cover of a finite-dimensional right $C$-comodule $M$ is given by $E(M^*)^*$. Thus we have:

\begin{lemma}
  \label{lem:DNR-thm-323-plus}
  Let $C$ be a right semiperfect coalgebra, and let $M \in \fdmod^C$.
  Then a projective cover of $M$ exists in $\Mod^C$ and is finite-dimensional.
\end{lemma}

If $C$ is semiperfect, then $({}_{C^*}C^*)^{\rat}$ is an idempotent ideal of $C^*$ and is equal to $(C^*_{C^*})^{\rat}$ as a subspace of $C^*$ \cite[Corollary 3.2.16]{MR1786197}.
Thus we may write $C^{*\rat} = ({}_{C^*}C^*)^{\rat} = (C^*_{C^*})^{\rat}$ for a semiperfect coalgebra $C$.

\begin{lemma}[{\cite[Lemma 2.9]{MR1717358}}]
  \label{lem:C-star-rat-M-rat}
  Suppose that $C$ is semiperfect.
  Then, for every left $C^*$-module $M$, we have
  \begin{equation*}
    M^{\rat} = \Span_{\bfk}\{ c^* \rightharpoonup m \mid c^* \in C^{*\rat}, m \in M \}.
  \end{equation*}
\end{lemma}

\subsection{The coHom functor}
\label{subsec:cohom-functor}

Let $\mathcal{A}$ be a $\bfk$-linear category, that is, a category enriched over $\Vect$.
Given objects $W \in \Vect$ and $X \in \mathcal{A}$ such that the functor $\mathcal{A} \to \Vect$ defined by $Y \mapsto \Hom_{\bfk}(W, \Hom_{\mathcal{A}}(X, Y))$ is representable, we call the object representing this functor the {\em copower} of $X$ by $W$ and denote it by $W \copow X \in \mathcal{A}$.
Thus, by definition, there is a canonical isomorphism
\begin{equation}
  \label{eq:def-copower}
  \Hom_{\bfk}(W, \Hom_{\mathcal{A}}(X, Y)) \cong \Hom_{\mathcal{A}}(W \copow X, Y)
\end{equation}
natural in the variable $Y \in \mathcal{A}$.

\begin{example}
  Let $C$ be a coalgebra. Then the copower of $X \in \Mod^C$ by $W \in \Vect$ is the vector space $W \otimes_{\bfk} X$ equipped with the right $C$-comodule structure $\id_{W} \otimes_{\bfk} \delta_X$.
\end{example}

For an object $X \in \mathcal{A}$ and a vector space $W$ with $\alpha = \dim_{\bfk} W$ (which is possibly infinite), the copower $W \copow X$ exists if and only if the direct sum $X^{\oplus \alpha}$ exists. Furthermore, if they exist, there is an isomorphism $W \copow X \cong X^{\oplus \alpha}$ in $\mathcal{A}$. We note that this isomorphism is `not canonical' in the sense that it depends on the choice of a basis of $W$.

If $\mathcal{W} \subset \Vect$ and $\mathcal{X} \subset \mathcal{A}$ are full subcategories such that $W \copow X$ exists for all $W \in \mathcal{W}$ and $X \in \mathcal{X}$, then the assignment $(W, X) \mapsto W \copow X$ extends to a functor from $\mathcal{W} \times \mathcal{X}$ to $\mathcal{A}$ in such a way that the isomorphism \eqref{eq:def-copower} is also natural in the variables $W$ and $X$.

Let $\mathcal{A}$ be a $\bfk$-linear category, and let $X \in \mathcal{A}$ be an object such that the copower $W \copow X$ exists for all $W \in \Vect$. Let $Y$ also be an object of $\mathcal{A}$. If the functor $\Vect \to \Vect$, $W \mapsto \Hom_{\mathcal{A}}(Y, W \copow X)$ is representable, then we call the object representing this functor the {\em coHom} space from $X$ to $Y$ and denote it by $\coHom_{\mathcal{A}}(X, Y)$.
Thus, by definition, there is an isomorphism
\begin{equation}
  \label{eq:def-coHom}
  \Hom_{\bfk}(\coHom_{\mathcal{A}}(X, Y), W) \cong \Hom_{\mathcal{A}}(Y, W \copow X)
\end{equation}
natural in $W \in \Vect$. If $\mathcal{X}$ and $\mathcal{Y}$ are full subcategories of $\mathcal{A}$ such that $\coHom_{\mathcal{A}}(X,Y)$ exists for all $X \in \mathcal{X}$ and $Y \in \mathcal{Y}$, then the assignment $(X, Y) \mapsto \coHom_{\mathcal{A}}(X,Y)$ extends to a functor from $\mathcal{X}^{\op}\times \mathcal{Y}$ to $\Vect$ in such a way that the isomorphism \eqref{eq:def-coHom} is also natural in $X$ and $Y$.

The definition of $\coHom_{\mathcal{A}}$ is based on \cite{MR472967}, where the coHom functor for the category of comodules was introduced and used to establish Morita theory for coalgebras.
Let $C$ be a coalgebra. For simplicity, we write
\begin{equation*}
  \coHom^C(X, Y) := \coHom_{\Mod^C}(X, Y)
\end{equation*}
for $X, Y \in \Mod^C$ (if the right-hand side exists).
A right $C$-comodule $X$ is said to be {\em quasi-finite} \cite{MR472967} if $\dim_{\bfk} \Hom^C(V, X) < \infty$ for all $V \in \fdmod^C$.
We denote by $\qfcomod^C$ the full subcategory of $\Mod^C$ consisting of all quasi-finite comodules.
It is known that $\coHom^C(X, Y)$ exists for all $Y \in \Mod^C$ if and only if $X \in \qfcomod^C$. Moreover, for $X \in \qfcomod^C$ and $Y \in \Mod^C$, we have
\begin{equation}
  \label{eq:coHom-varinjlim}
  \coHom^C(X, Y)
  = \varinjlim_{\lambda \in \Lambda} \Hom^C(Y_{\lambda}, X)^*,
\end{equation}
where $\{ Y_{\lambda} \}_{\lambda \in \Lambda}$ is the directed set of all finite-dimensional subcomodules of $Y$ \cite[Proposition 1.3]{MR472967}.
This expression of $\coHom^C$ may not be convenient for practical use. We note the following useful Lemmas \ref{lem:coHom-finite-1} and \ref{lem:coHom-finite-2}.

\begin{lemma}
  \label{lem:coHom-finite-1}
  For $X \in \qfcomod^C$ and $Y \in \fdmod^C$, we have
  \begin{equation*}
    \coHom^C(X, Y) \cong \Hom^C(Y, X)^*.
  \end{equation*}
\end{lemma}
\begin{proof}
  Obvious from \eqref{eq:coHom-varinjlim}.
\end{proof}

\begin{lemma}
  \label{lem:coHom-finite-2}
  For $X \in \fdmod^C$ and $Y \in \Mod^C$, we have
  \begin{equation*}
    \coHom^C(X, Y) \cong X^* \otimes_{C^*} Y.
  \end{equation*}
\end{lemma}
\begin{proof}
  By the tensor-Hom adjunction, we have
  \begin{gather*}
    \Hom_{\bfk}(X^* \otimes_{C^*} Y, W)
    \cong {}_{C^*}\Hom(Y, \Hom_{\bfk}(X^*, W))
    \cong \Hom^C(Y, W \otimes_{\bfk} X)
  \end{gather*}
  for all $W \in \Vect$, where the finite-dimensionality of $X$ is used for the second isomorphism. The proof is done.
\end{proof}

Now we suppose $X \in \qfcomod^C$. Since $\coHom^C(X, -)$ has a right adjoint, it preserves direst sums, and therefore it preserves copowers. Let $D$ be another coalgebra, and let ${}^D\Mod^C$ be the category of $D$-$C$-comodules. If $P \in {}^D\Mod^C$, then $\coHom^C(X, P)$ is a left $D$-comodule by the structure map
\begin{equation*}
  \coHom^C(X, P)
  \xrightarrow{\ \coHom^C(X, \delta^{\L}_P) \ }
  \coHom^C(X, D \otimes_{\bfk} P)
  \xrightarrow{\ \cong \ }
  D \otimes_{\bfk} \coHom^C(X, P),
\end{equation*}
where $\delta^{\L}_P : P \to D \otimes_{\bfk} P$ is the left $D$-comodule structure of $P$ and the second arrow follows from that $\coHom^C(X, -)$ preserves copowers. As noted in  \cite{MR472967}, this construction gives rise to the following functor:
\begin{equation*}
  (\qfcomod^C)^{\op} \times {}^D\Mod^C \to {}^D \Mod,
  \quad (X, P) \mapsto \coHom^C(X, P).
\end{equation*}

\subsection{Relation with coHom and adjoints}

Following the existence result on the coHom space in the category of comodules, given a $\bfk$-linear category $\mathcal{A}$, we denote by $\mathcal{A}_{\qf}$ the full subcategory of $\mathcal{A}$ consisting of all objects $X$ such that $\coHom_{\mathcal{A}}(X, Y)$ exists for all objects $Y \in \mathcal{A}$. Let $\mathcal{A}$ and $\mathcal{B}$ be $\bfk$-linear categories admitting direct sums, and let $F: \mathcal{A} \to \mathcal{B}$ be a $\bfk$-linear functor preserving the direct sums. Then we have a natural isomorphism
\begin{equation}
  \label{eq:preserves-copow}
  F(W \copow X) \cong W \copow F(X)
\end{equation}
for $X \in \mathcal{A}$ and $W \in \Vect$.
If $F(\mathcal{A}_{\qf}) \subset \mathcal{B}_{\qf}$, then there is a linear map
\begin{equation*}
  F^{\sharp}_{X,Y}: \coHom_{\mathcal{B}}(F(X), F(Y)) \to \coHom_{\mathcal{A}}(X, Y)
  \quad (X \in \mathcal{A}_{\qf}, Y \in \mathcal{A})
\end{equation*}
induced by the natural transformations
\begin{equation*}
  \newcommand{\XAR}[1]{\xrightarrow{\makebox[5em]{$\scriptstyle #1$}}}
  \begin{aligned}
    \Hom_{\bfk}(\coHom_{\mathcal{A}}(X, Y), W)
    & \XAR{\eqref{eq:def-coHom}}
    \Hom_{\mathcal{A}}(Y, W \copow X) \\
    & \XAR{f \mapsto F(f)} \Hom_{\mathcal{B}}(F(Y), F(W \copow X)) \\
    & \XAR{\eqref{eq:preserves-copow}} \Hom_{\mathcal{B}}(F(Y), W \copow F(X)) \\
    & \XAR{\eqref{eq:def-coHom}}
    \Hom_{\bfk}(\coHom_{\mathcal{B}}(F(X), F(Y)), W)
\end{aligned}
\end{equation*}
for $W \in \Vect$.

Suppose moreover that the functor $F$ has a left adjoint $F^{\ladj}$. Let $\eta$ and $\varepsilon$ be the unit and the counit of the adjunction $F^{\ladj} \dashv F$, respectively. Then there is a natural transformation
\begin{equation}
  \label{eq:coHom-iso-1}
 F^{\sharp}_{X,F^{\ladj}(Y)} \circ \coHom_{\mathcal{B}}(\id_{F(X)}, \eta_Y):\coHom_{\mathcal{B}}(F(X), Y)
  \to \coHom_{\mathcal{A}}(X, F^{\ladj}(Y))
\end{equation}
for $X \in \mathcal{A}_{\qf}$ and $Y \in \mathcal{B}$.
We note that $F^{\ladj}$ preserves direct sums (in fact any colimits) as it has a right adjoint $F$.
Thus there is also a natural transformation
\begin{equation}
  \label{eq:coHom-iso-2}
  (F^{\ladj})^{\sharp}_{F(X),Y}
  \circ \coHom_{\mathcal{A}}(\varepsilon_X, \id_{F^{\ladj}(Y)}):
  \coHom_{\mathcal{A}}(X, F^{\ladj}(Y))
  \to \coHom_{\mathcal{B}}(F(X), Y) 
\end{equation}
for $X \in \mathcal{A}_{\qf}$ and $Y \in \mathcal{B}$. One can check that \eqref{eq:coHom-iso-1} and \eqref{eq:coHom-iso-2} are mutually inverse to each other. Thus, in summary, we have a natural isomorphism
\begin{equation}
  \label{eq:adjunction-coHom}
  \coHom_{\mathcal{B}}(F(X), Y)
  \cong \coHom_{\mathcal{A}}(X, F^{\ladj}(Y))
  \quad (X \in \mathcal{A}_{\qf}, Y \in \mathcal{B}).
\end{equation}

\subsection{Ends and coends}

We assume that the reader is familiar with dealing with ends and coends \cite[IX]{MR1712872}.
Let $C$ be a coalgebra. The following coend formula for the $C$-bicomodule $C$, which should be well-known, is essential in proving the universal property of the Nakayama functors introduced in the next section.

\begin{lemma}
  The $C$-bicomodule $C$ is a coend of the functor
  \begin{equation*}
    (\fdmod^C)^{\op} \times (\fdmod^C) \to {}^C\Mod^C,
    \quad (X^{\op}, Y) \mapsto X^* \otimes_{\bfk} Y
  \end{equation*}
  with the universal dinatural transformation given by
  \begin{equation}
    \label{eq:Tannaka-dinatural}
    i_X : X^* \otimes_{\bfk} X \to C,
    \quad i_X(x^* \otimes x) = \langle x^*, x_{(0)} \rangle x_{(1)}
    \quad (x^* \in X^*, x \in X)
  \end{equation}
  for $X \in \fdmod^C$. Thus, using the integral notation for coends, we have
  \begin{equation}
    \label{eq:Tannaka}
    C = \int^{X \in \fdmod^C} X^* \otimes_{\bfk} X.
  \end{equation}
\end{lemma}
\begin{proof}
  By \eqref{eq:dual-coaction}, one can easily check that \eqref{eq:Tannaka-dinatural} is a morphism of $C$-bicomodules.
  The universal property of \eqref{eq:Tannaka-dinatural} is proved by the standard argument in the Tannaka theory; see, {\it e.g.}, \cite{MR1623637}.
\end{proof}

To close this preliminary section, we note some useful formulas for (co)ends.
Let $T: \mathcal{B}^{\op} \times \mathcal{A} \to \mathcal{V}$ and $F: \mathcal{A} \to \mathcal{B}$ be functors, where $\mathcal{A}$, $\mathcal{B}$ and $\mathcal{V}$ are arbitrary categories. Suppose that the functor $F$ admits a right adjoint. Then we have
\begin{equation}
  \label{eq:adjunction-and-coend}
  \int^{X \in \mathcal{A}} T(F(X), X) \cong \int^{Y \in \mathcal{B}}T(Y, F^{\radj}(Y))
\end{equation}
meaning that the coend on the left hand side exists precisely if the coend on the right hand side does and, if they exist, then they are canonically isomorphic \cite[Lemma 3.9]{MR2869176}. We assume that the both sides of \eqref{eq:adjunction-and-coend} exist. Let $C_1$ and $C_2$ be the left and the right hand side of \eqref{eq:adjunction-and-coend}, respectively, and let
\begin{equation*}
  i_X : T(F(X), X) \to C_1 \quad (X \in \mathcal{A})
  \quad \text{and} \quad
  j_Y : T(Y, F^{\radj}(Y)) \to C_2 \quad (Y \in \mathcal{B})
\end{equation*}
be the universal dinatural transformations for $C_1$ and $C_2$, respectively. According to the proof of \cite[Lemma 3.9]{MR2869176}, we find that the canonical isomorphism $\phi: C_1 \to C_2$ and its inverse are characterized by
\begin{equation}
  \label{eq:adjunction-and-coend-construction}
  \phi \circ i_X = j_{F(X)} \circ T(F(X), \eta_{X})
  \quad \text{and} \quad
  \phi^{-1} \circ j_Y = i_{F^{\radj}(Y)} \circ T(\varepsilon_Y, F^{\radj}(Y)),
\end{equation}
respectively, for all $X \in \mathcal{A}$ and $Y \in \mathcal{B}$, where $\eta : \id_{\mathcal{A}} \to F^{\radj} \circ F$ and $\varepsilon : F \circ F^{\radj} \to \id_{\mathcal{B}}$ are the unit and the counit for the adjunction $F \dashv F^{\radj}$, respectively.

By the dual argument, we have an analogous result for ends: If $F: \mathcal{A} \to \mathcal{B}$ is a functor admitting a left adjoint, then we have a canonical isomorphism
\begin{equation}
  \label{eq:adjunction-and-end}
  \int_{X \in \mathcal{A}} T(F(X), X) \cong \int_{Y \in \mathcal{B}}T(Y, F^{\ladj}(Y))
\end{equation}
provided that either the left or the right hand side exists.

\section{Nakayama functor for coalgebras}
\label{sec:Naka-for-coalgebras}

\subsection{Nakayama functor for coalgebras}
\label{subsec:Nakayama-coalg-def}

In this section, we introduce the left exact and the right exact Nakayama functor for coalgebras and investigate their properties.
Let $C$ be a coalgebra.
If $M$ is a right $C$-comodule, then the vector space $\Hom^{C}(C, M)$ is a left $C^*$-module by the action given by
\begin{equation*}
  (c^* \cdot f)(c)
  := f(c \leftharpoonup c^*)
  =\langle c^*, c_{(1)} \rangle f(c_{(2)})
\end{equation*}
for $c^* \in C$, $c \in C$ and $f \in \Hom^C(C, M)$. The functor $\Hom^C(C, -)$ has been studied in relation to integral theory of Hopf algebras; see, {\it e.g.}, \cite{MR1323693,MR1786197,MR3125851}. In general, $\Hom^{C}(C, M)$ is not a rational left $C^*$-module even if $M$ is finite-dimensional (see Example~\ref{ex:counter-example-0}).

\begin{definition}
  The left exact Nakayama functor for $C$ is defined by
  \begin{equation*}
    \Nakl_C : \Mod^C \to \Mod^C,
    \quad M \mapsto \Hom^C(C, M)^{\rat}.
  \end{equation*}
\end{definition}

For $M \in \Mod^C$, we set $\Nakr_C(M) = C \otimes_{C^*} M$, where $C$ and $M$ are regarded as a right $C^*$-module and a left $C^*$-module, respectively. The vector space $\Nakr_C(M)$ is a right $C$-comodule by the well-defined coaction
\begin{equation*}
  C \otimes_{C^*} M \to (C \otimes_{C^*} M) \otimes_{\bfk} C,
  \quad c \otimes_{C^*} m \mapsto (c_{(1)} \otimes_{C^*} m) \otimes c_{(2)}.
\end{equation*}

\begin{definition}
  The right exact Nakayama functor for $C$ is defined by
  \begin{equation*}
    \Nakr_C : \Mod^C \to \Mod^C,
    \quad M \mapsto C \otimes_{C^*} M.
  \end{equation*}  
\end{definition}

As their names suggest, the functors $\Nakl_C$ and $\Nakr_C$ are left exact and right exact, respectively. This follows from the following lemma and the general fact that a functor between abelian categories is left (right) exact if it has a left (right) adjoint.

\begin{lemma}
  \label{lem:Nakayama-adj}
  $\Nakr_C$ is left adjoint to $\Nakl_C$.
\end{lemma}
\begin{proof}
  For $M, M' \in \Mod^C$, we have natural isomorphisms
  \begin{align*}
    \Hom^C(\Nakr_C(M), M')
    & = {}_{C^*}\Hom(C \otimes_{C^*} M, M') \\
    & \cong {}_{C^*}\Hom(M, {}_{C^*}\Hom(C, M')) \\
    & = {}_{C^*}\Hom(M, \Hom^C(C, M')) \\
    & \cong \Hom^C(M, \Hom^C(C, M')^{\rat}) \\
    & = \Hom^C(M, \Nakl_C(M')). \qedhere
  \end{align*}
\end{proof}

Although we will not use them in this paper, it would be worth to include a description of the unit and the counit of the adjunction $\Nakr_C \dashv \Nakl_C$. By the proof of Lemma~\ref{lem:Nakayama-adj}, the unit $u$ is given by
\begin{equation*}
  u_M : M \to \Nakl_C \Nakr_C(M) = \Hom^C(C, C \otimes_{C^*} M)^{\rat},
  \quad u_M(m)(c) = c \otimes_{C^*} m
\end{equation*}
for $c \in C$ and $m \in M \in \Mod^C$. The counit $e$ is given by
\begin{equation*}
  e_M : \Nakr_C \Nakl_C(M) = C \otimes_{C^*} \Hom(C, M)^{\rat} \to M,
  \quad e_M(c \otimes_{C^*} \xi) = \xi(c)
\end{equation*}
for $c \in C$, $M \in \Mod^C$ and $\xi \in \Hom^C(C, M)^{\rat}$.

Let $\qfcomod^C$ denote the full subcategory of $\Mod^C$ consisting of all quasi-finite right $C$-comodules (see Subsection~\ref{subsec:cohom-functor}). The restriction of $\Nakl_C$ to $\qfcomod^C$ is expressed by the coHom functor as follows:

\begin{lemma}
  \label{lem:Nakayama-l-coHom}
  For $M \in \qfcomod^C$, we have a natural isomorphism
  \begin{equation*}
    \Nakl_C(M) \cong \coHom^C(M, C)^{* \rat}.
  \end{equation*}
\end{lemma}

The precise meaning of the right hand side is as follows:
Let, in general, $C$ and $D$ be coalgebras, and let $P$ be a $C$-$D$-bicomodule. Then there is a contravariant functor
$\coHom^D(-, P) : \qfcomod^D \to {}^C\Mod$.
If $M$ is a left $C$-comodule, then $M^*$ is a left $C^*$-module as the dual space of the right $C^*$-module $M$. This lemma actually states that the restriction of $\Nakl_C$ to $\qfcomod^C$ is isomorphic to the composition
\begin{equation*}
  \qfcomod^C \xrightarrow{\quad \coHom^C(-, C) \quad}
  {}^C\Mod \xrightarrow{\quad (-)^* \quad}
  {}_{C^*}\Mod \xrightarrow{\quad (-)^{\rat} \quad} \Mod^C.
\end{equation*}

\begin{proof}
  By the definition of the coHom functor, we have a natural isomorphism
  \begin{equation}
    \label{eq:coHom-dual}
    \coHom^C(X, Y)^*
    = \Hom_{\bfk}(\coHom^C(X, Y), \bfk)
    \cong \Hom^C(Y, X)
  \end{equation}
  for $X \in \qfcomod^C$ and $Y \in \Mod^C$. It is easy to check that \eqref{eq:coHom-dual} is left $C^*$-linear when $Y$ is a $C$-bicomodule. Thus, for $M \in \qfcomod^C$, we have an isomorphism
  \begin{equation*}
    \Nakl_C(M) = \Hom^C(C, M)^{\rat} \cong \coHom^C(M, C)^{*\rat}
  \end{equation*}
  of right $C$-comodules. The proof is done.
\end{proof}

\begin{remark}
  \label{rem:Chin-Simson}
  The Nakayama functor of Chin and Simson \cite[Equation 1.12]{MR2684139}, which we denote by $\Nak_{\mathrm{CS}}$ here, is defined on a particular class of quasi-finite comodules. By the definition of the domain of $\Nak_{\mathrm{CS}}$ and \cite[Theorem 1.8 (a)]{MR2684139}, we have
  $\dim_{\bfk} \coHom^C(M, C) < \infty$
  and
  $\Nak_{\mathrm{CS}}(M) = \coHom^C(M, C)^{*}$
whenever $M$ belongs to the domain of $\Nak_{\mathrm{CS}}$. Thus, by Lemma~\ref{lem:Nakayama-l-coHom}, $\Nak_{\mathrm{CS}}$ is the restriction of our Nakayama functor $\Nakl_C$.
\end{remark}

\subsection{Nakayama functor as a (co)end}

Let $C$ be a coalgebra. Here we show that the Nakayama functors have the following universal property:

\begin{theorem}
  \label{thm:Nakayama-(co)end}
  For every $M \in \Mod^C$, $\Nakl_C(M)$ is an end of the functor
  \begin{equation}
    \label{eq:Nakayama-l-end}
    (\fdmod^C)^{\op} \times \fdmod^C \to \Mod^C,
    \quad (X^{\op}, Y) \mapsto \Hom^C(X, M) \otimes_{\bfk} Y
  \end{equation}
  and $\Nakr_C(M)$ is a coend of the functor
  \begin{equation}
    \label{eq:Nakayama-r-coend}
    (\fdmod^C)^{\op} \times \fdmod^C \to \Mod^C,
    \quad (X^{\op}, Y) \mapsto \coHom^C(X, M) \otimes_{\bfk} Y.
  \end{equation}
\end{theorem}

The claim about $\Nakl_C(M)$ may be expressed as the equation
\begin{equation*}
  \Nakl_C(M) = \int_{X \in \fdmod^C} \Hom^C(X, M) \otimes_{\bfk} X,
\end{equation*}
however, we shall care about the category in which the integrand lives when we use this sort of notation. Indeed, let $\mathcal{U} : \Mod^C \to \Vect$ be the forgetful functor. Confusingly, the vector space $\mathcal{U}(\Nakl_C(M))$ may not be isomorphic to an end of
\begin{equation}
  \label{eq:Nakayama-l-coend-2}
  (\fdmod^C)^{\op} \times \fdmod^C \to \Vect,
  \quad (X^{\op}, Y) \mapsto \Hom^C(X, M) \otimes_{\bfk} Y,
\end{equation}
since $\mathcal{U}$ does not preserve limits in general \cite[Proposition 31]{MR414567}.
The difference between the end of \eqref{eq:Nakayama-l-end} and that of \eqref{eq:Nakayama-l-coend-2} will be discussed after the proof of this theorem; see Remark~\ref{rem:Nakayama-(co)end}.

The situation is a bit easier for $\Nakr_C(M)$ compared to the case of $\Nakl_C(M)$. 
The claim about $\Nakr_C(M)$ is expressed as the equation
\begin{equation*}
  \Nakr_C(M) = \int^{X \in \fdmod^C} \coHom^C(X, M) \otimes_{\bfk} X,
\end{equation*}
and the integrand may also be viewed as a mere vector space.
Namely, since $\mathcal{U}$ preserves colimits \cite[Corollary 3]{MR414567}, Theorem~\ref{thm:Nakayama-(co)end} implies that the vector space $\mathcal{U}(\Nakr_C(M))$ is a coend of the functor
\begin{equation*}
  (\fdmod^C)^{\op} \times \fdmod^C \to \Vect,
  \quad (X^{\op}, Y) \mapsto \coHom^C(X, M) \otimes_{\bfk} Y.
\end{equation*}

\begin{proof}[Proof of Theorem~\ref{thm:Nakayama-(co)end}]
  For $P \in {}^C\Mod^C$ and $M, N \in \Mod^C$, we have
  \begin{equation}
    \label{eq:Nakayama-(co)end-proof-1}
    \begin{aligned}
      {}_{C^*}\Hom(P \otimes_{C^*} N, M)
      & \cong {}_{C^*}\Hom_{C^*}(P, \Hom_{\bfk}(N, M)) \\
      & \cong \Hom^{C^e} \! (P, \mathrm{Rat}_{C^e}(\Hom_{\bfk}(N, M))),
    \end{aligned}
  \end{equation}
  where the first isomorphism is the tensor-Hom adjunction and $\mathrm{Rat}_{C^e}$ is the rational part functor \eqref{eq:rational-functor} for the enveloping coalgebra $C^e = C \otimes_{\bfk} C^{\cop}$.

  Now we fix a right $C$-comodule $M$ and show that $\Nakl_C(M)$ is an end as stated. For $X, Y \in \fdmod^C$, there is the following natural isomorphism of left $C^*$-modules:
  \begin{equation*}
    \Hom^C(X, M) \otimes_{\bfk} Y
    \xrightarrow{\ \cong \ } \Hom^C(Y^* \otimes_{\bfk} X, M),
    \quad \xi \otimes y
    \mapsto \big[ y^* \otimes x \mapsto \xi(x) \langle y^*, y \rangle \big].
  \end{equation*}
  Since the left $C^*$-module $\Hom^C(X, M) \otimes_{\bfk} Y$ is rational, we have
  \begin{equation}
    \label{eq:Nakayama-(co)end-proof-2}
    \Hom^C(Y^* \otimes_{\bfk} X, M)^{\rat} \cong (\Hom^C(X, M) \otimes_{\bfk} Y)^{\rat} = \Hom^C(X, M) \otimes_{\bfk} Y
  \end{equation}
  as right $C$-comodules.
  For $P \in {}^C\Mod^C$ and $N \in \Mod^C$, we have
  \begin{equation}
    \label{eq:Nakayama-(co)end-proof-3}
    \begin{aligned}
      \Hom^C(N, \Hom^C(P, M)^{\rat})
      & \cong \Hom^C(P \otimes_{C^*} N, M) \\
      & \cong \Hom^{C^e}(P, \mathrm{Rat}_{C^e}(\Hom_{\bfk}(N, M))),
    \end{aligned}
  \end{equation}
  where the first isomorphism is obtained by the same way as the proof of Lemma   \ref{lem:Nakayama-adj} and the second one follows from \eqref{eq:Nakayama-(co)end-proof-1}. Hence the contravariant functor $\Psi_M: {}^C\Mod^C \to \Mod^C$, $P \mapsto \Hom^C(P, M)^{\rat}$ sends a colimit to a limit. Thus,
  \begin{equation*}
    \Nakl_C(M)
    = \Psi_M(C)
    \mathop{\cong}^{\eqref{eq:Tannaka}}
    \int_{X \in \fdmod^C} \Hom^C(X^* \otimes_{\bfk} X, M)^{\rat}
    \mathop{\cong}^{\eqref{eq:Nakayama-(co)end-proof-2}}
    \int_{X \in \fdmod^C} \Hom^C(X, M) \otimes_{\bfk} X.
  \end{equation*}

  The statement for $\Nakr_C(M)$ is proved as follows:
  By \eqref{eq:Nakayama-(co)end-proof-1}, the functor $\Phi_M: {}^C\Mod^C \to \Mod^C$, $P \mapsto P \otimes_{C^*} M$ preserves colimits. Thus,
  \begin{equation*}
    \Nakr_C(M)
    = \Phi_M(C)
    \mathop{\cong}^{\eqref{eq:Tannaka}}
    \int^{X \in \fdmod^C} (X^* \otimes_{\bfk} X) \otimes_{C^*} M
    \cong
    \int^{X \in \fdmod^C} \coHom^C(X, M) \otimes_{\bfk} X,
  \end{equation*}
  where the second isomorphism is given in Lemma~\ref{lem:coHom-finite-2}.
  The proof is done.
\end{proof}

\begin{remark}
  \label{rem:Nakayama-(co)end}
  For $M \in \Mod^C$, we set $\widetilde{\Nak}(M) = \Hom^C(C, M)$.
  By definition, $\widetilde{\Nak}(M)$ is a left $C^*$-module whose rational part is $\Nakl_C(M)$.
  There is a natural isomorphism
  \begin{equation*}
    {}_{C^*}\Hom(N, \Hom^C(P, M))
    \cong \Hom^{C^e}(P, \Rat_{C^e}(\Hom_{\bfk}(N, M)))    
  \end{equation*}
  for $P \in {}^C\Mod^C$ and $N \in {}_{C^*}\Mod$ obtained by the same way as \eqref{eq:Nakayama-(co)end-proof-3}.
  Thus the contravariant functor $\Hom^C(-,M) : {}^C\Mod^C \to {}_{C^*}\Mod$ turns a coend into an end. In particular, the left $C^*$-module $\widetilde{\Nak}(M)$ is an end of the functor
  \begin{equation*}
    (\fdmod^C)^{\op} \times \fdmod^C \to {}_{C^*}\Mod,
    \quad (X^{\op}, Y) \mapsto \Hom^C(X, M) \otimes_{\bfk} Y.
  \end{equation*}
  Since the forgetful functor ${}_{C^*}\Mod \to \Vect$ preserves limits, the vector space $\widetilde{\Nak}(M)$ is in fact an end the functor \eqref{eq:Nakayama-l-end}.

  In summary, the right $C$-comodule $\Nakl_C(M)$ and the vector space $\widetilde{\Nak}(M)$ are an end of \eqref{eq:Nakayama-l-end} and \eqref{eq:Nakayama-l-coend-2}, respectively.
  Now we suppose that $C^{*\rat} = 0$ ({\it e.g.}, take $C$ to be the coalgebra $\bfk[X]$ discussed in Example~\ref{ex:counter-example-3}). Then $\Nakl_C(C) \cong C^{*\rat} = 0$ while $\widetilde{\Nak}(C) \cong C^* \ne 0$. Thus an end of \eqref{eq:Nakayama-l-end} and that of \eqref{eq:Nakayama-l-coend-2} are not isomorphic as vector spaces.
\end{remark}

Two coalgebras $C$ and $D$ are said to be {\em Morita-Takeuchi equivalent} if $\Mod^C$ and $\Mod^D$ are equivalent as $\bfk$-linear categories.
As an application of the universal property of the Nakayama functor, we show that the left and the right exact Nakayama functor are Morita-Takeuchi invariant in the following sense:

\begin{lemma}
  \label{lem:Nakayama-Mori-Take-equiv}
  Let $C$ and $D$ be coalgebras, and let $\Phi : \Mod^C \to \Mod^D$ be an equivalence of $\bfk$-linear categories. Then there are isomorphisms of functors
  \begin{equation*}
    \Phi \circ \Nakl_C \cong \Nakl_D \circ \Phi,
    \quad
    \Phi \circ \Nakr_C \cong \Nakr_D \circ \Phi.
  \end{equation*}
\end{lemma}
\begin{proof}
  By the fundamental theorem for comodules, a comodule over a coalgebra is finite-dimensional if and only if it is of finite length.
  Since an equivalence between abelian categories preserves lengths, we see that the equivalence $\Phi$ induces an equivalence between $\fdmod^C$ and $\fdmod^D$.
  The proof of this lemma is now obvious by the universal property of the Nakayama functor given in Theorem~\ref{thm:Nakayama-(co)end}.
\end{proof}

\subsection{Semiperfect coalgebras}

Let $C$ be a coalgebra. An {\em idempotent} for $C$ is an idempotent of the algebra $C^*$. Given an idempotent $e$ for $C$ and a right $C$-comodule $M$, we write $e M := e \rightharpoonup M$ (where $M$ is regarded as a left $C^*$-module by $\rightharpoonup$). A similar notation will be used for left comodules. It is straightforward to show:

\begin{lemma}
  For $M \in \Mod^C$, $W \in \Vect$ and an idempotent $e$ for $C$, the map
  \begin{equation}
    \label{eq:Hom-idempotent}
    \Hom^C(M, W \otimes_{\bfk} C e) \to \Hom_{\bfk}(e M, W),
    \quad f \mapsto (\id_W \otimes \varepsilon|_{Ce}) \circ (f|_{e M})
  \end{equation}
  is an isomorphism of vector spaces, where $\varepsilon$ is the counit of $C$.
\end{lemma}

If $e$ is an idempotent for $C$, then $C e$ is an injective right $C$-comodule as a direct summand of the right $C$-comodule $C$. Moreover, every indecomposable injective right $C$-comodule is isomorphic to $C e$ for some primitive idempotent $e$ for $C$ \cite[Proposition 1.13]{MR1904645}.
Let $\sfInj^C$ be the full subcategory of $\Mod^C$ consisting of all injective right $C$-comodules whose socle is of finite length.
Every object of $\sfInj^C$ is a finite direct sum of indecomposable injective right $C$-comodules, and thus it is isomorphic to $\bigoplus_i e_i C$ for some finite number of primitive idempotents $e_i$ for $C$. We define ${}^C\sfInj \subset {}^C\Mod$ analogously.

\begin{lemma}
  The contravariant functor
  \begin{equation}
    \label{eq:coHom-restricted-to-inj}
    \sfInj^C \to {}^C\sfInj,
    \quad E \mapsto \coHom^C(E, C)
  \end{equation}
  is a well-defined anti-equivalence.
\end{lemma}
\begin{proof}
  Let $e$ be an idempotent for $C$.
  Then, by~\eqref{eq:Hom-idempotent}, we have an isomorphism
  \begin{equation}
    \label{eq:coHom-idempotent}
    \coHom^C(C e, C) \cong e C
  \end{equation}
  of left $C$-comodules \cite[Corollary 1.6]{MR1904645}. Thus the functor \eqref{eq:coHom-restricted-to-inj} is essentially surjective. For two idempotents $e$ and $f$ for $C$, we have $\Hom^C(C e, C f) \cong (f C e)^* \cong e C^* f$ by \eqref{eq:Hom-idempotent}. By applying this formula to $C^{\cop}$, we also have an isomorphism
  ${}^C\Hom(f C, e C)
  \cong f *^{\cop} (C^{\cop})^* *^{\cop} e
  = e C^* f$, where $*^{\cop}$ is the convolution product for $(C^{\cop})^*$. If we identify
  $\Hom^C(C e, C f) = e C^* f = {}^C\Hom(f C, e C)$
  by these isomorphisms, the map $\Hom^C(C e, C f) \to {}^C\Hom(f C, e C)$ induced by the functor \eqref{eq:coHom-restricted-to-inj} is the identity. Thus \eqref{eq:coHom-restricted-to-inj} is fully faithful. The proof is done.
\end{proof}

Since the duality functor $(-)^*$ is an anti-equivalence between ${}^C\fdmod$ and $\fdmod^C$, it restricts to an anti-equivalence between ${}^C\fdInj := \Inj({}^C\fdmod)$ and $\fdPro^C := \Proj(\fdmod^C)$. The discussion so far is summarized into the diagram of Figure~\ref{fig:the-discussion-so-far}, which commutes up to isomorphisms.

\begin{figure}
    \begin{equation*}
  \begin{tikzcd}[column sep = 80pt, row sep = 24pt]
    \Mod^C \arrow[rr, "{\Nakl_C}"] & & \Mod^C \\
    \qfcomod^C
    \arrow[u, hookrightarrow]
    \arrow[r, "{\coHom^C(-,C)}"]
    & {}^C\Mod \arrow[r, "{(-)^{*\rat}}"]
    & \Mod^C \arrow[u, equal] \\
    \sfInj^C
    \arrow[u, hookrightarrow]
    \arrow[r, "C e \mapsto e C", "\text{anti-equivalence}"']
    & {}^C\sfInj
    \arrow[u, hookrightarrow]
    \arrow[r, "{(-)^{*\rat}}"]
    & \Mod^C \arrow[u, equal] \\
    & {}^C\fdInj
    \arrow[u, hookrightarrow]
    \arrow[r, "{(-)^*}", "\text{anti-equivalence}"']
    & \fdPro^C
    \arrow[u, hookrightarrow]
  \end{tikzcd}
\end{equation*}
    \caption{Functors between full subcategories}
    \label{fig:the-discussion-so-far}
\end{figure}

It is known that the Nakayama functor $A^* \otimes_A (-)$ for a finite-dimensional algebra $A$ restricts to an equivalence from $\Proj({}_A\Mod_{\fd})$ to $\Inj({}_A \Mod_{\fd})$. For coalgebras, a similar result does not hold in general as Figure~\ref{fig:the-discussion-so-far} suggests. In fact, there is an example of a coalgebra $C$ and $E \in \fdInj^C$ such that $\Nakl_C(E)$ is not projective (see Example~\ref{ex:counter-example-0}).

As we have discussed in Remark~\ref{rem:Chin-Simson}, the Nakayama functor of Chin and Simson \cite{MR2684139} is a restriction of the functor $\Nakl_C$. It has been noted in \cite[p.\ 2228]{MR2684139} without a proof that $\Nakl_C$ restricts to an equivalence from $\sfInj^C$ to $\fdPro^C$ provided that $C$ is right semiperfect. We give a proof of this result for the sake of completeness, as well as explaining that its quasi-inverse is the restriction of $\Nakr_C$.

\begin{lemma}
  \label{lem:Nakayama-semiperfect-1}
  If $C$ is right semiperfect, then the functors $\Nakl_C : \sfInj^C \to \fdPro^C$ and $\Nakr_C : \fdPro^C \to \sfInj^C$ are mutually quasi-inverse to each other.
\end{lemma}
\begin{proof}
  Suppose that $C$ is right semiperfect. Then every simple left $C$-comodule has a finite-dimensional injective hull \cite[Theorem 3.2.3]{MR1786197}, and thus the equation ${}^C\sfInj = {}^C\fdInj$ holds. By the diagram of Figure \ref{fig:the-discussion-so-far}, we see that the functor $\Nakl_C : \sfInj^C \to \fdPro^C$ is an equivalence as a composition of two anti-equivalences.

  To show that $\Nakr_C : \fdPro^C \to \sfInj^C$ is a quasi-inverse of $\Nakl_C : \sfInj^C \to \fdPro^C$, we first note that every object of $\fdPro^C$ is isomorphic to the direct sum of right $C$-comodules of the form $(e C)^* \cong C^* e$ for some primitive idempotent $e$ for $C$. For such an idempotent $e$, there is an isomorphism
  \begin{equation*}
    \Nakr_C(C^* e) = C \otimes_{C^*} C^* e \cong C e
  \end{equation*}
  of right $C$-comodules.
  Thus $\Nakr_C(P)$ belongs to $\sfInj^C$ whenever $P \in \fdPro^C$.
  This implies that the adjunction $\Nakr_C \dashv \Nakl_C$ restricts to an adjunction between $\sfInj^C$ and $\fdPro^C$. Since $\Nakl_C : \fdPro^C \to \sfInj^C$ is an equivalence, its left adjoint $\Nakr_C : \fdPro^C \to \sfInj^C$ is a quasi-inverse of it. The proof is done.
\end{proof}

The Nakayama functor for a finite-dimensional algebra sends the injective hull of a simple module to the projective cover of that simple module. We now prove an analogous result for right semiperfect coalgebras as follows:
Suppose that $C$ is a right semiperfect coalgebra.
Given $V \in \fdmod^C$, we denote by $E(V)$ and $P(V)$ the injective hull and the projective cover of $V$, respectively.

\begin{lemma}
  \label{lem:Nakayama-semiperfect-2}
  There are isomorphisms
  \begin{equation*}
    \Nakl_C(E(S)) \cong P(S)
    \quad \text{and} \quad
    \Nakr_C(P(S)) \cong E(S)
  \end{equation*}
  for each simple right $C$-comodule $S$.
\end{lemma}
\begin{proof}
  Let $S$ be a simple right $C$-comodule, and let $e$ be a primitive idempotent for $C$ such that $E(S) \cong C e$ as right $C$-comodules.
  Then we have isomorphisms
  \begin{equation}
    \label{eq:Nakayama-semiperfect-2-proof-1}
    {}^C\Hom(S^*, e C)
    \mathop{\cong}^{\eqref{eq:Hom-idempotent}} (S^* e)^*
    \cong  (e S)^{**}
    \mathop{\cong}^{\eqref{eq:Hom-idempotent}} \Hom^C(S ,C e)^* \ne 0.
  \end{equation}
  Since $e$ is primitive, $e C$ is an injective hull of a simple left $C$-comodule.
  By \eqref{eq:Nakayama-semiperfect-2-proof-1}, $e C$ is in fact an injective hull of $S^*$.
  Thus $\Nakl_C(E(S)) \cong (e C)^*$ is a projective cover of $S^{**} \cong S$.
  Hence the former isomorphism in the claim is obtained.
  The latter one is obtained by Lemma~\ref{lem:Nakayama-semiperfect-1} and the former one.
\end{proof}

For our applications of the Nakayama functor, it is important to know when $\fdmod^C$ is preserved by $\Nakl_C$ and $\Nakr_C$.
The following Lemmas~\ref{lem:Nakayama-semiperfect-4} and \ref{lem:Nakayama-semiperfect-5} give sufficient conditions:

\begin{lemma}
  \label{lem:Nakayama-semiperfect-4}
  Suppose that $C$ is right semiperfect. Then we have
  \begin{equation*}
    \Nakl_C(M) = \Hom^C(C, M)
    \quad \text{and} \quad
    \dim_{\bfk} \Nakl_C(M) < \infty
  \end{equation*}
  for all $M \in \fdmod^C$. In particular, $\Nakl_C$ preserves $\fdmod^C$.
\end{lemma}
\begin{proof}
  We consider two functors $F, G: \Mod^C \to \Vect$ defined by
  \begin{equation*}
    F(M) = \Nakl_C(M), \quad G(M) = \Hom^C(C, M) \quad (M \in \Mod^C).
  \end{equation*}
  Let $S$ be a simple right $C$-comodule, and let $e$ be a primitive idempotent for $C$ such that $C e$ is an injective hull of $S$.
  Then $e C$ is finite-dimensional by the right semiperfectness of $C$ and, as we have seen in the proof of Lemma~\ref{lem:Nakayama-semiperfect-2}, $(e C)^*$ is a projective cover of $S$.
  There are isomorphisms
  \begin{equation*}
    G(C e)
    = \Hom^C(C, C e)
    \mathop{\cong}^{\eqref{eq:coHom-dual}}
    \coHom^C(C e, C)^*
    \mathop{\cong}^{\eqref{eq:coHom-idempotent}} (e C)^*
    \cong F(C e)
  \end{equation*}
  of vector spaces. Thus, for all objects $M \in \sfInj^C$, we have
  \begin{equation}
    \label{eq:Nakayama-semiperfect-4-proof-1}
    \dim_{\bfk} F(M) = \dim_{\bfk} G(M) < \infty.
  \end{equation}
  Since $F$ and $G$ are left exact, \eqref{eq:Nakayama-semiperfect-4-proof-1} actually holds for all right $C$-comodule $M$ fitting into an exact sequence of the form $0 \to M \to E \to E'$ in $\Mod^C$ for some $E, E' \in \sfInj^C$. In particular, \eqref{eq:Nakayama-semiperfect-4-proof-1} holds for all $M \in \fdmod^C$.
  Since $G(M) \subset F(M)$, we have $F(M) = G(M)$ for all $M \in \fdmod^C$. The proof is done.
\end{proof}

\begin{lemma}
  \label{lem:Nakayama-semiperfect-5}
  If $C$ is left semiperfect, then we have $\Nakr_C(M) \in \fdmod^C$ for all $M \in \fdmod^C$.
\end{lemma}
\begin{proof}
  By the assumption, $C^{\cop}$ is {\em right} semiperfect. For $M \in \fdmod^C$, we have
  \begin{gather*}
    \Nakr_C(M)^*
    = \Hom_{\bfk}(C \otimes_{C^*} M, \bfk)
    \cong\Hom_{C^*}(C, \Hom_{\bfk}(M, \bfk)) \\
    = {}^C\Hom(C, M^*)
    = \Nakl_{C^{\cop}}(M^*),
  \end{gather*}
  where the last equality follows from Lemma~\ref{lem:Nakayama-semiperfect-4} applied to $C^{\cop}$.
  By that lemma, we also have that $\Nakl_{C^{\cop}}(M^*)$ is finite-dimensional.
  Thus $\Nakr_C(M) \in \fdmod^C$. The proof is done.
\end{proof}

By Lemmas~\ref{lem:Nakayama-semiperfect-1}, \ref{lem:Nakayama-semiperfect-4} and \ref{lem:Nakayama-semiperfect-5}, we now have the following theorem:

\begin{theorem}
  \label{thm:Nakayama-semiperfect}
  Suppose that $C$ is a semiperfect coalgebra.
  Then the following hold:
  \begin{enumerate}
  \item [\rm (a)] The adjunction $\Nakr_C \dashv \Nakl_C$ restricts to the adjunction
  \begin{equation*}
    \begin{tikzcd}[column sep = 120pt]
      \fdmod^C   \arrow[r, bend right=10, "{\Nakl_C}"']
      \arrow[r, phantom, "{\perp}"]
      & \fdmod^C \arrow[l, bend right=10, "{\Nakr_C}"']
    \end{tikzcd}
  \end{equation*}
  \item [\rm (b)] The functors $\Nakl_C : \fdInj^C \to \fdPro^C$ and $\Nakr_C : \fdPro^C \to \fdInj^C$ are mutually quasi-inverse to each other.
\end{enumerate}
\end{theorem}
\begin{proof}
  (a) This follows from Lemmas~\ref{lem:Nakayama-semiperfect-4} and \ref{lem:Nakayama-semiperfect-5}.
  (b) By the semiperfectness of $C$, we have not only ${}^C\sfInj = {}^C\fdInj$ but also $\sfInj^{C} = \fdInj^{C}$. Thus the claim follows from Lemma~\ref{lem:Nakayama-semiperfect-1}.
\end{proof}

It seems to be difficult to characterize when $\Nakl_C$ or $\Nakr_C$ preserves $\fdmod^C$.
Indeed, there is a right semiperfect coalgebra $C$ such that $\Nakl_C$ does not preserve $\fdmod^C$ (Example~\ref{ex:counter-example-1}). Similarly, there is a left semiperfect coalgebra $C$ such that $\Nakr_C$ does not preserve $\fdmod^C$ (Example~\ref{ex:counter-example-2}). On the other hand, there is a coalgebra $C$, which is neither left nor right semiperfect, such that both $\Nakl_C$ and $\Nakr_C$ preserve $\fdmod^C$ (Example~\ref{ex:counter-example-3}).

\subsection{Co-Frobenius coalgebras}

Here we consider the case where the coalgebra is {\em co-Frobenius} (see below).
Let $V$ be a vector space, and let $\beta : V \times V \to \bfk$ be a bilinear map. We say that $\beta$ is {\em left} ({\it resp}. {\em right}) {\em non-degenerate} if the map
\begin{equation*}
  V \to V^*, \quad v \mapsto \beta(v, -)
  \quad (\text{{\em resp}. $v \mapsto \beta(-, v)$})
\end{equation*}
is injective. If $\beta$ is both left and right non-degenerate, we say that $\beta$ is non-degenerate. The definition of a left and a right co-Frobenius coalgebra is found in \cite[Section 3.3]{MR1786197}. A co-Frobenius coalgebra is a left and right co-Frobenius coalgebra. According to \cite[Theorem 2.8]{MR2253657}, we adopt the following equivalent definition:

\begin{definition}
  \label{def:co-Frob}
  A {\em co-Frobenius coalgebra} is a coalgebra $C$ equipped with a non-degenerate $C^*$-balanced bilinear form $\beta : C \times C \to \bfk$, called the {\em Frobenius pairing}.
\end{definition}

Let $C$ be a co-Frobenius coalgebra with Frobenius pairing $\beta$.
It is known that a co-Frobenius coalgebra is semiperfect \cite[Section 3.3]{MR1786197}.
According to \cite[Section 1]{MR2078404}, the linear maps $\beta(c, -)$ and $\beta(-, c)$ belong to $C^{*\rat}$ ($= ({}_{C^*}C^*)^{\rat} = (C^*_{C^*})^{\rat}$) for all $c \in C$. Furthermore, two linear maps
\begin{equation*}
  \beta^{\L} : C \to C^{*\rat},
  \quad c \mapsto \beta(c, -)
  \quad \text{and} \quad
  \beta^{\R} : C \to C^{*\rat},
  \quad c \mapsto \beta(-,c)
\end{equation*}
are bijective. The following definition is due to \cite[Section 6]{MR2078404}.

\begin{definition}
  We call $\nu_C := (\beta^{\L})^{-1} \circ \beta^{\R}$ the {\em Nakayama automorphism} of $C$ with respect to the Frobenius pairing $\beta$.
\end{definition}

In other words, $\nu_C : C \to C$ is the unique linear map such that
\begin{equation}
  \label{eq:Nakayama-auto-def}
  \beta(y, x) = \beta(\nu_C(x), y)
  \quad (x, y \in C).
\end{equation}

\begin{lemma}
  \label{lem:Nakayama-auto-lemma-1}
  $\nu_C$ is a coalgebra automorphism of $C$.
\end{lemma}
\begin{proof}
  We consider the linear automorphism $\mu := \beta^{\L} \circ (\beta^{\R})^{-1}$ on $C^{*\rat}$.
  According to \cite[Proposition 2.1]{MR2078404}, the following equation holds:
  \begin{equation*}
    (\beta^{\L})^{-1}(\beta^{\L}(x) * \beta^{\L}(y))
    = (\beta^{\R})^{-1}(\beta^{\R}(x) * \beta^{\R}(y))
    \quad (x, y \in C).
  \end{equation*}
  Thus, for $\xi, \zeta \in C^{*\rat}$ with $\xi = \beta^{\R}(x)$ and $\zeta = \beta^{\R}(y)$, we have
  \begin{align*}
    \mu(\xi * \zeta)
    & = \beta^{\L}((\beta^{\R})^{-1}(\beta^{\R}(x) * \beta^{\R}(y))) \\
    & = \beta^{\L}((\beta^{\L})^{-1}(\beta^{\L}(x) * \beta^{\L}(y)))
    = \mu(\xi) * \mu(\zeta).
  \end{align*}
  By the defining equation \eqref{eq:Nakayama-auto-def} of $\nu := \nu_C$, we also have
  \begin{equation*}
    \langle \mu(\xi), c \rangle = \langle \beta^{\L}(x), c \rangle
    = \beta(x, c)
    = \beta(\nu(c), x)
    = \langle \beta^{\R}(x), \nu(c) \rangle
    = \langle \xi, \nu(c) \rangle
  \end{equation*}
  for all $c \in C$. By the discussion so far, we obtain
  \begin{equation*}
    (\xi \otimes \zeta) \circ \Delta \circ \nu
    = \mu(\xi * \zeta)
    = \mu(\xi) * \mu(\zeta)
    = (\xi \otimes \zeta) \circ (\nu \otimes \nu) \circ \Delta
  \end{equation*}
  for all $\xi, \zeta \in C^{*\rat}$.
  Since $C^{*\rat}$ is dense in $C^*$, this implies $\Delta \circ \nu = (\nu \otimes \nu) \circ \Delta$. Namely, $\nu$ preserves the comultiplication of $C$. The bijectivity of $\nu$ ensures that $\nu$ also preserves the counit of $C$. In conclusion, $\nu$ is a coalgebra automorphism of $C$.
\end{proof}

\begin{lemma}
  \label{lem:Nakayama-auto-lemma-2}
  $\nu_C$ is a unique linear endomorphism on $C$ satisfying
  \begin{equation}
    \label{eq:Nakayama-auto-lemma-2}
    \beta(x_{(1)}, y) x_{(2)}=  \nu_C(y_{(1)}) \beta(x, y_{(2)})
    \quad (x, y \in C).
  \end{equation}
\end{lemma}
\begin{proof}
  Write $\nu = \nu_C$.
  For $x, y \in C$ and $c^* \in C^*$, we compute
  \begin{gather*}
    \beta(x_{(1)}, y) \langle c^*, x_{(2)}\rangle
    = \beta(c^* \rightharpoonup x, y)
    \mathop{=}^{\eqref{eq:Nakayama-auto-def}} \beta(\nu(y), c^* \rightharpoonup x) \\
    = \beta(\nu(y) \leftharpoonup c^*, x)
    \mathop{=}^{\eqref{eq:Nakayama-auto-def}}
    \beta(x, \nu^{-1}(\nu(y) \leftharpoonup c^*))
    = \langle c^*, \nu(y_{(1)}) \rangle \beta(x, y_{(2)}),
  \end{gather*}
  where the third equality follows from that $\beta$ is $C^*$-balanced and the last one from that $\nu$ is a coalgebra automorphism. Thus $\nu$ satisfies \eqref{eq:Nakayama-auto-lemma-2}.
 
  Now we prove the uniqueness. Let $\eta: C \to C$ be a linear map satisfying
  \begin{equation}
    \label{eq:Nakayama-auto-lemma-3}
    \beta(x_{(1)}, y) x_{(2)}=  \eta(y_{(1)}) \beta(x, y_{(2)})
    \quad (x, y \in C).
  \end{equation}
  We fix $y \in C$ and choose a finite-dimensional subcoalgebra $D \subset C$ such that $y \in C$. Since $C^{*\rat}$ is dense in $C^{*}$, there is an element $x \in C$ such that $\beta(x, d) = \varepsilon(d)$ for all $d \in D$. Thus,
  \begin{equation*}
    \eta(y) = \eta(y_{(1)}) \beta(x, y_{(2)})
    \mathop{=}^{\eqref{eq:Nakayama-auto-lemma-3}}
    \beta(x_{(1)}, y) x_{(2)}
    \mathop{=}^{\eqref{eq:Nakayama-auto-lemma-2}}
    \nu(y_{(1)}) \beta(x, y_{(2)}) = \nu(y).
  \end{equation*}
  Hence $\eta = \nu$. The proof is done.
\end{proof}

Given a coalgebra map $\phi : C \to D$ between coalgebras, we denote by $(-)^{(\phi)} : \Mod^C \to \Mod^D$ the functor induced by $\phi$.
For co-Frobenius coalgebras, Nakayama functors take the following simple form:

\begin{lemma}
  \label{lem:Nakayama-co-Fro}
  Let $C$ be a co-Frobenius coalgebra. Then there are natural isomorphisms
  \begin{equation*}
    \Nakl_C(M) \cong M^{(\nu^{-1})},
    \quad
    \Nakr_C(M) \cong M^{(\nu)}
    \quad (M \in \Mod^C),
  \end{equation*}
  where $\nu = \nu_C$ is the Nakayama automorphism of $C$.
\end{lemma}
\begin{proof}
  Let $M \in \Mod^C$.
  By Lemma \ref{lem:C-star-rat-M-rat}, there is an isomorphism
  \begin{equation}
    \label{eq:Nakayama-co-Fro-proof-1}
    C^{*\rat} \otimes_{C^*} M \to M,
    \quad c^* \otimes_{C^*} m \mapsto m_{(0)} \langle c^*, m_{(1)} \rangle
  \end{equation}
  of left $C^*$-modules. Let $\beta$ be the Frobenius pairing of $C$.
  We note that $\beta^{\L}$ is right $C^*$-linear because of the $C^*$-balancedness of $\beta$.
  Composing \eqref{eq:Nakayama-co-Fro-proof-1}
 and $\beta^{\L} \otimes_{C^*} \id_M$, we have an isomorphism
  \begin{equation*}
    \beta^{\natural}_M : C \otimes_{C^*} M \to M,
    \quad c \otimes_{C^*} m \mapsto m_{(0)} \beta(c, m_{(1)})
  \end{equation*}
  of vector spaces. Now, for $c \in C$, $c^* \in C^*$ and $m \in M$, we compute
  \begin{gather*}
    \beta^{\natural}_M(c^* \rightharpoonup c \otimes_{C^*} m)
    = m_{(0)} \beta(c_{(1)}, m_{(1)}) \langle c^*, c_{(2)} \rangle \\
    \mathop{=}^{\eqref{eq:Nakayama-auto-lemma-2}}
    m_{(0)} \langle c^*, \nu(m_{(1)}) \rangle \beta(c, m_{(2)})
    = c^* \nu \rightharpoonup \beta^{\natural}_M(c \otimes_{C^*} m).
  \end{gather*}
  This means that $\beta^{\natural}_M$ is actually an isomorphism $\beta^{\natural}_M : \Nakr_C(M) \to M^{(\nu)}$ of right $C$-comodules.
  Since $\nu$ is a coalgebra automorphism, a right adjoint of $\Nakr_C$ is given by $(-)^{(\nu^{-1})}$.
  Thus, by the uniqueness of adjoints, $\Nakl_C$ is isomorphic to $(-)^{(\nu^{-1})}$.
  The proof is done.
\end{proof}

\subsection{Quasi-co-Frobenius coalgebras}

Here we discuss the following class of coalgebras:

\begin{definition}
  \label{def:QcF}
  A coalgebra $C$ is said to be {\em left} ({\em right}) {\em quasi-co-Frobenius} (QcF for short) if there is a cardinal $\kappa$ and an injective homomorphism $C \to (C^{*})^{\oplus \kappa}$ of left (right) $C^*$-modules \cite[Definition 3.3.1]{MR1786197}. A {\em QcF coalgebra} is a left and right QcF coalgebra.
\end{definition}

Several equivalent conditions for being QcF are given in \cite[Section 3.3]{MR1786197}.
It is known that a co-Frobenius coalgebra is QcF.
The converse does not hold:
There are examples of QcF coalgebras which are not co-Frobenius \cite[Example 4.7]{MR3125851}.

Lemma~\ref{lem:Nakayama-co-Fro} implies that $\Nakl_C$ and $\Nakr_C$ are equivalences if $C$ is a co-Frobenius coalgebra.
Theorem~\ref{thm:Nakayama-QcF} below shows that the same holds for QcF coalgebras and, in fact, QcF coalgebras are characterized by this property.

\begin{theorem}
  \label{thm:Nakayama-QcF}
  For a coalgebra $C$, the following are equivalent:
  \begin{enumerate}
  \item $C$ is QcF.
  \item $\Nakl_C : \Mod^C \to \Mod^C$ is an equivalence.
  \item $\Nakr_C : \Mod^C \to \Mod^C$ is an equivalence.
  \end{enumerate}
\end{theorem}
\begin{proof}
  We first remark that (2) and (3) are equivalent since $\Nakr_C$ is left adjoint to $\Nakl_C$.
  Next, we suppose that (1) holds.
  It is known that every QcF coalgebra is Morita-Takeuchi equivalent to a co-Frobenius coalgebra \cite{MR1998048}.
  Thus, by Lemma~\ref{lem:Nakayama-Mori-Take-equiv}, we may assume that $C$ is co-Frobenius. If we assume so, then (2) and (3) follow from Lemma~\ref{lem:Nakayama-co-Fro}.

  To complete the proof, we prove that (2) implies (1).
  Our proof relies on Iovanov's result \cite[Lemma 1.7]{MR3150709}, which states that a coalgebra $Q$ is QcF if and only if $({}_{Q^*}Q^{*\rat})^{\oplus \kappa} \cong Q^{\oplus \kappa}$ as right $Q$-comodules for some non-zero cardinal $\kappa$.
  Let $\{ S_{i} \}_{i \in I}$ be a complete set of representatives of the isomorphism classes of simple right $C$-comodules.
  Then we have $C \cong \bigoplus_{i \in I} E(S_{i})^{\oplus m(i)}$ for some integers $m(i) > 0$.
  Now we suppose that $\Nakl_C$ is an equivalence.
  Then there is a bijection $\pi : I \to I$ determined by $\Nakl_C(S_{i}) \cong S_{\pi(i)}$ for all $i \in I$ and we have
  \begin{equation*}
    \Nakl_C(C)
    \cong \bigoplus_{i \in I} \Nakl_C(E(S_{i}))^{\oplus m(i)}
    \cong \bigoplus_{i \in I} E(S_{\pi(i)})^{\oplus m(i)}
    \cong \bigoplus_{i \in I} E(S_{i})^{\oplus m(\pi^{-1}(i))}.
  \end{equation*}
  On the other hand, we have $\Nakl_C(C) = \Hom^C(C, C)^{\rat} \cong ({}_{C^*}C^{*})^{\rat}$. Thus,
  \begin{equation*}
    ({}_{C^*}C^{*\rat})^{\oplus \aleph_0}
    \cong \Nakl_C(C)^{\oplus \aleph_0}
    \cong \bigoplus_{i \in I} E(S_{i})^{\oplus \aleph_0}
    \cong C^{\oplus \aleph_0}.
  \end{equation*}
  Hence the proof is completed by \cite[Lemma 1.7]{MR3150709}.
\end{proof}

Suppose that $C$ is QcF. Then, by Theorem~\ref{thm:Nakayama-QcF}, $\Nakl_C$ induces a permutation on the set of isomorphism classes of simple right $C$-comodules. Following the representation theory of quasi-Frobenius algebras, we may call this permutation {\em the Nakayama permutation} for $C$. As in the case of algebras, one can use this permutation to describe the socle of a projective cover. More specifically,

\begin{theorem}
  \label{thm:Nakayama-permutation}
  Let $C$ be a QcF coalgebra. Then, for every simple right $C$-comodule $S$, there are isomorphisms
  $P(S) \cong E(\Nakl_C(S))$,
  $\socle P(S) \cong \Nakl_C(S)$,
  $E(S) \cong P(\Nakr_C(S))$,
  and $\mathop{\mathrm{top}} E(S) \cong \Nakr_C(S)$
  of right $C$-comodules.
\end{theorem}
\begin{proof}
  Let $S$ be a simple right $C$-comodule.
  Since $\Nakl_C$ is an equivalence, $\Nakl_C(E(S))$ is an injective hull of $\Nakl_C(S)$.
  By Lemma~\ref{lem:Nakayama-semiperfect-2}, $\Nakl_C(E(S))$ is also isomorphic to $P(S)$.
  Thus we have the first isomorphism. The third one is obtained in a similar way.
  Remaining ones follow from $\socle E(S) \cong S$ and $\mathop{\mathrm{top}} P(S) \cong S$.
\end{proof}

It is known that a QcF coalgebra is co-Frobenius if and only if the socle and the top of any indecomposable injective comodule have the same dimension \cite[Corollary 2.3]{MR3150709}. Combining this fact with Theorem~\ref{thm:Nakayama-permutation}, we obtain the following characterization of co-Frobenius coalgebras:

\begin{theorem}
  \label{thm:QcF-to-be-co-Frob}
  For a QcF coalgebra $C$, the following are equivalent:
  \begin{enumerate}
  \item $C$ is a co-Frobenius coalgebra.
  \item $\dim_{\bfk} \Nakl_{C}(S) = \dim_{\bfk} S$ for all simple right $C$-comodule $S$.
  \item $\dim_{\bfk} \Nakr_{C}(S) = \dim_{\bfk} S$ for all simple right $C$-comodule $S$.
  \end{enumerate} 
\end{theorem}

\subsection{Symmetric coalgebras}

Let $C$ be a coalgebra. A coalgebra automorphism $f : C \to C$ is said to be {\em coinner} if there is an invertible element $\alpha \in C^*$ such that $f(c) = \alpha \rightharpoonup c \leftharpoonup \alpha^{-1}$ for all $c \in C$. The following lemma is standard and its proof is omitted.

\begin{lemma}
  \label{lem:inner-auto}
  For a coalgebra automorphism $f : C \to C$, the following are equivalent:
  \begin{enumerate}
  \item $f$ is a coinner automorphism.
  \item The functor $(-)^{(f)} : \Mod^C \to \Mod^C$ is isomorphic to the identity functor.
  \end{enumerate}
\end{lemma}

Now we recall from \cite[Definition 3.2]{MR2078404} the following class of coalgebras:

\begin{definition}
  A coalgebra $C$ is said to be {\em symmetric} if there is an injective homomorphism $C \to C^*$ of $C^*$-bimodules.
\end{definition}

It is known that a symmetric coalgebra is precisely a co-Frobenius coalgebra whose Nakayama automorphism is coinner \cite[Proposition 6.2]{MR2078404}. By rephrasing this fact, we give the following characterization of symmetric coalgebras in terms of the Nakayama functor:

\begin{theorem}
  \label{thm:Nakayama-symmetric}
  The following are equivalent:
  \begin{enumerate}
  \item $C$ is a symmetric coalgebra.
  \item $\Nakl_C : \Mod^C \to \Mod^C$ is isomorphic to the identity functor.
  \item $\Nakr_C : \Mod^C \to \Mod^C$ is isomorphic to the identity functor.
  \end{enumerate}
\end{theorem}
\begin{proof}
  We remark that (2) and (3) are equivalent since $\Nakr_C$ is left adjoint to $\Nakl_C$.
  If (1) holds, then the Nakayama automorphism of $C$ is coinner, and thus (2) and (3) hold by Lemma \ref{lem:Nakayama-co-Fro}.
  
  Finally, we suppose that (3) holds.
  Then, by Theorems~\ref{thm:Nakayama-QcF} and \ref{thm:QcF-to-be-co-Frob}, $C$ is a co-Frobenius coalgebra.
  By Lemmas~\ref{lem:Nakayama-co-Fro} and \ref{lem:inner-auto}, the Nakayama automorphism of $C$ is coinner.
  Thus $C$ is a symmetric coalgebra. The proof is done.
\end{proof}

\subsection{Counterexamples}
\label{subsec:counterexamples}

A {\em quiver} is a quadruple $Q = (Q_0, Q_1, \src, \tgt)$ consisting of the non-empty set $Q_0$ of {\em vertices}, the set $Q_1$ of {\em arrows} and two maps $\src$ and $\tgt$ from $Q_1$ to $Q_0$, called the {\em source} and the {\em target}, respectively.

We fix a quiver $Q = (Q_0, Q_1, \src, \tgt)$. According to \cite[Section 3]{MR3125851}, we construct the coalgebra $C_Q$ as follows: As a vector space, $C_Q$ has a basis $Q_0 \sqcup Q_1$. The coalgebra structure is given by
\begin{equation*}
  \Delta(v) = v \otimes v,
  \quad \varepsilon(v) = 1,
  \quad \Delta(e) = \src(e) \otimes e + e \otimes \tgt(e),
  \quad \varepsilon(e) = 0
\end{equation*}
for a vertex $v \in Q_0$ and an arrow $e \in Q_1$.
Every simple comodule of $C_Q$ is isomorphic to $\bfk v$ for some $v \in Q_0$. For $v \in Q_0$, we define
\begin{equation*}
E(v) = \bfk v \oplus \Span_{\bfk} \, \src^{-1}(v) \quad \text{and} \quad
F(v) = \bfk v \oplus \Span_{\bfk} \, \tgt^{-1}(v).
\end{equation*}
The subspaces $E(v)$ and $F(v)$ of $C_Q$ are an injective hull of the right $C_Q$-comodule $\bfk v$ and the left $C_Q$-comodule $\bfk v$, respectively. Thus, as noted in \cite[Section 3]{MR3125851}, $C_Q$ is left (respectively, right) semiperfect if and only if the set $\src^{-1}(v)$ (respectively, $\tgt^{-1}(v)$) is finite for all $v \in Q_0$.

For notational convenience, we define the right $C_Q$-comodule $\bfk_v$ ($v \in Q_0$) to be the vector space $\bfk$ equipped with the coaction given by $1_{\bfk} \mapsto 1_{\bfk} \otimes v$. The following lemma is useful to compute the left exact Nakayama functor for $C_Q$.

\begin{lemma}
  \label{lem:Hom-C-C-kw}
  We fix a quiver $Q$ and write $C = C_Q$.
  Let $w \in Q_0$ be a vertex of $Q$.
  If $\src^{-1}(w) \ne \emptyset$, then there is an isomorphism
  \begin{equation*}
    \Hom^C(C, \bfk_w) \cong \prod_{e \in \tgt^{-1}(w)} \bfk_{\src(e)}
  \end{equation*}
  of left $C^*$-modules. Otherwise, $\Hom^C(C, \bfk_w)$ fits into an exact sequence
  \begin{equation*}
    0 \longrightarrow \prod_{e \in \tgt^{-1}(w)} \bfk_{\src(e)} \longrightarrow \Hom^C(C, \bfk_w) \longrightarrow \bfk_w \longrightarrow 0
  \end{equation*}
  of left $C^*$-modules.
\end{lemma}
\begin{proof}
  We write $\mathfrak{X} = \Hom^C(C, \bfk_w)$ for short.
  Since $C = \bigoplus_{v \in Q_0} E(v)$, we have
  \begin{equation}
    \label{eq:proof-Hom-C-C-kw-1}
    \mathfrak{X} \cong \prod_{v \in Q_0} \Hom^C(E(v), \bfk_w)
    \subset \prod_{v \in Q_0} E(v)^*
  \end{equation}
  as vector spaces. We fix $v \in Q_0$. By the definition of the right $C$-comodule $E(v)$, a linear map $f: E(v) \to \bfk$ belongs to the set $\Hom^C(E(v), \bfk_w)$ if and only if the following \eqref{eq:proof-Hom-C-C-kw-3} and~\eqref{eq:proof-Hom-C-C-kw-4} hold:
  \begin{gather}
    \label{eq:proof-Hom-C-C-kw-3}
    f(v) v = f(v) w. \\
    \label{eq:proof-Hom-C-C-kw-4}
    \text{$f(v) e + f(e) \tgt(e) = f(e) w$ for all $e \in \src^{-1}(v)$}.
  \end{gather}
  We first consider the case where $\src^{-1}(w) \ne \emptyset$.
  Given $f \in \Hom^C(E(v), \bfk_w)$, we prove that $f(v) = 0$ and $f(e) = 0$ for all $e \in \src^{-1}(v)$ such that $\tgt(e) \ne w$ as follows:
  \begin{enumerate}
  \item Suppose $v \ne w$. Then $f(v) = 0$ by \eqref{eq:proof-Hom-C-C-kw-3}.
    Let $e \in Q_1$ be an arrow such that $\src(e) = v$ and $\tgt(e) \ne w$. Then we have $f(e) = 0$ by \eqref{eq:proof-Hom-C-C-kw-4}.
  \item Suppose $v = w$. We fix an element $e_0 \in \src^{-1}(v)$ ($= \src^{-1}(w) \ne \emptyset$).
    Then, by \eqref{eq:proof-Hom-C-C-kw-4}, we have $f(v) e_0 + f(e_0) \tgt(e_0) = f(e_0) w$. Since $e_0$ is not a linear combination of elements of $Q_0$, we obtain $f(v) = 0$. Now let $e \in Q_1$ be an arrow such that $\src(e) = v$ and $\tgt(e) \ne w$. Then, by \eqref{eq:proof-Hom-C-C-kw-4}, we have $f(e) \tgt(e) = f(e) w$. Since $w \ne \tgt(e)$, we get $f(e) = 0$.
  \end{enumerate}
  For $x, y \in Q_0$, we set $Q(x,y) := \{ e \in Q_1 \mid \src(e) = x, \tgt(e) = y \}$.
  In either of the cases (1) or (2), we see that $f(e)$ for $e \in Q(v, w)$ can be arbitrary.
  By the above discussion, we have
  \begin{equation*}
    \Hom^C(E(v), \bfk_w)
    = \left\{ f \in C^* \mathrel{}\middle|\mathrel{}
    \begin{gathered}
    \text{$f(x) = 0$ for all $x \in Q_0$ and} \\[-2pt]
    \text{$f(e) = 0$ for all $e \in Q_1 \setminus Q(v,w)$}
    \end{gathered}
    \right\}
  \end{equation*}
  if we regard $\Hom^C(E(v), \bfk_w)$ as a subspace of $\mathfrak{X}$ by \eqref{eq:proof-Hom-C-C-kw-1}.
  Now we compute the action of $C^*$ on $\mathfrak{X}$. For $f \in \Hom^C(E(v), \bfk_w)$ and $c^* \in C^*$, we have
  \begin{gather*}
    \langle c^* \cdot f, x \rangle
    = \langle c^*, x \rangle \langle f, x \rangle = 0
    = \langle c^*, v \rangle \langle f, v \rangle \quad (x \in Q_0), \\
    \langle c^* \cdot f, e \rangle
    = \langle c^*, \src(e) \rangle \langle f, e \rangle + \langle c^*, e \rangle \langle f, \tgt(e) \rangle
    = \langle c^*, v \rangle \langle f, e \rangle
    \quad (e \in Q_1).
  \end{gather*}
  Namely, $c^* \cdot f = \langle c^*, v \rangle f$. Hence the map
  \begin{equation}
    \label{eq:proof-Hom-C-C-kw-2}
    \Hom^C(E(v), \bfk_w) \cong \prod_{e \in Q(v,w)} \bfk_v, \quad f \mapsto (f(e))_{e \in Q(v,w)}
  \end{equation}
  is an isomorphism of left $C^*$-modules.
  By \eqref{eq:proof-Hom-C-C-kw-1} and \eqref{eq:proof-Hom-C-C-kw-2}, we have isomorphisms
  \begin{equation*}
    \mathfrak{X} \cong \prod_{v \in Q_0} \prod_{e \in Q(v,w)} \bfk_{v} \cong \prod_{e \in \tgt^{-1}(w)} \bfk_{\src(e)}
  \end{equation*}
  of left $C^*$-modules. The proof for the case where $\src^{-1}(w) \ne \emptyset$ is done.

  Next, we consider the case where $\src^{-1}(w) = \emptyset$.
  Then there is a right $C$-comodule map $\xi: C \to \bfk_w$ defined by $\xi(v) = \delta_{v,w}$ for all $v \in Q_0$ and $\xi(e) = 0$ for all $e \in Q_1$.
  Now we set $\mathfrak{X}_0 = \{ f \in \mathfrak{X} \mid f(w) = 0 \}$ and $\mathfrak{X}_1 = \bfk \xi$ so that $\mathfrak{X} = \mathfrak{X}_0 \oplus \mathfrak{X}_1$.
  By the same argument as above, we see that $\mathfrak{X}_0$ is a $C^*$-submodule of $\mathfrak{X}$ and there is an isomorphism $\mathfrak{X}_0 \cong \prod_{e \in \tgt^{-1}(w)} \bfk_{\src(e)}$ of left $C^*$-modules.
  For $f \in \mathfrak{X}$, the decomposition $f = f_0 + f_1$ ($f_i \in \mathfrak{X}_i$) is given by $f_1 = f(w) \xi$ and $f_0 = f - f_1$. Thus, for $c^* \in C^*$, we have $c^* \cdot \xi = c^*(w) \xi$ in $\mathfrak{X}/\mathfrak{X}_0$. This means that $\mathfrak{X}/\mathfrak{X}_0$ is isomorphic to $\bfk_w$ as a left $C^*$-module. Now the short exact sequence of the statement of this lemma is obtained from $0 \to \mathfrak{X}_0 \to \mathfrak{X} \to \mathfrak{X}/\mathfrak{X}_0 \to 0$. The proof is done.
\end{proof}

\begin{example}
  \label{ex:counter-example-0}
  We fix an infinite set $P$ and consider the coalgebra $C = C_Q$ associated to the quiver $Q$ defined by $Q_0 = \{ u, v, w \}$, $Q_1 = \{ e_0 \} \sqcup P$, $\src(e_0) = u$, $\tgt(e_0) = v$, $\src(e) = v$ and $\tgt(e) = w$ for $e \in P$ (see Figure~\ref{fig:quiver-1}, where $\displaystyle \mathop{\bullet}^{a} \xrightarrow{\ e \ } \mathop{\bullet}^b$ means that $e$ is en element of $Q_1$ such that $a = \src(e)$ and $b = \tgt(e)$).
  We set  $\mathfrak{X} = \Hom^C(C, \bfk_w)$ and define $\xi$, $\mathfrak{X}_0$ and $\mathfrak{X}_1$ as in the proof of Lemma~\ref{lem:Hom-C-C-kw}. Then, by that lemma, we have
  \begin{equation*}
    \mathfrak{X}_0
    \cong \prod_{e \in \tgt^{-1}(w)} \bfk_{\src(e)}
    \cong \prod_{|P|} \bfk_{v}
    \cong \bigoplus_{\alpha} \bfk_{v}
    \quad (\alpha := \dim_{\bfk} \bfk^{|P|}).
  \end{equation*}
  Thus the left $C^*$-module $\mathfrak{X}_0$ is rational.
  For $f \in Q_1$, we define $\delta_f \in C^*$ by $\delta_{f}(e) = \delta_{e,f}$ ($e \in Q_1$) and $\delta_{f}(x) = 0$ ($x \in Q_0$).
  Then we have $\delta_f \cdot \xi = \delta_f$ for all $f \in Q_1$.
  Since $\delta_f$'s are linearly independent, the space $C^* \cdot \xi$ is infinite-dimensional.
  By \cite[Theorem 2.2.6]{MR1786197}, $\xi$ does not belong to the rational part of $\mathfrak{X}$.
  Hence $\Nakl_C(\bfk_w) = \mathfrak{X}^{\rat} = \mathfrak{X}_0 \subsetneq \mathfrak{X}$.

  This example shows that the left $C^*$-module $\Hom^C(C, V)$ for $V \in \Mod^C$ is not rational in general.
  It also shows that $\Nakl_C(E)$ for $E \in \Inj(\Mod^C)$ is not projective in general.
  Indeed, since $\bfk v \in {}^C\Mod$ is not injective, $(\bfk v)^* \in \Mod^C$ is not projective.
  The right $C$-comodule $\mathfrak{X}^{\rat} = \Nakl_C(\bfk_w)$ is not projective since its direct summand $\bfk_v \cong (\bfk v)^*$ is not, while $\bfk_w = E(w)$ is injective.
\end{example}

\begin{example}
  \label{ex:counter-example-1}
  We fix an infinite set $P$ and consider the coalgebra $C = C_Q$ associated to the quiver $Q$ defined by $Q_0 = \{ w, w' \} \sqcup P$, $Q_1 = \{ e_0 \} \sqcup P$, $\src(e_0) = w$, $\tgt(e_0) = w'$, $\src(v) = v$ and $\tgt(v) = w$ for all $v \in P$ (see Figure~\ref{fig:quiver-2}).
  We set $\mathfrak{X} = \Hom^C(C, \bfk_w)$.
  By Lemma~\ref{lem:Hom-C-C-kw} we have $\mathfrak{X} \cong \prod_{v \in P} \bfk_v$ as a left $C^*$-module.
  Since the right $C$-comodule $\bigoplus_{v \in P} \bfk_v$ is quasi-finite, the rational part of $\mathfrak{X}$ is isomorphic to $\bigoplus_{v \in P} \bfk_v$ by \cite[Lemma 2.5]{MR2253657}. Thus, in particular, $\Nakl_C(\bfk_w) = \mathfrak{X}^{\rat}$ is infinite-dimensional.
  This example shows that $\Nakl_C$ does not preserve $\fdmod^C$, even if $C$ is left semiperfect ({\it cf}. Lemma~\ref{lem:Nakayama-semiperfect-4}).
\end{example}

\begin{figure}
  \begin{equation*}
    \begin{tikzpicture}[x = 24pt, y = 24pt]
      \node (u) at (0,0) {$\bullet$};
      \node (v) at (3,0) {$\bullet$};
      \node (w) at (6,0) {$\bullet$};
      \node at (u) [above = 2] {$u$};
      \node at (v) [above = 2] {$v$};
      \node at (w) [above = 2] {$w$};
      \draw [->] (u) -- node[midway,above] {$e_0$} (v);
      \draw [->] (v) to [out= 25, in=155] (w);
      \draw [->] (v) to [out=-25, in=205] (w);
      \node at (4.5, 0.1) {$\vdots$};
      \node at (5, -.75) [right] {$|P|$ arrows};
    \end{tikzpicture}
  \end{equation*}
  \caption{The quiver $Q$ of Example~\ref{ex:counter-example-0}}
  \label{fig:quiver-1}

  \begin{equation*}
    \begin{tikzpicture}[x = 24pt, y = 24pt]
      \node (w) at (0,0) {$\bullet$};
      \node (u) at (3,0) {$\bullet$};
      \node at (.25, .35) {$w$};
      \node at (u) [above = 2] {$w'$};
      \draw [->] (w) -- node[midway,above] {$e_0$} (u);
      \path (w) -- ++(120:2) node (v1) {$\bullet$};
      \path (w) -- ++(160:2) node (v2) {$\bullet$};
      \path (w) -- ++(210:2) node (v3) {$\bullet$};
      \draw [->] (v1) -- (w);
      \draw [->] (v2) -- (w);
      \draw [->] (v3) -- (w);
      \node at (v1) [right = 2] {$v \in P$};
      \node at (-1, 0) {$\vdots$};
      \node at (-1.5, -.25) [left] {$|P|$ arrows};
    \end{tikzpicture}
  \end{equation*}
  \caption{The quiver $Q$ of Example~\ref{ex:counter-example-1}}
  \label{fig:quiver-2}
\end{figure}

\begin{example}
  \label{ex:counter-example-2}
  Let $C$ be the left semiperfect coalgebra considered in Example~\ref{ex:counter-example-1}, and let $D = C^{\cop}$.
  Then $D$ is {\em right} semiperfect. For $M \in \fdmod^D$, there is an isomorphism $\Nakr_D(M)^* \cong \Nakl_C(M^*)$ of vector spaces by the proof of Lemma~\ref{lem:Nakayama-semiperfect-5}. Thus, by the discussion in Example~\ref{ex:counter-example-1}, $\Nakr_D$ does not preserve $\fdmod^D$.
\end{example}

\begin{example}
  \label{ex:counter-example-3}
  Let $C$ be a coalgebra. Then, by the tensor-Hom adjunction, we have
  \begin{align*}
    \Nakr_C(M)^*
    = \Hom_{\bfk}(C \otimes_{C^*} M, \bfk)
    & \cong {}_{C^*}\Hom(M, \Hom_{\bfk}(C, \bfk)) \\
    & \cong \Hom^C(M, ({}_{C^*}C^*)^{\rat})
  \end{align*}
  for all $M \in \Mod^C$. Thus $\Nakr_C = 0$ (as a functor) if and only if $({}_{C^*}C^*)^{\rat} = 0$. Since $\Nakl_C$ is right adjoint to $\Nakr_C$, this is also equivalent to that $\Nakl_C = 0$. Thus, in particular, both $\Nakl_C$ and $\Nakr_C$ preserve $\fdmod^C$ if $({}_{C^*}C^*)^{\rat} = 0$.

  Now we consider the polynomial algebra $C = \bfk[X]$ and endow it with a structure of a Hopf algebra by $\Delta(X) = X \otimes 1_C + 1_C \otimes X$ and $\varepsilon(X) = 0$. The dual algebra $C^*$ is isomorphic to the algebra of formal power series. Since $C^* f$ is infinite-dimensional for all non-zero element $f \in C^*$, we have $({}_{C^*}C^*)^{\rat} = 0$. Thus, by the above discussion, $\Nakl_C = \Nakr_C = 0$. Since $\bfk 1_C$ is a unique simple comodule of $C$ (up to isomorphisms) and its injective hull is $C$, we have $\Inj(\fdmod^C) = \Proj(\fdmod^C) = 0$. Hence $\Nakl_C$ and $\Nakr_C$ induce equivalences between $\Inj(\fdmod^C)$ and $\Proj(\fdmod^C)$, although $C$ is not semiperfect ({\it cf}. Theorem~\ref{thm:Nakayama-semiperfect}).
\end{example}

\section{Nakayama functor for locally finite abelian categories}
\label{sec:Naka-for-locally-finite}

\subsection{Finiteness conditions for comodules}

A $\bfk$-linear abelian category $\mathcal{A}$ is said to be {\em locally finite} \cite[Definition 1.8.1]{MR3242743} if it is essentially small, every object of $\mathcal{A}$ is of finite length, and $\Hom_{\mathcal{A}}(X, Y)$ is finite-dimensional for all $X, Y \in \mathcal{A}$. It is known that a $\bfk$-linear category is a locally finite abelian category if and only if it is equivalent to $\fdmod^C$ for some coalgebra $C$ \cite[Theorem 1.9.15]{MR3242743}.
In this section, we rephrase the results of the previous section in the setting of locally finite abelian categories and their ind-completions. First, we recall the following finiteness conditions:

\begin{definition}
  By a {\em filtered colimit}, we mean a colimit of a functor from a small filtered category.
  Let $\mathcal{A}$ be a category, and let $X$ be an object of $\mathcal{A}$.
  We say that $X$ is {\em finitely presented} (respectively, {\em finitely generated}) if the functor $\Hom_{\mathcal{A}}(X, -) : \mathcal{A} \to \Sets$ preserves filtered colimits (respectively, filtered colimits of monomorphisms) existing in the category $\mathcal{A}$ \cite[Chapter 6]{MR2182076}. We denote by $\mathcal{A}_{\fp}$ and $\mathcal{A}_{\fg}$ the full subcategory of finitely presented objects and finitely generated objects of $\mathcal{A}$, respectively.
\end{definition}

If $\mathcal{A} = {}_R \Mod$ for some algebra $R$, then $\mathcal{A}_{\fp}$ and $\mathcal{A}_{\fg}$ are precisely the category of finitely presented and finitely generated left $R$-modules, respectively.
In this paper, we are interested in the case where $\mathcal{A}$ is the category of comodules over a coalgebra. The following result seems to be well-known:

\begin{lemma}
  \label{lem:compact-obj-in-comodules}
  If $\mathcal{A} = \Mod^C$ for some coalgebra $C$, then we have
  \begin{equation*}
      \mathcal{A}_{\fp} = \mathcal{A}_{\fg} = \fdmod^C.
  \end{equation*}
\end{lemma}
\begin{proof}
  Since $\mathcal{A}$ is a Grothendieck category \cite[Corollary 2.2.8]{MR1786197}, an object $X \in \mathcal{A}$ is finitely generated if and only if the following condition holds: For every directed system $\{ X_i \}_{i \in I}$ of subobjects of $X$ such that $X = \sum_{i \in I} X_i$, there is an element $i_0 \in I$ such that $X = X_{i_0}$ \cite[Chapter V, Proposition 3.2]{MR0389953}. By this characterization of finitely generated objects and the fundamental theorem for comodules, one can easily check that the equation $\mathcal{A}_{\fg} = \Mod^C_{\fd}$ holds.

  The fundamental theorem for comodules also implies that a set of representatives of the isomorphism classes of finite-dimensional right $C$-comodules is a generator of $\mathcal{A}$. Hence an object $X \in \mathcal{A}$ is finitely presented if and only if the following condition holds: For every epimorphism $q: W \to X$ in $\mathcal{A}$ with $W$ finitely generated, the kernel of $q$ is finitely generated \cite[Chapter V, Proposition 3.4]{MR0389953}. By this characterization of finitely presented objects, one can easily check that the equation $\mathcal{A}_{\fp} = \Mod^C_{\fd}$ holds. The proof is done.
\end{proof}

Given a category $\mathcal{A}$, we denote its ind-completion \cite[Chapter 6]{MR2182076} by $\Ind(\mathcal{A})$.
If $\mathcal{A}$ is a cocomplete category such that every object of $\mathcal{A}$ is a filtered colimit of finitely presented objects, then $\mathcal{A}$ is equivalent to $\Ind(\mathcal{A}_{\fp})$ \cite[Corollary 6.5]{MR2182076}.
By this fact and Lemma~\ref{lem:compact-obj-in-comodules}, we have $\Ind(\fdmod^C) \approx \Mod^C$ for any coalgebra $C$.

\subsection{Nakayama functor for locally finite abelian categories}

Let $\mathcal{A}$ be a locally finite abelian category.
By counterexamples given in Subsection \ref{subsec:counterexamples}, the (co)end formula for the Nakayama functors (Theorem \ref{thm:Nakayama-(co)end}) cannot be used to define endofunctors on $\mathcal{A}$ in general. Thus we first introduce endofunctors on $\Ind(\mathcal{A})$ as follows:

\begin{definition}
  We define two endofunctors $\Nakl_{\Ind(\mathcal{A})}$ and $\Nakr_{\Ind(\mathcal{A})}$ on $\Ind(\mathcal{A})$ by
  \begin{align*}
    \Nakl_{\Ind(\mathcal{A})}(M) & = \int_{X \in \mathcal{A}} \Hom_{\Ind(\mathcal{A})}(X, M) \copow X, \\
    \Nakr_{\Ind(\mathcal{A})}(M) & = \int^{X \in \mathcal{A}} \coHom_{\Ind(\mathcal{A})}(X, M) \copow X,
  \end{align*}
  respectively, for $M \in \Ind(\mathcal{A})$. We call the functors $\Nakl_{\Ind(\mathcal{A})}$ and $\Nakr_{\Ind(\mathcal{A})}$ the {\em left exact} and the {\em right exact Nakayama functor} for $\mathcal{A}$, respectively.
\end{definition}

By Theorem \ref{thm:Nakayama-(co)end}, $\Nakl_{\Ind(\mathcal{A})}$ and $\Nakr_{\Ind(\mathcal{A})}$ are identified with $\Nakl_C$ and $\Nakr_C$, respectively, if $\mathcal{A} = \fdmod^C$ for some coalgebra $C$.
Since the ind-completion of a locally finite abelian category is equivalent to $\Mod^C$ for some coalgebra $C$, we see that the functors $\Nakl_{\Ind(\mathcal{A})}$ and $\Nakr_{\Ind(\mathcal{A})}$ do exist for every locally finite abelian category $\mathcal{A}$.

We say that an abelian category $\mathcal{A}$ is {\em semiperfect} if every finitely generated object of $\mathcal{A}$ has a projective cover ({\it cf}. the definition of semiperfect rings). There is the following characterization of one-sided semiperfectness of coalgebras:

\begin{lemma}
  \label{lem:semiperfect}
  For a coalgebra $C$, the following are equivalent:
  \begin{enumerate}
  \item The coalgebra $C$ is right semiperfect.
  \item The category $\Mod^C$ is semiperfect.
  \item The category $\fdmod^C$ is semiperfect.
  \item The category $\fdmod^C$ has enough projective objects.
  \item The category ${}^C\fdmod$ has enough injective objects.
  \end{enumerate}
\end{lemma}
\begin{proof}
  It is trivial from the definition of right semiperfectness that (1) and (2) are equivalent.
  It is also trivial that (3) implies (4).
  Since there is an anti-equivalence between $\fdmod^C$ and ${}^C\fdmod$, the assertions (4) and (5) are equivalent.
  The implication (4) $\Rightarrow$ (1) follows from (iii) $\Rightarrow$ (i) of \cite[Corollary 2.4.21]{MR1786197} applied to $C^{\cop}$.
  The implication (1) $\Rightarrow$ (3) follows from Lemmas~\ref{lem:fd-injective} and \ref{lem:DNR-thm-323-plus}.
\end{proof}

Let $\mathcal{A}$ be a locally finite abelian category.
We are interested in when $\Nakl_{\Ind(\mathcal{A})}$ and $\Nakr_{\Ind(\mathcal{A})}$ induce endofunctors on the full subcategory $\mathcal{A} \subset \Ind(\mathcal{A})$. For this to hold, the semiperfectness of $\mathcal{A}$ is not sufficient as we have seen in Example \ref{ex:counter-example-2}.
Here we give the following sufficient condition:

\begin{theorem}
  \label{thm:Nakayama-compact-restriction}
  Let $\mathcal{A}$ be a locally finite abelian category.
  Suppose that $\mathcal{A}$ has both enough projectives and enough injectives.
  Then $\Nakl_{\Ind(\mathcal{A})}$ and $\Nakr_{\Ind(\mathcal{A})}$ restrict to endofunctors on $\mathcal{A}$.
\end{theorem}
\begin{proof}
  Let $C$ be a coalgebra such that $\mathcal{A} \approx \fdmod^C$.
  By Lemma~\ref{lem:semiperfect} and the assumption that $\mathcal{A}$ has enough projective objects, $C$ is right semiperfect.
  Since ${}^C\fdmod$ ($\approx \mathcal{A}^{\op}$) has enough projective objects, $C$ is left semiperfect again by Lemma~\ref{lem:semiperfect}.
  Thus $C$ is semiperfect.
  Now the result follows from Theorem~\ref{thm:Nakayama-semiperfect}.
\end{proof}

\begin{definition}
  Suppose that $\mathcal{A}$ is a locally finite abelian category satisfying the assumption of Theorem \ref{thm:Nakayama-compact-restriction}. Then we denote by $\Nakl_{\mathcal{A}}$ and $\Nakr_{\mathcal{A}}$ the endofunctors on $\mathcal{A}$ defined as restrictions of $\Nakl_{\Ind(\mathcal{A})}$ and $\Nakr_{\Ind(\mathcal{A})}$, respectively.
\end{definition}

\subsection{Interaction with adjoint functors}

Fuchs, Schaumann and Schweigert \cite{MR4042867} observed that several non-trivial results on finite tensor categories easily follow from the universal property of the Nakayama functor.
Applications of our results to tensor categories, which are not necessarily finite, will be given in the next section.
Here we remark the following relation between the Nakayama functor and adjoint functors, which generalizes \cite[Theorem 3.18]{MR4042867} to the locally finite setting.

\begin{theorem}
  \label{thm:Nakayama-and-adjunctions}
  Let $\mathcal{A}$ and $\mathcal{B}$ be locally finite abelian categories, and let $F: \Ind(\mathcal{A}) \to \Ind(\mathcal{B})$ be a $\bfk$-linear functor.
  Then the following hold:
  \begin{enumerate}
  \item [\rm (a)] Suppose that $F$ preserves limits and the double right adjoint $F^{\rradj} := (F^{\radj})^{\radj}$ exists, there is an isomorphism of functors
    \begin{equation*}
      \phi^{\L}_{F} : F \circ \Nakl_{\Ind(\mathcal{A})}
      \xrightarrow{\quad \cong \quad} \Nakl_{\Ind(\mathcal{B})} \circ F^{\rradj}.
    \end{equation*}
  \item [\rm (b)] Suppose that $F$ preserves colimits and the double left adjoint $F^{\lladj} := (F^{\ladj})^{\ladj}$ exists, there is a canonical isomorphism of functors
    \begin{equation*}
      \phi^{\R}_{F} :
      F \circ \Nakr_{\Ind(\mathcal{A})}
      \xrightarrow{\quad \cong \quad}
      \Nakr_{\Ind(\mathcal{B})} \circ F^{\lladj}.
    \end{equation*}
  \end{enumerate}
  The isomorphisms $\phi^{\L}_F$ and $\phi^{\R}_F$ are natural in the variable $F$ and `coherent' in the following sense: For every functors $F: \Ind(\mathcal{A}) \to \Ind(\mathcal{B})$ and $G: \Ind(\mathcal{B}) \to \Ind(\mathcal{C})$, where $\mathcal{A}$, $\mathcal{B}$ and $\mathcal{C}$ are locally finite abelian categories, the diagram
  \begin{equation*}
    \begin{tikzcd}[column sep = 80pt]
      G \circ F \circ \Nakl_{\Ind(\mathcal{A})}
      \arrow[r, "\phi^{\L}_{G \circ F}"]
      \arrow[d, "\id_G \circ \phi^{\L}_{F}"']
      & \Nakl_{\Ind(\mathcal{C})} \circ (G \circ F)^{\rradj}
      \arrow[d] \\
      G \circ \Nakl_{\Ind(\mathcal{B})} \circ F^{\rradj}
      \arrow[r, "\phi^{\L}_{G} \circ \id_{F^{\rradj}}"]
      & \Nakl_{\Ind(\mathcal{C})} \circ G^{\rradj} \circ F^{\rradj}
    \end{tikzcd}
  \end{equation*}
  commutes whenever $F$ and $G$ satisfy the assumptions of Part (a).
  Here, the unlabeled arrow in the diagram represents the canonical isomorphism $(G \circ F)^{\rradj} \cong G^{\rradj} \circ F^{\rradj}$ arising from the uniqueness of right adjoints.
  An analogous diagram for $\phi^{\R}$ also commutes provided that $F$ and $G$ satisfy the assumptions of Part (b).
\end{theorem}
\begin{proof}
  The proof goes along a similar way as \cite[Theorem 3.18]{MR4042867}.
  Suppose that $F$ fulfills the assumption of Part (a).
  We note that $F$ preserves direct sums as it is assumed to have a right adjoint.
  We define $\phi^{\L}_F: F \circ \Nakl_{\Ind(\mathcal{A})} \to \Nakl_{\Ind(\mathcal{B})} \circ F^{\rradj}$ to be the composition of natural isomorphisms
  \begin{align*}
    F(\Nakl_{\Ind(\mathcal{A})}(M))
    & = \textstyle F(\int_{X \in \mathcal{A}} \Hom_{\Ind(\mathcal{A})}(X, M) \copow X) \\
    & \cong \textstyle \int_{X \in \mathcal{A}} \Hom_{\Ind(\mathcal{A})}(X, M) \copow F(X) \\
    & \cong \textstyle \int_{Y \in \mathcal{B}} \Hom_{\Ind(\mathcal{A})}(F^{\radj}(Y), M) \copow Y \\
    & \cong \textstyle \int_{Y \in \mathcal{B}} \Hom_{\Ind(\mathcal{B})}(Y, F^{\rradj}(M)) \copow Y
      = \Nakl_{\Ind(\mathcal{B})}(F^{\rradj}(M))
  \end{align*}
  for $M \in \Ind(\mathcal{A})$,
  where the first isomorphism follows from the assumption that $F$ preserves limits and direct sums, the second one is \eqref{eq:adjunction-and-end}, and the third one is the adjunction isomorphism.
  Since each isomorphism is `coherent' in a similar sense as for $\phi^{\L}$, we obtain the coherent property of the isomorphism $\phi^{\L}_F$ as stated.
  
  Similarly, if $F$ satisfies the assumption of Part (b), then we have
  \begin{align*}
    F(\Nakr_{\Ind(\mathcal{A})}(M))
    & = \textstyle F(\int^{X \in \mathcal{A}} \coHom_{\Ind(\mathcal{A})}(X, M) \copow X) \\
    & \cong \textstyle \int^{X \in \mathcal{A}} \coHom_{\Ind(\mathcal{A})}(X, M) \copow F(X) \\
    & \cong \textstyle \int^{Y \in \mathcal{B}} \coHom_{\Ind(\mathcal{A})}(F^{\ladj}(Y), M) \copow Y \\
    & \cong \textstyle \int^{Y \in \mathcal{B}} \coHom_{\Ind(\mathcal{B})}(Y, F^{\lladj}(M)) \copow Y
      = \Nakr_{\Ind(\mathcal{B})}(F^{\lladj}(M))
  \end{align*}
  for $M \in \Ind(\mathcal{A})$. Here, the first isomorphism follows from the assumption that $F$ preserves colimits,
  the second from \eqref{eq:adjunction-and-coend}, and the third from \eqref{eq:adjunction-coHom}.
  By the same reason as in Part (a), the isomorphism obtained by the composition is coherent.
\end{proof}

The following variant of Theorem~\ref{thm:Nakayama-and-adjunctions} is essential for our applications:

\begin{corollary}
  \label{cor:Nakayama-and-adjunctions}
  Let $\mathcal{A}$ and $\mathcal{B}$ be locally finite abelian categories that have both enough projectives and injectives, and let $F: \mathcal{A} \to \mathcal{B}$ be a $\bfk$-linear functor.
  \begin{itemize}
  \item [(a)] If both $F^{\ladj}$ and $F^{\rradj}$ exist, then $F \circ \Nakl_{\mathcal{A}} \cong \Nakl_{\mathcal{B}} \circ F^{\rradj}$ as functors.
  \item [(b)] If both $F^{\radj}$ and $F^{\lladj}$ exist, then $F \circ \Nakr_{\mathcal{A}} \cong \Nakr_{\mathcal{B}} \circ F^{\lladj}$ as functors.
  \end{itemize}
\end{corollary}
\begin{proof}
  We only prove Part (a), since the other one is verified in a similar way.
  Suppose that $F$ fulfills the assumptions of Part (a).
  By the functorial property of the ind-completion, we have a chain
  \begin{equation*}
    \Ind(F^{\ladj}) \dashv \Ind(F) \dashv \Ind(F^{\radj}) \dashv \Ind(F^{\rradj})
  \end{equation*}
  of adjunctions. By applying Theorem~\ref{thm:Nakayama-and-adjunctions} (a) to $\Ind(F) : \Ind(\mathcal{A}) \to \Ind(\mathcal{B})$, we have an isomorphism
  \begin{equation*}
    \Ind(F) \circ \Nakl_{\Ind(\mathcal{A})} \cong \Nakl_{\Ind(\mathcal{B})} \circ \Ind(F^{\rradj}).
  \end{equation*}
  The proof is completed by restricting the both sides to $\mathcal{A} = \Ind(\mathcal{A})_{\fp}$.
\end{proof}

\section{Applications to Frobenius tensor categories}
\label{sec:Frobenius-tensor}

\subsection{Conventions for tensor categories}

In this section, we give applications of our results on Nakayama functors to tensor categories.
We first fix conventions on monoidal categories.
Our notation and terminology basically follow \cite{MR3242743}.
In view of Mac Lane's strictness theorem, we assume that every monoidal category is strict.
Unless otherwise noted, the tensor product and the unit object of a monoidal category will be written as $\otimes$ and $\unitobj$, respectively.
Given a monoidal category $\mathcal{C}$, we denote by $\mathcal{C}^{\rev}$ the monoidal category obtained from $\mathcal{C}$ by reversing the order of the tensor product.

Let $L$ and $R$ be objects of $\mathcal{C}$, and let $\varepsilon: L \otimes R \to \unitobj$ and $\eta: \unitobj \to R \otimes L$ be morphisms in $\mathcal{C}$.
We say that the triple $(L, \varepsilon, \eta)$ is a {\em left dual object} of $R$ and the triple $(R, \varepsilon, \eta)$ is a {\em right dual object} if the equations $(\varepsilon \otimes \id_L) (\id_L \otimes \eta) = \id_L$ and $(\id_R \otimes \varepsilon) (\eta \otimes \id_R) = \id_R$ hold. We say that $\mathcal{C}$ is {\em left (right) rigid} if every object of $\mathcal{C}$ has a left (right) dual object. A {\em rigid monoidal category} is a left and right rigid monoidal category.

Given an object $X$ of a left rigid monoidal category $\mathcal{C}$, we denote by
\begin{equation*}
  (X^{\vee}, \ \eval_X : X^{\vee} \otimes X \to \unitobj, \ \coev_X : \unitobj \to X \otimes X^{\vee})
\end{equation*}
the fixed left dual object of $X$. The assignment $X \mapsto X^{\vee}$ extends to a contravariant strong monoidal functor from $\mathcal{C}$ to $\mathcal{C}^{\rev}$, which we call the {\em left duality functor} of $\mathcal{C}$. By replacing $\mathcal{C}$ with an equivalent one and choose left dual objects in an appropriate way, we may assume that the left duality functor $(-)^{\vee}$ is a strict monoidal functor.

A left rigid monoidal category is rigid if and only if its left duality functor is an anti-equivalence.
Let $\mathcal{C}$ be a rigid monoidal category.
A quasi-inverse of the left duality functor of $\mathcal{C}$, which we denote by $X \mapsto {}^{\vee}X$, is given by taking a right dual object.
We may assume that $(-)^{\vee}$ and ${}^{\vee}(-)$ are strict monoidal functors and mutually inverse to each other.

A {\em tensor category} \cite{MR3242743} is a locally finite abelian category $\mathcal{C}$ equipped with a structure of a rigid monoidal category such that the tensor product $\otimes: \mathcal{C} \times \mathcal{C} \to \mathcal{C}$ is $\bfk$-linear in each variable and the unit object $\unitobj \in \mathcal{C}$ is absolutely simple, that is, it is a simple object such that $\End_{\mathcal{C}}(\unitobj) \cong \bfk$.

An object of a tensor category $\mathcal{C}$ is said to be {\em trivial} if it is isomorphic to the direct sum of finitely many copies of the unit object $\unitobj$.
We denote by $\mathcal{C}_{\triv}$ the full subcategory of $\mathcal{C}$ consisting of all trivial objects of $\mathcal{C}$.

\begin{lemma}
  \label{lem:NS-Lemma-7-1}
  Let $\mathcal{C}$ be a tensor category. There is a unique natural isomorphism
  \begin{equation}
    \tau_{X, T} : X \otimes T \to T \otimes X
    \quad (X \in \mathcal{C}, T \in \mathcal{C}_{\triv})
  \end{equation}
  characterized by the property that $\tau_{X,\unitobj} = \id_X$ for all $X \in \mathcal{C}$.
  If $\mathcal{C}$ has a braiding $\sigma$, then we have
  \begin{equation}
    \label{eq:NS-Lemma-7-1-braided}
    \tau_{X,T} = \sigma_{X,T} = (\sigma_{T,X})^{-1}
    \quad (X \in \mathcal{C}, T \in \mathcal{C}_{\triv}).
  \end{equation}
\end{lemma}
\begin{proof}
  See \cite[Lemma 7.1]{MR2381536} for the existence of such a natural isomorphism (there $\mathcal{C}$ is assumed to be semisimple, however, that assumption is not needed for proving this). Equation~\eqref{eq:NS-Lemma-7-1-braided} follows from the fact that the braiding $\sigma$ satisfies $\sigma_{X,\unitobj} = \id_X = (\sigma_{\unitobj,X})^{-1}$ for all $X \in \mathcal{C}$.
\end{proof}

This lemma is helpful when we consider copowers in a tensor category. Let $\mathcal{C}$ be a tensor category. Then the category $\Vect_{\fd}$ is identified with $\mathcal{C}_{\triv}$ through the functor sending $\bfk$ to $\unitobj$. Under this identification, the copower of $T \in \Vect_{\fd}$ and $X \in \mathcal{C}$ is just the tensor product $T \otimes X$ of $T \in \mathcal{C}_{\triv}$ and $X \in \mathcal{C}$.
Since every $\bfk$-linear functor $F: \mathcal{C} \to \mathcal{C}$ preserves finite direct sums, there is a canonical isomorphism $F(T \otimes X) \cong T \otimes F(X)$ for $T \in \mathcal{C}_{\triv}$ and $X \in \mathcal{C}$.
If $F = X \otimes (-)$ for some $X \in \mathcal{C}$, then this canonical isomorphism is given by
\begin{equation*}
  F(T \otimes Y) = X \otimes T \otimes Y
  \xrightarrow{\quad \tau_{X,T} \otimes \id_Y \quad}
  T \otimes X \otimes Y
  = T \otimes F(Y)
\end{equation*}
for $Y \in \mathcal{C}$ and $T \in \Vect_{\fd} = \mathcal{C}_{\triv}$.

\subsection{A remark on `one-sided rigid' tensor categories}

Let $H$ be a Hopf algebra.
Then the $\bfk$-linear locally finite abelian category $\fdmod^H$ is naturally a left rigid monoidal category, however, it may not be a tensor category because of the lack of right rigidity.
Indeed, the monoidal category $\fdmod^H$ is rigid if and only if the antipode of $H$ is bijective \cite{MR1098991}, and some examples of Hopf algebras with non-bijective antipodes are known \cite{MR292876}. It could be important to know when the one-sided rigidity of a monoidal category implies the rigidity on the other side. Motivated by the fact that a one-sided semiperfect Hopf algebra is QcF and the antipode of such a Hopf algebra is bijective (see Subsection~\ref{subsec:semiperfect-Hopf}), here we prove:

\begin{theorem}
  \label{thm:one-sided-rigidity}
  Let $\mathcal{C}$ be a locally finite abelian category equipped with a structure of a monoidal category such that the tensor product of $\mathcal{C}$ is $\bfk$-linear and exact in each variable, and the unit object $\unitobj$ of $\mathcal{C}$ is a simple object.
  Suppose that $\mathcal{C}$ has either a non-zero projective object or a non-zero injective object.
  Then the following assertions are equivalent:
  \begin{enumerate}
  \item $\mathcal{C}$ is rigid.
  \item $\mathcal{C}$ is left rigid.
  \item $\mathcal{C}$ is right rigid.
  \end{enumerate}
  If these equivalent conditions hold, then $\mathcal{C} \approx \fdmod^Q$ as $\bfk$-linear categories for some QcF coalgebra $Q$.
\end{theorem}

For the case where $\mathcal{C}$ is a finite abelian category, a similar result has been known \cite[Proposition 4.2.10]{MR3242743}.
The proof given in {\em loc.\ cit.} relies on the finiteness of the category $\mathcal{C}$ and thus cannot be applied in the case considered in Theorem~\ref{thm:one-sided-rigidity}. Our proof actually relies on basic properties of the Nakayama functor established in the previous section.

To prove Theorem~\ref{thm:one-sided-rigidity}, we first remark:

\begin{lemma}
  \label{lem:one-sided-rigidity-proof-1}
  Let $\mathcal{C}$ be a $\bfk$-linear abelian category equipped with a structure of a left rigid monoidal category such that the tensor product of $\mathcal{C}$ is $\bfk$-linear and exact in each variable, and let $P$ and $X$ be objects of $\mathcal{C}$. If $P$ is projective, then so are $P \otimes X$ and $P^{\vee}$.
\end{lemma}
\begin{proof}
  Let $P$ be a projective object of $\mathcal{C}$. Then $P \otimes X$ is projective for all objects $X \in \mathcal{C}$ \cite[Proposition 4.2.12]{MR3242743}. Although the rigidity is assumed in \cite[Proposition 6.1.3]{MR3242743}, one can prove that $P^{\vee}$ is projective even in our setting in the same way as that proposition.
\end{proof}

\begin{lemma}
  \label{lem:one-sided-rigidity-proof-2}
  For $\mathcal{C}$ as in the statement of Lemma~\ref{lem:one-sided-rigidity-proof-1}, we have $\Inj(\mathcal{C}) = \Proj(\mathcal{C})$.
\end{lemma}
\begin{proof}
  This result is well-known for finite multitensor categories \cite[Remark 6.1.4]{MR3242743}.
We here demonstrate that neither the finiteness nor the right rigidity is necessarily for proving this result.
  Let $P \in \Proj(\mathcal{C})$, and let $0 \to X \to Y \to Z \to 0$ be a short exact sequence in $\mathcal{C}$.
  By applying the functor $P^{\vee} \otimes (-)$, which is exact by our assumption, we obtain an exact sequence $0 \to P^{\vee} \otimes X \to P^{\vee} \otimes Y \to P^{\vee} \otimes Z \to 0$ in $\mathcal{C}$. Since $P^{\vee}$ is projective by Lemma~\ref{lem:one-sided-rigidity-proof-1}, each term of this sequence is projective again by Lemma~\ref{lem:one-sided-rigidity-proof-1}, and thus it splits. Hence the sequence
  \begin{equation*}
    0 \to \Hom_{\mathcal{C}}(P^{\vee} \otimes Z, \unitobj) \to \Hom_{\mathcal{C}}(P^{\vee} \otimes Y, \unitobj) \to \Hom_{\mathcal{C}}(P^{\vee} \otimes X, \unitobj) \to 0
  \end{equation*}
  is exact. By the left rigidity, this sequence is isomorphic to the sequence
  \begin{equation*}
    0 \to \Hom_{\mathcal{C}}(Z, P) \to \Hom_{\mathcal{C}}(Y, P) \to \Hom_{\mathcal{C}}(X, P) \to 0,
  \end{equation*}
  which proves $P \in \Inj(\mathcal{C})$.
  We have verified that $\Proj(\mathcal{C}) \subset \Inj(\mathcal{C})$.
  The converse inclusion is proved by applying the same argument to the left rigid monoidal category $\mathcal{C}^{\op,\rev}$.
\end{proof}

Even if $\mathcal{C}$ is a category satisfying the assumptions of Lemma~\ref{lem:one-sided-rigidity-proof-2}, there may not be a QcF coalgebra $Q$ such that $\mathcal{C} \approx \fdmod^Q$ as $\bfk$-linear categories. Indeed, there are non-QcF Hopf algebras $Q$ such that the equations $\Proj(\fdmod^Q) = \{ 0 \} = \Inj(\fdmod^Q)$ hold (see Example~\ref{ex:counter-example-3}).
If we impose the semiperfectness assumption, then the equation ``$\Proj = \Inj$'' implies being QcF. Precisely, we have:

\begin{lemma}
  \label{lem:one-sided-rigidity-proof-3}
  A semiperfect coalgebra $Q$ satisfying $\Proj(\fdmod^Q) = \Inj(\fdmod^Q)$ is QcF.
\end{lemma}
\begin{proof}
  Let $Q$ be such a coalgebra.
  Then $Q = \bigoplus_{i \in I} E_i$ for some indecomposable injective right $C$-comodules $E_i$.
  Since $Q$ is left semiperfect, each $E_i$ is of finite-dimensional.
  By the assumption that $\Proj(\fdmod^Q) = \Inj(\fdmod^Q)$, each $E_i$ is projective, and thus $Q$ is projective. By \cite[Theorem 3.3.4]{MR1786197}, $Q$ is left QcF.

  Since $V \in {}^Q\fdmod$ is injective (projective) if and only if $V^* \in \fdmod^Q$ is projective (injective), we also have $\Proj({}^Q\fdmod) = \Inj({}^Q\fdmod)$. Thus, by the same argument applied to $Q^{\cop}$, we see that $Q$ is also right QcF. The proof is done.
\end{proof}

\begin{proof}[Proof of Theorem~\ref{thm:one-sided-rigidity}]
  Suppose that $\mathcal{C}$ satisfies the assumptions of this theorem.
  Let $Q$ be a coalgebra such that $\mathcal{C} \approx \fdmod^Q$.
  It is trivial that (1) implies both (2) and (3).
  Below we show that (2) implies (1) and find that $Q$ is QcF during the proof.
  The implication (3) $\Rightarrow$ (1) is proved by applying the same argument to $\mathcal{C}^{\rev}$.

  Suppose that $\mathcal{C}$ is left rigid.
  By Lemma~\ref{lem:one-sided-rigidity-proof-2} and the assumption of this theorem, $\mathcal{C}$ has a non-zero projective object, say $P$.
  Then $\eval_P : P^{\vee} \otimes P \to \unitobj$ is an epimorphism in $\mathcal{C}$ as a non-zero morphism to the simple object $\unitobj$. For $V \in \mathcal{C}$, we consider the morphism $\pi_V = \eval_P \otimes \id_V : P^{\vee} \otimes P \otimes V \to V$.
  Since the tensor product of $\mathcal{C}$ is assumed to be exact, we have $\Img(\pi_V) = \Img(\eval_P) \otimes V = V$, and thus $\pi_V$ is an epimorphism.
  By Lemmas~\ref{lem:one-sided-rigidity-proof-1} and \ref{lem:one-sided-rigidity-proof-2}, $P^{\vee} \otimes P \otimes V$ is projective.
  To sum up, we have proved that $\mathcal{C}$ has enough projective objects.

  Applying the same argument to $\mathcal{C}^{\op,\rev}$ (which is left rigid because $\mathcal{C}$ is), we see that $\mathcal{C}^{\op,\rev}$ has enough projective objects. Thus $\mathcal{C}$ also has enough injective objects.
  Hence, by Lemma~\ref{lem:one-sided-rigidity-proof-3}, the coalgebra $Q$ is QcF.
  By Theorem~\ref{thm:Nakayama-QcF}, the Nakayama functor $\Nakl := \Nakl_{\Ind(\mathcal{C})}$ and $\Nakr := \Nakr_{\Ind(\mathcal{C})}$ are mutually quasi-inverse to each other.
  Set $\modobj := \Nakr(\unitobj)$ and $\overline{\modobj} := \Nakl(\unitobj)$ for simplicity.
  We fix an object $X \in \mathcal{C}$. Then there is a chain of adjunctions
  \begin{equation*}
    X^{\vee\vee\vee} \otimes (-) \dashv X^{\vee\vee} \otimes (-) \dashv X^{\vee} \otimes (-) \dashv X \otimes (-).
  \end{equation*}
  By applying Corollary~\ref{cor:Nakayama-and-adjunctions} (a) to the functor $F = X^{\vee\vee} \otimes (-)$, we have isomorphisms
  \begin{equation}
    \label{eq:one-sided-rigidity-proof-1}
    \Nakl(X) = \Nakl(F^{\rradj}(\unitobj)) \cong F \Nakl(\unitobj) \cong X^{\vee\vee} \otimes \overline{\modobj}.
  \end{equation}
  Similarly, by applying Corollary~\ref{cor:Nakayama-and-adjunctions} (b) to $F = (-) \otimes X^{\vee\vee}$, we also have
  \begin{equation}
    \label{eq:one-sided-rigidity-proof-2}
    \Nakr(X) = \Nakr(F^{\lladj}(\unitobj)) \cong F \Nakr(\unitobj) \cong \modobj \otimes X^{\vee\vee}.
  \end{equation}
  By these isomorphisms, we have
  \begin{equation}
    \label{eq:one-sided-rigidity-proof-3}
    \unitobj \cong \Nakl\Nakr(\unitobj) \cong \modobj^{\vee\vee} \otimes \overline{\modobj},
    \quad
    \unitobj \cong \Nakr\Nakl(\unitobj) \cong \modobj \otimes \overline{\modobj}^{\vee\vee}.
  \end{equation}
  By taking the left duals, we have
  \begin{equation}
    \label{eq:one-sided-rigidity-proof-4}
    \overline{\modobj}^{\vee} \otimes \modobj^{\vee\vee\vee} \cong \unitobj \cong \overline{\modobj}^{\vee\vee\vee} \otimes \modobj^{\vee}.
  \end{equation}
  Since $\Nakl$ and $\Nakr$ are equivalences, and since $\unitobj$ is a simple object by the assumption, both $\modobj$ and $\overline{\modobj}$ are simple objects. We recall that the inequality $l(X \otimes Y) \ge l(X) \cdot l(Y)$ holds for all objects $X, Y \in \mathcal{C}$, where $l(X)$ denotes the length of $X$ (follow \cite[Exercise 4.3.11]{MR3242743}). Applying this inequality to \eqref{eq:one-sided-rigidity-proof-4}, we obtain $l(\modobj^{\vee}) = 1 = l(\overline{\modobj}^{\vee})$. This means that $\modobj^{\vee}$ and $\overline{\modobj}^{\vee}$ are simple objects.
  Furthermore, by \eqref{eq:one-sided-rigidity-proof-3},
  \begin{gather*}
    \Hom_{\mathcal{C}}(\overline{\modobj}, \modobj^{\vee})
    \cong \Hom_{\mathcal{C}}(\modobj^{\vee\vee} \otimes \overline{\modobj}, \unitobj)
    \cong \Hom_{\mathcal{C}}(\unitobj, \unitobj) \ne 0, \\
    \Hom_{\mathcal{C}}(\overline{\modobj}^{\vee}, \modobj)
    \cong \Hom_{\mathcal{C}}(\unitobj, \modobj \otimes \overline{\modobj}^{\vee\vee})
    \cong \Hom_{\mathcal{C}}(\unitobj, \unitobj) \ne 0.
  \end{gather*}
  By Schur's lemma, $\overline{\modobj} \cong \modobj^{\vee}$ and $\overline{\modobj}^{\vee} \cong \modobj$. Now, by~\eqref{eq:one-sided-rigidity-proof-3}, one can easily deduce that $\modobj$ and $\overline{\modobj}$ are invertible objects (follow \cite[Exercise 4.3.11]{MR3242743} again). By \eqref{eq:one-sided-rigidity-proof-1} and that $\Nakl$ is an equivalence, the double dual functor is an equivalence. This means that $\mathcal{C}$ is right rigid.
\end{proof}

\subsection{Frobenius tensor categories}

\begin{definition}
  A {\em Frobenius tensor category} \cite{MR3410615} is a tensor category $\mathcal{C}$ such that every simple object of $\mathcal{C}$ has an injective hull in $\mathcal{C}$.
\end{definition}

Stated differently, a Frobenius tensor category is a tensor category $\mathcal{C}$ such that $\mathcal{C} \approx \fdmod^Q$ for some left semiperfect coalgebra $Q$.
A left semiperfect Hopf algebra is known to be QcF (as we will recall in Subsection~\ref{subsec:semiperfect-Hopf}). Motivated by this fact, we give some characterizations of Frobenius tensor categories as follows:

\begin{theorem}
  \label{thm:Frob-tensor-cat-def}
  Let $\mathcal{C}$ be a tensor category, and let $Q$ be a coalgebra such that $\mathcal{C}$ is equivalent to $\fdmod^Q$ as a $\bfk$-linear category. Then the following are equivalent:
  \begin{enumerate}
  \item $\mathcal{C}$ is a Frobenius tensor category.
  \item $\mathcal{C}$ has a non-zero projective object.
  \item $\mathcal{C}$ has a non-zero injective object.
  \item $Q$ is either left or right semiperfect.
  \item $Q$ is either left or right QcF.
  \item $Q$ is semiperfect.
  \item $Q$ is QcF.
  \end{enumerate}
\end{theorem}
\begin{proof}
  The duality $(-)^{\vee}$ gives an equivalence between $\mathcal{C}$ and $\mathcal{C}^{\op}$.
  Thus $\fdmod^Q$ and $\fdmod^{Q^{\cop}}$ ($\approx (\fdmod^Q)^{\op}$) are equivalent.
  By taking their ind-completions, we find that $Q$ and $Q^{\cop}$ are Morita-Takeuchi equivalent.
  Thus the assertions (1), (4) and (6) are equivalent.
  By the same reason, (5) and (7) are equivalent.
  By Lemma \ref{lem:one-sided-rigidity-proof-2}, (2) and (3) are equivalent.
  It is trivial that (7) implies (6), and that (6) implies (2).
  By Theorem \ref{thm:one-sided-rigidity}, we see that (2) implies (7).
  The proof is done.
\end{proof}

\subsection{A categorical analogue of the modular function}

Let $\mathcal{C}$ be a Frobenius tensor category.
By Theorem \ref{thm:Frob-tensor-cat-def}, $\mathcal{C}$ has enough projectives and enough injectives. Thus, by Theorem \ref{thm:Nakayama-compact-restriction}, the Nakayama functors $\Nakl_{\Ind(\mathcal{C})}$ and $\Nakr_{\Ind(\mathcal{C})}$ restrict to endofunctors $\Nakl_{\mathcal{C}}$ and $\Nakr_{\mathcal{C}}$ on $\mathcal{C}$, respectively. Now we define
\begin{equation}
  \modobj_{\mathcal{C}} := \Nakr_{\mathcal{C}}(\unitobj) \in \mathcal{C}.
\end{equation}
As we will see in Section~\ref{sec:hopf-algebras}, $\modobj_{\mathcal{C}}$ is a categorical analogue of the modular function of a co-Frobenius Hopf algebra.
According to the nomenclature in the Hopf algebra theory, we introduce the following terminology:

\begin{definition}
  We call $\modobj_{\mathcal{C}}$ the {\em modular object} of $\mathcal{C}$.
  We say that $\mathcal{C}$ is {\em unimodular} if $\modobj_{\mathcal{C}} \cong \unitobj$.
\end{definition}

We use the language of module categories over tensor categories (see \cite{MR3242743}).
Given a $\mathcal{C}$-bimodule category $\mathcal{M}$ and two integers $a$ and $b$, we denote by ${}_{(2a)}\mathcal{M}_{(2b)}$ the $\mathcal{C}$-bimodule category obtained from $\mathcal{M}$ by twisting the left and the right $\mathcal{C}$-action by $a$-th and $b$-th power of the double dual functor $(-)^{\vee\vee}$, respectively.
As in the proof of Theorem~\ref{thm:one-sided-rigidity}, we obtain the following theorem by applying Corollary \ref{cor:Nakayama-and-adjunctions} to $F = X^{\vee\vee} \otimes (-)$ or its variants.

\begin{theorem}[{\it cf}. {\cite[Theorems 4.4 and 4.5]{MR4042867}}]
  \label{thm:Frob-tensor-cat-Nakayama}
  Let $\mathcal{C}$ be a Frobenius tensor category.
  Then there are natural isomorphisms
  \begin{align*}
    X^{\vee\vee} \otimes \Nakl_{\mathcal{C}}(Y)
    & \cong \Nakl_{\mathcal{C}}(X \otimes Y)
      \cong \Nakl_{\mathcal{C}}(X) \otimes {}^{\vee\vee} Y, \\
    {}^{\vee\vee}\!X \otimes \Nakr_{\mathcal{C}}(Y)
    & \cong \Nakr_{\mathcal{C}}(X \otimes Y)
      \cong \Nakr_{\mathcal{C}}(X) \otimes Y^{\vee\vee}
  \end{align*}
  for $X, Y \in \mathcal{C}$ making $\Nakl_{\mathcal{C}}$ and $\Nakr_{\mathcal{C}}$ into $\mathcal{C}$-bimodule functors
  \begin{equation*}
    \Nakl_{\mathcal{C}}: \mathcal{C} \to {}_{(2)}\mathcal{C}_{(-2)}
    \quad \text{and} \quad
    \Nakr_{\mathcal{C}}: \mathcal{C} \to {}_{(-2)}\mathcal{C}_{(2)},
  \end{equation*}
  respectively.
\end{theorem}

In particular, there are isomorphisms
\begin{equation}
  \label{eq:Frob-tensor-cat-Naka-formula}
  \overline{\modobj}_{\mathcal{C}} \otimes {}^{\vee\vee}\!X
  \cong \Nakl_{\mathcal{C}}(X)
  \cong X^{\vee\vee} \otimes \overline{\modobj}_{\mathcal{C}},
  \quad
  \modobj_{\mathcal{C}} \otimes X^{\vee\vee}
  \cong \Nakr_{\mathcal{C}}(X)
  \cong {}^{\vee\vee}\!X \otimes \modobj_{\mathcal{C}}
\end{equation}
for $X \in \mathcal{C}$, where $\overline{\modobj}_{\mathcal{C}} = \Nakl_{\mathcal{C}}(\unitobj)$ and $\modobj_{\mathcal{C}} = \Nakr_{\mathcal{C}}(\unitobj)$.
Since the functors $\Nakl_{\mathcal{C}}$ and $\Nakr_{\mathcal{C}}$ are mutually inverse to each other, the modular object $\modobj_{\mathcal{C}}$ is invertible and $\overline{\modobj}_{\mathcal{C}}$ is a dual object of $\modobj_{\mathcal{C}}$ (as we have seen in the proof of Theorem~\ref{thm:one-sided-rigidity}).

Let $\psi_X: {}^{\vee\vee}\!X \otimes \modobj_{\mathcal{C}} \to \Nakr_{\mathcal{C}}(X)$ ($X \in \mathcal{C}$) denote the isomorphism given in Theorem~\ref{thm:Frob-tensor-cat-Nakayama}. By the construction, the morphism $\psi_X$ is the $\unitobj$-component of the natural isomorphism $\phi^{\R}_{F}$ with $F = {}^{\vee\vee}\!X \otimes (-)$ given by Theorem \ref{thm:Nakayama-and-adjunctions}. By tracing the proof of that theorem back, $\psi_X$ is a unique morphism making the diagram
\begin{equation*}
  \begin{tikzcd}[column sep = 80pt]
    {}^{\vee\vee}\!X \otimes \Nakr_{\mathcal{C}}(\unitobj)
    \arrow[ddd, "{\psi_{X}}"']
    & {}^{\vee\vee}\!X \otimes \coHom(V, \unitobj) \otimes V
    \arrow[d, "\tau_{{}^{\vee\vee}\!X, \coHom(V, \unitobj)}^{} \otimes \id"]
    \arrow[l, "{\id_{{}^{\vee\vee}\!X} \otimes i_V(\unitobj)}"'] \\
    & \coHom(V, \unitobj) \otimes {}^{\vee\vee}\!X \otimes V
    \arrow[d, "{\coHom(\eval_{{}^{\vee\vee}\!X} \otimes \id, \unitobj) \otimes \id \otimes \id}"] \\
    & \coHom({}^{\vee}\!X \otimes {}^{\vee\vee}\!X \otimes V, \unitobj) \otimes {}^{\vee\vee}\!X \otimes V
    \arrow[d, "{\text{\eqref{eq:adjunction-coHom} for $F = {}^{\vee\vee}\!X^{\vee} \otimes (-)$}}"] \\
    \Nakr_{\mathcal{C}}(X)
    & \coHom({}^{\vee\vee}\!X \otimes V, X) \otimes {}^{\vee\vee}\!X \otimes V
    \arrow[l, "i_{{}^{\vee\vee}\!X \otimes V}(X)"']
    \end{tikzcd}
\end{equation*}
commute for all $V \in \mathcal{C}$, where $\coHom$ is the coHom functor for $\Ind(\mathcal{C})$,
\begin{equation*}
  i_V(X) : \coHom(V, X) \otimes V \to \Nakr_{\mathcal{C}}(X)
  \quad (V \in \mathcal{C}, X \in \Ind(\mathcal{C}))
\end{equation*}
is the universal dinatural transformation of $\Nakr_{\mathcal{C}}(X)$ and $\tau$ is the natural isomorphism given in Lemma~\ref{lem:NS-Lemma-7-1}. Other isomorphisms of Theorem   \ref{thm:Frob-tensor-cat-Nakayama}
are described in a similar manner.

Now we give immediate consequences of Theorem~\ref{thm:Frob-tensor-cat-Nakayama}.
Cuadra \cite[Theorem 5.2]{MR2236104} gave a formula of the top of the injective hull of a simple comodule over a co-Frobenius Hopf algebra.
This formula has been extended to Frobenius tensor categories admitting a dimension function in \cite[Remark 2.10]{MR3410615}. By Theorems~\ref{thm:Nakayama-permutation} and \ref{thm:Frob-tensor-cat-Nakayama}, we see that the same holds without assuming the existence of a dimension function. Namely, we have:

\begin{corollary}
  \label{cor:ACE15-Rem-2-10}
  Let $\mathcal{C}$ be a Frobenius tensor category.
  Given an object $V$ of $\mathcal{C}$, we denote by $E(V)$ and $P(V)$ the injective hull and the projective cover of $V$, respectively. Then we have
  \begin{equation*}
    E(S) \cong P(\modobj_{\mathcal{C}} \otimes S^{\vee\vee}),
    \quad
    P(S) \cong E(\modobj_{\mathcal{C}}^{\vee} \otimes {}^{\vee\vee}S)
  \end{equation*}
  for all simple objects $S \in \mathcal{C}$. In particular, $\modobj_{\mathcal{C}}$ is the top of $E(\unitobj)$ as well as the dual of the socle of $P(\unitobj)$.
\end{corollary}

By Theorems \ref{thm:Nakayama-symmetric} and \ref{thm:Frob-tensor-cat-Nakayama}, we also have:

\begin{corollary}
  \label{cor:Frob-tensor-cat-symmetricity}
  For a Frobenius tensor category $\mathcal{C}$, the following are equivalent:
  \begin{enumerate}
  \item $\mathcal{C}$ is unimodular and the double dual functor $(-)^{\vee\vee}$ on $\mathcal{C}$ is isomorphic to the identity functor.
  \item $\mathcal{C} \approx \fdmod^Q$ for some symmetric coalgebra $Q$.
  \end{enumerate}
\end{corollary}

For the finite case, this has been proved in \cite{MR4042867}.

\subsection{Categorical Radford $S^4$-formula}

Radford \cite{MR407069} gave a useful formula for the fourth power of the antipode of a finite-dimensional Hopf algebra, which we call {\em the Radford $S^4$-formula}. This formula has been generalized to co-Frobenius Hopf algebras in \cite[Theorem 2.8]{MR2278058}.
The formula also has a generalization to finite tensor categories \cite{MR2097289}.
It may be natural to ask whether there is a categorical analogue of the Radford $S^4$-formula for `infinite' tensor categories covering the result of \cite{MR2278058}.

Recently, it turned out that the categorical Radford $S^4$-formula for finite tensor categories easily follows from a property of the Nakayama functor \cite{MR4042867}.
Now we have the same functor for Frobenius tensor categories.
Thus, by the same way as \cite{MR4042867}, we obtain a categorical Radford $S^4$-formula for Frobenius tensor categories.
Here is the detail:
Given an object $X$ of a rigid monoidal category, we write $X^{\vee 3} = X^{\vee\vee\vee}$, $X^{\vee 4} = X^{\vee\vee\vee\vee}$ and so on.

\begin{definition}
  \label{def:Radford-iso}
  Given a Frobenius tensor category $\mathcal{C}$, we define $\RadIso_X$ by
  \begin{equation*}
    \RadIso_X := \Big(
    X \otimes \modobj_{\mathcal{C}}
    \xrightarrow[\cong]{\quad \eqref{eq:Frob-tensor-cat-Naka-formula} \quad}
    \Nakr_{\mathcal{C}}(X^{\vee\vee})
    \xrightarrow[\cong]{\quad \eqref{eq:Frob-tensor-cat-Naka-formula} \quad}
    \modobj_{\mathcal{C}} \otimes X^{\vee 4}
    \Big)
  \end{equation*}
  for $X \in \mathcal{C}$ and call $\RadIso = \{ \RadIso_X \}_{X \in \mathcal{C}}$ the {\em Radford isomorphism}.
\end{definition}

The isomorphism \eqref{eq:Frob-tensor-cat-Naka-formula} is an instance of the natural isomorphism given in Theorem~\ref{thm:Nakayama-and-adjunctions}.
The coherent property claimed in that theorem ensures that the equation
\begin{equation*}
  \RadIso_{X \otimes Y}
  = (\RadIso_X \otimes \id_{Y^{\vee 4}}) \circ (\id_X \otimes \RadIso_{Y})
\end{equation*}
holds for all objects $X, Y \in \mathcal{C}$.
The following generalization of \cite[Theorem 3.3]{MR2097289} to Frobenius tensor categories is now easily obtained:

\begin{corollary}[the categorical Radford $S^4$-formula]
  \label{cor:categorical-Radford-S4}
  For a Frobenius tensor category $\mathcal{C}$ with modular object $\modobj := \modobj_{\mathcal{C}}$, there is the following isomorphism of monoidal functors:
  \begin{equation*}
    \modobj^{\vee} \otimes X \otimes \modobj
    \xrightarrow{\quad \id_{\modobj^{\vee}} \otimes \RadIso_X \quad}
    \modobj^{\vee} \otimes \modobj \otimes X^{\vee 4}
    \xrightarrow{\quad \eval_{\modobj} \otimes \id_{X^{\vee 4}} \quad}
    X^{\vee 4}
    \quad (X \in \mathcal{C}).
  \end{equation*}
\end{corollary}

We will explain how this result relates to the Radford $S^4$-formula for co-Frobenius Hopf algebras \cite[Theorem 2.8]{MR2278058} in Subsection~\ref{subsec:Hopf-Radford-S4}.

\subsubsection{The braided case}

Let $\mathcal{C}$ be a braided rigid monoidal category with braiding $\sigma$. Then we define
\begin{equation}
  \label{eq:Drinfeld-iso-def}
  \DriIso^{\sigma}_X = (\eval_{X} \otimes \id_{X^{\vee\vee}})
  \circ (\sigma_{X,X^{\vee}} \otimes \id_{X^{\vee\vee}})
  \circ (\id_X \otimes \coev_{X^{\vee}})
  \quad (X \in \mathcal{C})
\end{equation}
and call $\DriIso^{\sigma}$ the {\em Drinfeld isomorphism} associated to $\sigma$. We will use string diagrams to express morphisms in a braided rigid monoidal category. Our convention is that a morphism goes from the top to the bottom of the diagram. The evaluation, the coevaluation, the braiding and its inverse are expressed as follows:
\begin{equation*}
  \tikzset{cross/.style={preaction={-,draw=white,line width=6pt}}}
  \eval_X =
  \begin{tikzpicture}[x = 12pt, y = 10pt, baseline = 0pt]
    \draw (0,1) to [out=-90, in=180] (1,-1) to [out=0, in=-90] (2,1);
    \node at (0,1) [above] {$X^{\vee}$};
    \node at (2,1) [above] {$X$};
  \end{tikzpicture}
  \quad \coev_X =
  \begin{tikzpicture}[x = 12pt, y = 10pt, baseline = 0pt]
    \draw (0,-1) to [out=90, in=180] (1,1) to [out=0, in=90] (2,-1);
    \node at (2,-1) [below] {$X^{\vee}$};
    \node at (0,-1) [below] {$X$};
  \end{tikzpicture}
  \quad \sigma_{X,Y} =
  \begin{tikzpicture}[x = 10pt, y = 10pt, baseline = 0pt]
    \draw (0,1) to [out = -90, in = 90] (2,-1);
    \draw [cross] (2,1) to [out = -90, in = 90] (0,-1);
    \node at (0, 1) [above] {$X$};
    \node at (2, 1) [above] {$Y$};
    \node at (0,-1) [below] {$Y$};
    \node at (2,-1) [below] {$X$};
  \end{tikzpicture}
  \quad
  (\sigma_{X,Y})^{-1} =
  \begin{tikzpicture}[x = 10pt, y = 10pt, baseline = 0pt]
    \draw (2,1) to [out = -90, in = 90] (0,-1);
    \draw [cross] (0,1) to [out = -90, in = 90] (2,-1);
    \node at (0, 1) [above] {$Y$};
    \node at (2, 1) [above] {$X$};
    \node at (0,-1) [below] {$X$};
    \node at (2,-1) [below] {$Y$};
  \end{tikzpicture}
\end{equation*}
The Drinfeld isomorphism and its inverse are given as follows:
\begin{equation}
  \label{eq:Drinfeld-iso-picture}
  \tikzset{cross/.style={preaction={-,draw=white,line width=6pt}}}%
  \DriIso^{\sigma}_X =
  \begin{tikzpicture}[x = 15pt, y = 15pt, baseline = 0pt]
    \coordinate (P1) at (0, 1.5);
    \draw (P1) -- ++(0, -.75)
    to [out=-90, in=90] ++(1, -1.5)
    to [out=-90, in=-90, looseness=1.5] ++(-1, 0) coordinate (SAVE);
    \draw [cross] (SAVE)
    to [out=90, in=-90] ++(1, 1.5)
    to [out=90, in=90, looseness=1.5] ++(1, 0)
    to ++(0,-2.5) coordinate (P2);
    \node at (P1) [above] {$X$};
    \node at (P2) [below] {$X^{\vee\vee}$};
  \end{tikzpicture}
  = \ \ \begin{tikzpicture}[x = 15pt, y = 15pt, baseline = 0pt]
    \coordinate (P1) at (4, 1.5);
    \draw (P1) -- ++(0,-.75)
    to [out=-90, in=90] ++(-1, -1.5)
    to [out=-90, in=-90, looseness=1.5] ++(-1, 0)
    to ++(0, 1.5) coordinate (P3)
    to [out=90, in=90, looseness=1.5] ++(1, 0) coordinate (SAVE);
    \draw [cross] (SAVE)
    to [out=-90, in=90] ++(1, -1.5)
    to [out=-90, in=90] ++(0, -1) coordinate (P2);
    \node at (P1) [above] {$X$};
    \node at (P2) [below] {$X^{\vee\vee}$};
  \end{tikzpicture}
  \quad (\DriIso^{\sigma}_X)^{-1} =
  \begin{tikzpicture}[x = -15pt, y = 15pt, baseline = 0pt]
    \coordinate (P1) at (0, 1.5);
    \draw (P1) -- ++(0, -.75)
    to [out=-90, in=90] ++(1, -1.5)
    to [out=-90, in=-90, looseness=1.5] ++(-1, 0) coordinate (SAVE);
    \draw [cross] (SAVE)
    to [out=90, in=-90] ++(1, 1.5)
    to [out=90, in=90, looseness=1.5] ++(1, 0)
    to ++(0,-2.5) coordinate (P2);
    \node at (P1) [above] {$X^{\vee\vee}$};
    \node at (P2) [below] {$X$};
  \end{tikzpicture}
  = \begin{tikzpicture}[x = -15pt, y = 15pt, baseline = 0pt]
    \coordinate (P1) at (4, 1.5);
    \draw (P1) -- ++(0,-.75)
    to [out=-90, in=90] ++(-1, -1.5)
    to [out=-90, in=-90, looseness=1.5] ++(-1, 0)
    to ++(0, 1.5) coordinate (P3)
    to [out=90, in=90, looseness=1.5] ++(1, 0) coordinate (SAVE);
    \draw [cross] (SAVE)
    to [out=-90, in=90] ++(1, -1.5)
    to [out=-90, in=90] ++(0, -1) coordinate (P2);
    \node at (P1) [above] {$X^{\vee\vee}$};
    \node at (P2) [below] {$X$};
  \end{tikzpicture}
\end{equation}
It is known that the Radford isomorphism of a braided finite tensor category is written in terms of the braiding \cite[Theorem 8.10.7]{MR3242743}. We give a generalization of this result to the setting of braided Frobenius tensor categories as follows:
  
\begin{theorem}
  \label{thm:Radford-iso-braiding}
  For a braided Frobenius tensor category $\mathcal{C}$, we have
  \begin{equation}
    \label{eq:Radford-iso-braiding}
    \RadIso_{X} = (\id_{\modobj} \otimes (\DriIso_{X^{\vee}}^{\vee})^{-1} \DriIso_{X}) \circ (\sigma_{\modobj, X})^{-1}
  \end{equation}
  for all $X \in \mathcal{C}$, where $\sigma$ is the braiding of $\mathcal{C}$, $\modobj = \modobj_{\mathcal{C}}$ and $\DriIso = \DriIso^{\sigma}$.
\end{theorem}
\begin{proof}
  In view of Lemma~\ref{lem:coHom-finite-1}, we assume that the coHom functor of $\Ind(\mathcal{C})$ is given by $\coHom( X, Y ) := \Hom_{\mathcal{C}}(Y, X)^*$ for $X, Y \in \mathcal{C}$. For $V, X \in \mathcal{C}$, we denote by $i_V(X)$ the universal dinatural transformation for $\Nakr(X) := \Nakr_{\Ind(\mathcal{C})}(X)$ as a coend.
  We fix two objects $V, X \in \mathcal{C}$ and consider the commutative diagram given as Figure \ref{fig:proof-Radford-iso-braiding}.
  By \eqref{eq:NS-Lemma-7-1-braided} and the hexagon axiom for a braiding, we have
  \begin{equation}
    \label{eq:Radford-iso-braiding-proof-1}
    (\tau_{X,\coHom( V, \unitobj )} \otimes \id_V) \circ (\sigma_{X,\coHom( V, \unitobj )) \otimes V})^{-1}
    = \id_{\coHom( V, \unitobj )} \otimes \sigma_{V,X}.
  \end{equation}
  The path from $\underline{\qquad}_{\ (1)}$ to $\underline{\qquad}_{\ (2)}$ in the diagram is equal to $\Theta_{V,X}^* \otimes \id_{X} \otimes \id_V$, where $\Theta_{V,X}^*$ is the linear dual of the following linear map:
  \begin{gather*}
    \Theta_{V,X} : \Hom_{\mathcal{C}}(\unitobj, X \otimes V \otimes X^{\vee 3})
    \to \Hom_{\mathcal{C}}(\unitobj, V), \\
    f \mapsto (\eval_X \otimes \id_V \otimes \eval_{X^{\vee}}) \circ (\id_{X^{\vee}} \otimes f \otimes \id_{X^{\vee\vee}}) \circ \coev_{X^{\vee}}.
  \end{gather*}
  Graphically, the map $\Theta_{V,X}$ is given by
  \begin{equation*}
    \tikzset{cross/.style={preaction={-,draw=white,line width=6pt}}}%
    \Theta_{V,X}(f) =
    \begin{tikzpicture}[x = 15pt, y = 15pt, baseline = 0pt]
      \node [rectangle, draw, thick] (B) at (0,0) {\makebox[2em]{$f$}};
      \coordinate (T1) at ([shift={(-.5,0)}]B.south);
      \coordinate (T2) at ([shift={(+.5,0)}]B.south);
      \draw let \n1 = {0.75}, \n2 = {1} in
      (T1) to [out=-90, in=-90, looseness=2] ++(-\n2,0)
      to ++(0, \n1) coordinate (SAVE1)
      to [out=90, in=90, looseness=2] ($(T2)+(\n2, \n1)$) coordinate (SAVE2)
      to ++(0, -\n1) to [out=-90, in=-90, looseness=2] (T2);
      \draw let \p1 = (B.south) in (\p1) -- (\x1, -2) node [below] {$V$};
      \node at ([shift={(0,1)}]SAVE1) [left] {$X^{\vee}$};
      \node at ([shift={(0,1)}]SAVE2) [right] {$X^{\vee\vee}$};
    \end{tikzpicture}
    = \quad \begin{tikzpicture}[x = 15pt, y = 15pt, baseline = 0pt]
      \node [rectangle, draw, thick] (B) at (0,2) {\makebox[2em]{$f$}};
      \coordinate (T1) at ([shift={(-.5,0)}]B.south);
      \coordinate (T2) at (B.south);
      \coordinate (T3) at ([shift={(+.5,0)}]B.south);
      \draw (T2) to [out=-90, in=90] ++(-1,-2) coordinate (SAVE1);
      \draw let \p1 = (SAVE1) in (\p1) -- (\x1, -2) node [below] {$V$};
      \draw (T3) to [out=-90, in=90] ++(-1,-2)
      to [out=-90, in=-90, looseness=2] ++(1, 0) coordinate (SAVE1);
      \draw (T1) to [out=-90, in=90] ++(-1,-2)
      to [out=-90, in=-90, looseness=2] ++(-1, 0)
      to ++(0,2)
      to [out=90, in=90, looseness=2] ++(1,0) coordinate (SAVE2);
      \draw [cross] (SAVE2) to [out=-90, in=90] (SAVE1);
    \end{tikzpicture}
    \ \ \mathop{=}^{\eqref{eq:Drinfeld-iso-picture}} \quad
    \begin{tikzpicture}[x = 15pt, y = 15pt, baseline = 0pt]
      \node [rectangle, draw, thick] (B1) at (0,2) {\makebox[2em]{$f$}};
      \node [rectangle, draw, thick] (B2) at (-1.25,0) {\makebox[2em]{$\DriIso_X\strut$}};
      \node [rectangle, draw, thick] (B3) at ( 1.25,0) {\makebox[2em]{$\DriIso_{X^{\vee}}^{-1}$}};
      \draw ([shift={(-.5,0)}]B1.south) to [out=-90, in=90, looseness=.8] (B2.north);
      \draw ([shift={( .5,0)}]B1.south) to [out=-90, in=90, looseness=.8] (B3.north);
      \draw let \p1 = (B1.south), \p2 = (B2.south) in
      (\p1) -- (\x1, \y2) to [out=-90, in=90] (\x2, -2)
      node [below] {$V$};
      \draw [cross] (B2.south) to [out=-90, in=-90, looseness=1.5] (B3.south);
    \end{tikzpicture}
  \end{equation*}
  for $f \in \Hom_{\mathcal{C}}(\unitobj, X \otimes V \otimes X^{\vee 3})$.
  Thus we have $\Theta_{V,X}^* = \coHom(\unitobj, \xi)$, where
  \begin{equation*}
    \xi := (\id_{V} \otimes \eval_X) \circ (\sigma_{V,X}^{-1} \otimes \id_{X^{\vee}}) \circ (\DriIso_X \otimes \id_V \otimes \DriIso_{X^{\vee}}^{-1})
    : X \otimes V \otimes X^{\vee 3} \to V.
  \end{equation*}
  By~\eqref{eq:Radford-iso-braiding-proof-1} and the dinaturality of $i_{-}(\unitobj) : \coHom(-, \unitobj) \otimes (-) \to \Nakr(\unitobj)$, we compute:
  \begin{align*}
    & \RadIso_X \circ \sigma_{\modobj,X} \circ (i_{V}(\unitobj) \otimes \id_{X}) \\
    & = \begin{aligned}[t](i_{X \otimes V \otimes X^{\vee\vee}}(\unitobj) \otimes \id_{X^{\vee 4}})
      & \circ (\coHom(\xi, \unitobj) \otimes \id_{X} \otimes \id_V \otimes \coev_{X^{\vee 3}}) \\
      & \circ (\id_{\coHom( V, \unitobj )} \otimes \sigma_{V, X}) \end{aligned} \\
    & = (i_V(\unitobj) \otimes \id_{X})
      \circ (\id_{\coHom( V, \unitobj )} \otimes
      ((\xi \otimes \id_{X^{\vee 4}})
      \circ (\id_{X} \otimes \id_{V} \otimes \coev_{X^{\vee 3}})
      \circ \sigma_{V, X}) \\
    & = (i_V(\unitobj) \otimes \id_{X})
      \circ (\id_{\coHom( V, \unitobj )} \otimes (\DriIso_{X^{\vee}}^{-1})^{\vee} \DriIso_X).
  \end{align*}
  Thus, by the universal property, we have $\RadIso_X \circ \sigma_{\modobj,X} = \id_{\modobj} \otimes (\DriIso_{X^{\vee}}^{-1})^{\vee} \DriIso_{X}$. The proof is done.
\end{proof}

\begin{figure}
  \begin{equation*}
    \begin{tikzcd}[column sep = 90pt]
      \modobj \otimes X
      \arrow[d, "{\sigma_{\modobj, X}}"']
      \arrow[dr, phantom, "\scriptstyle \text{(the naturality of the braiding)}"]
      & \coHom( V, \unitobj ) \otimes V \otimes X
      \arrow[d, "{\sigma_{\coHom( V, \unitobj ) \otimes V, X}^{}}"]
      \arrow[l, "i_V(\unitobj) \otimes \id_{X}"'] \\
      X \otimes \modobj
      \arrow[ddd, "\eqref{eq:Frob-tensor-cat-Naka-formula}"']
      & X \otimes \coHom( V, \unitobj ) \otimes V
      \arrow[d, "\tau_{X, \coHom( V, \unitobj )}^{} \otimes \id"]
      \arrow[l, "{\id_{X} \otimes i_V(\unitobj)}"] \\
      & \underline{\coHom( V, \unitobj ) \otimes X \otimes V}_{\ (1)}
      \arrow[d, "{\coHom( \eval_{X} \otimes \id, \unitobj ) \otimes \id \otimes \id}"] \\
      & \coHom( X^{\vee} \otimes X \otimes V, \unitobj ) \otimes X \otimes V
      \arrow[d, "{\text{\eqref{eq:adjunction-coHom} for $F = X^{\vee} \otimes (-)$}}"] \\
      \Nakr(X^{\vee\vee})
      \arrow[dd, "\eqref{eq:Frob-tensor-cat-Naka-formula}"']
      & \coHom( X \otimes V, X^{\vee\vee} ) \otimes X \otimes V
      \arrow[d, "{\text{\eqref{eq:adjunction-coHom} for $F = (-) \otimes X^{\vee 3}$}}"]
      \arrow[l, "i_{X \otimes V}(X^{\vee\vee})"'] \\
      & \underline{\coHom( X \otimes V \otimes X^{\vee 3}, \unitobj ) \otimes X \otimes V}_{\ (2)}
      \arrow[d, "{\id \otimes \id \otimes \id \otimes \coev_{X^{\vee 3}}}"] \\
      \modobj \otimes X^{\vee 4}
      & \coHom( X \otimes V \otimes X^{\vee 3}, \unitobj ) \otimes X \otimes V \otimes X^{\vee 3} \otimes X^{\vee 4}
      \arrow[l, "i_{X \otimes V \otimes X^{\vee 3}}(\unitobj) \otimes \id_{X^{\vee 4}}"']
    \end{tikzcd}
  \end{equation*}
  \caption{Proof of Theorem \ref{thm:Radford-iso-braiding}}
  \label{fig:proof-Radford-iso-braiding}
\end{figure}

An object $T$ of a braided monoidal category $\mathcal{C}$ is said to be {\em transparent} if $\sigma_{X,T} \circ \sigma_{T,X} = \id_{T} \otimes \id_{X}$ for all objects $X \in \mathcal{C}$, where $\sigma$ is the braiding of $\mathcal{C}$.
It is known that the modular object of a braided finite tensor category is transparent \cite[Corollary 8.10.8]{MR3242743}.
With the help of Theorem~\ref{thm:Radford-iso-braiding}, we can prove the following result by the same way as \cite[Corollary 8.10.8]{MR3242743}.

\begin{corollary}
  \label{cor:Radford-iso-braiding}
  The modular object of a braided Frobenius tensor category is transparent.
\end{corollary}

The M\"uger center of a braided monoidal category $\mathcal{C}$, denoted by $\mathcal{C}'$, is the full subcategory of $\mathcal{C}$ consisting of all transparent objects of $\mathcal{C}$. If $\mathcal{C}$ is a braided finite tensor category with the trivial M\"uger center ({\it i.e.}, $\mathcal{C}' = \mathcal{C}_{\triv}$), then it is unimodular by \cite[Proposition 4.5]{MR2097289} and \cite{MR3996323}. Corollary~\ref{cor:Radford-iso-braiding} implies that the same conclusion holds in the infinite setting as follows:

\begin{corollary}
  A braided Frobenius tensor category with the trivial M\"uger center is unimodular.
\end{corollary}

\subsubsection{The semisimple case}

The fact that a cosemisimple coalgebra is a symmetric coalgebra \cite{MR2078404} implies that a semisimple tensor category is a Frobenius tensor category whose Nakayama functor is isomorphic to the identity.
Thus, by Corollaries~\ref{cor:Frob-tensor-cat-symmetricity} and \ref{cor:categorical-Radford-S4}, we have:

\begin{corollary}
  \label{cor:Radford-S4-semisimple}
  If $\mathcal{C}$ is a semisimple tensor category, then the following hold:
  \begin{enumerate}
  \item [(a)] $\mathcal{C}$ is unimodular.
  \item [(b)] The double dual functor $(-)^{\vee\vee}$ is isomorphic to $\id_{\mathcal{C}}$ as a mere functor.
  \item [(c)] The quadruple dual functor $(-)^{\vee 4}$ is isomorphic to $\id_{\mathcal{C}}$ as a monoidal functor.
  \end{enumerate}
\end{corollary}

Specifically speaking, the isomorphism $\id_{\mathcal{C}} \to (-)^{\vee 4}$ of Part (c) of Corollary \ref{cor:Radford-S4-semisimple} is given by the natural isomorphism $\tilde{\RadIso}$ characterized by the following equation:
\begin{equation}
  \label{eq:Radford-isomorphism-tilde}
  \modobj_{\mathcal{C}} \otimes \tilde{\RadIso}_X
  = \RadIso_{X} \circ (\tau_{X, \modobj_{\mathcal{C}}})^{-1}
  \quad (X \in \mathcal{C}).
\end{equation}
In this subsection, we show that the isomorphism $\tilde{\RadIso}_X$ is characterized by a certain trace condition as in \cite[Theorem 7.3]{MR2097289}. To state the result, we introduce the following notation: Given a morphism $f : X \to X^{\vee\vee}$ in $\mathcal{C}$, we define its trace $\trace(f) \in \bfk$ by
\begin{equation*}
  \trace(f) \cdot \id_{\unitobj}
  = \eval_{X^{\vee}} \circ (f \otimes \id_{X^{\vee}}) \circ \coev_X.
\end{equation*}

\begin{theorem}
  \label{thm:Radford-iso-semisimple}
  Let $\mathcal{C}$ be a semisimple tensor category such that every simple object of $\mathcal{C}$ is absolutely simple. Then, for every simple object $X \in \mathcal{C}$, we have
  \begin{equation}
    \label{eq:Radford-iso-semisimple}
    \tilde{\RadIso}_X = \frac{\trace((\phi^{-1})^{\vee})}{\trace(\phi)} \cdot \phi^{\vee\vee} \circ \phi,
  \end{equation}
  where $\phi : X \to X^{\vee\vee}$ is an arbitrary isomorphism in $\mathcal{C}$. Thus,
  \begin{equation}
    \label{eq:Radford-iso-semisimple-2}
    \trace((\phi^{-1})^{\vee\vee} \circ \tilde{\RadIso}_X) = \trace((\phi^{-1})^{\vee}).
  \end{equation}
\end{theorem}
\begin{proof}
  We note that $\trace(\phi) \ne 0$ \cite[Proposition 4.8.4]{MR3242743}.
  In view of Lemma~\ref{lem:coHom-finite-1}, we assume that the coHom functor of $\Ind(\mathcal{C})$ is given by $\coHom( X, Y ) := \Hom_{\mathcal{C}}(Y, X)^*$ for $X, Y \in \mathcal{C}$.
  We realize the coend $\Nakr(M) := \Nakr_{\Ind(\mathcal{C})}(M)$ for $M \in \mathcal{C}$ as follows:
  We fix a set $\{ L_i \}_{i \in I}$ of complete representatives of the isomorphism classes of simple objects of $\mathcal{C}$.
  By the semisimplicity of $\mathcal{C}$, for each object $X \in \mathcal{C}$, there are a finite set $A(X)$, a map $\lambda : A(X) \to I$ and a family
  \begin{equation*}
    \{ p^X_{a} \in \Hom_{\mathcal{C}}(X, L_{\lambda(a)}),
    \ q^X_{a} \in \Hom_{\mathcal{C}}(L_{\lambda(a)}, X) \}_{a \in A(X)}
  \end{equation*}
  of morphisms in $\mathcal{C}$ such that the equations $\sum_{a \in A(X)} q^X_{a} \circ p^X_{a} = \id_X$ and $p^X_{a} \circ q^X_{b} = \delta_{a,b} \, \id_{L_{\lambda(a)}}$ hold for all $a, b \in A(X)$. The coend $\Nakr(M)$ is now given by
  \begin{equation}
    \label{eq:Radford-iso-semisimple-proof-0}
    \Nakr(M) = \bigoplus_{i \in I} \coHom( L_i, M ) \otimes L_i
  \end{equation}
  with the universal dinatural transformation
  \begin{equation}
    \label{eq:Radford-iso-semisimple-proof-1}
    i_X(M) = \sum_{a \in A(X)} \coHom( q^X_a, M ) \otimes p^X_a:
    \coHom( X, M ) \otimes X \to \Nakr(M)
  \end{equation}
  for $X \in \mathcal{C}$, where the inclusion morphism $\coHom( L_i, M ) \otimes L_i \hookrightarrow \Nakr(M)$ ($i \in I$) is omitted.

  We compute the Radford isomorphism by using the formula \ref{eq:Radford-iso-semisimple-proof-0} of the coend $\Nakr(M)$. We fix a simple object $X$ of $\mathcal{C}$ and an isomorphism $\phi : X \to X^{\vee\vee}$ in $\mathcal{C}$. We set $Y = X \otimes X^{\vee 3}$. Since
  \begin{equation*}
    \Hom_{\mathcal{C}}(\unitobj, Y)
    = \Hom_{\mathcal{C}}(\unitobj, X \otimes X^{\vee 3})
    \cong \Hom_{\mathcal{C}}(X^{\vee\vee}, X)
    \cong \Hom_{\mathcal{C}}(X, X)
    \cong \bfk,
  \end{equation*}
  we may assume that there is a special element $0 \in A(Y)$ such that
  \begin{equation*}
    L_{\lambda(0)} = \unitobj,
    \quad p^Y_{0} = \eval_{X^{\vee 3}} (\phi^{\vee\vee}\phi \otimes \id),
    \quad q^Y_0 = \trace(\phi)^{-1} \cdot (\phi^{-1} \otimes \id_{X^{\vee 3}}) \coev_{X^{\vee\vee}}
  \end{equation*}
  and $L_{\lambda(a)} \not \cong \unitobj$ for all $a \in A(Y) \setminus \{ 0 \}$. Then we have
  \begin{equation}
    \label{eq:Radford-iso-semisimple-proof-2}
    i_Y(\unitobj)
    \mathop{=}^{\eqref{eq:Radford-iso-semisimple-proof-1}} \sum_{a \in A(Y)} \coHom( q^Y_a, \unitobj ) \otimes p^Y_a
    = \coHom( q^Y_0, \unitobj ) \otimes p^Y_0.
  \end{equation}
  Now we consider the diagram given as Figure~\ref{fig:proof-Radford-iso-semisimple}. The path from $\underline{\qquad}_{\ (1)}$ to $\underline{\qquad}_{\ (2)}$ in the diagram is equal to the morphism $\Theta_{\unitobj,X}^* \otimes \id_X$ (see the proof of Theorem~\ref{thm:Radford-iso-braiding} for the definition of $\Theta_{\unitobj, X}$). Note:
  \begin{align*}
    \trace(\phi) \cdot \Theta_{\unitobj,X}(q_0^Y)
    & = (\eval_X \otimes \eval_{X^{\vee\vee}}) \circ (\id_{X^{\vee}} \otimes q^Y_0 \otimes \id_{X^{\vee\vee}}) \circ \coev_{X^{\vee}} \\
    & = \trace((\phi^{-1})^{\vee}) \cdot \id_{\unitobj}.
  \end{align*}
  Thus, for $\xi \in \coHom( \unitobj, \unitobj )$, we have
  \begin{equation}
    \label{eq:Radford-iso-semisimple-proof-3}
    \langle \coHom( q_0^Y, \unitobj ) \Theta_{\unitobj,X}^*(\xi), \id_{\unitobj} \rangle
    = \langle \xi, \Theta_{\unitobj,X}(q_0^Y) \rangle
    = \frac{\trace((\phi^{-1})^{\vee})}{\trace(\phi)} \langle \xi, \id_{\unitobj} \rangle.
  \end{equation}
  This shows $\displaystyle \coHom( q_0^Y, \unitobj ) \Theta_{\unitobj,X}^* = \trace((\phi^{-1})^{\vee})\trace(\phi)^{-1} \cdot \id_{\coHom( \unitobj, \unitobj) \otimes \unitobj}$.
    By \eqref{eq:Radford-iso-semisimple-proof-0} and \eqref{eq:Radford-iso-semisimple-proof-1}, we have $\modobj_{\mathcal{C}} = \coHom( \unitobj, \unitobj ) \otimes \unitobj$ and $i_{\unitobj}(\unitobj) = \id_{\coHom( \unitobj, \unitobj ) \otimes \unitobj}$.
  Thus, by \eqref{eq:Radford-iso-semisimple-proof-2} and \eqref{eq:Radford-iso-semisimple-proof-3}, we compute
  \begin{align*}
    & \RadIso_X \circ (\id_X \otimes i_{\unitobj}(\unitobj)) \\
    & = (i_Y(\unitobj) \otimes \id_{X^{\vee 4}})
    \circ (\id_{\coHom( Y, \unitobj )} \otimes \coev_{X^{\vee 3}}) \circ (\Theta_{\unitobj,X}^* \otimes \id_X) \circ \tau_{X, \coHom( \unitobj, \unitobj ) \otimes \unitobj} \\
    & = (\coHom( q_0^Y, \unitobj ) \Theta_{\unitobj,X}^* \otimes \phi^{\vee\vee} \phi) \circ \tau_{X, \coHom( \unitobj, \unitobj ) \otimes \unitobj} \\
    & = \trace((\phi^{-1})^{\vee}) \trace(\phi)^{-1} \cdot (\id_{\coHom( \unitobj,\unitobj ) \otimes \unitobj} \otimes \phi^{\vee\vee} \phi)
      \circ \tau_{X, \coHom( \unitobj, \unitobj ) \otimes \unitobj} \\
    & = \trace((\phi^{-1})^{\vee}) \trace(\phi)^{-1} \cdot (i_{\unitobj}(\unitobj) \otimes \phi^{\vee\vee} \phi)
      \circ \tau_{X, \modobj_{\mathcal{C}}}.
  \end{align*}
  Now \eqref{eq:Radford-iso-semisimple} is verified by the defining formula \eqref{eq:Radford-isomorphism-tilde} of $\tilde\RadIso_X$. The proof is done.
\end{proof}

\begin{figure}
  \centering
  \begin{equation*}
    \begin{tikzcd}[column sep = 90pt]
      X \otimes \modobj
      \arrow[ddd, "\eqref{eq:Frob-tensor-cat-Naka-formula}"']
      & X \otimes \coHom( \unitobj, \unitobj ) \otimes \unitobj
      \arrow[d, "\tau_{X, \coHom( \unitobj, \unitobj )}^{} \otimes \id_{\unitobj}"]
      \arrow[l, "{\id_{X} \otimes i_\unitobj(\unitobj)}"'] \\
      & \underline{\coHom( \unitobj, \unitobj ) \otimes X \otimes \unitobj}_{\ (1)}
      \arrow[d, "{\coHom( \eval_{X} \otimes \id, \unitobj ) \otimes \id \otimes \id}"] \\
      & \coHom( X^{\vee} \otimes X \otimes \unitobj, \unitobj ) \otimes X \otimes \unitobj
      \arrow[d, "{\text{\eqref{eq:adjunction-coHom} for $F = X^{\vee} \otimes (-)$}}"] \\
      \Nakr(X^{\vee\vee})
      \arrow[dd, "\eqref{eq:Frob-tensor-cat-Naka-formula}"']
      & \coHom( X, X^{\vee\vee} ) \otimes X \otimes \unitobj
      \arrow[d, "{\text{\eqref{eq:adjunction-coHom} for $F = (-) \otimes X^{\vee 3}$}}"]
      \arrow[l, "i_{X \otimes \unitobj}(X^{\vee\vee})"'] \\
      & \underline{\coHom( X \otimes X^{\vee 3}, \unitobj ) \otimes X \otimes \unitobj}_{\ (2)}
      \arrow[d, "{\id \otimes \id \otimes \coev_{X^{\vee 3}}}"] \\
      \modobj \otimes X^{\vee 4}
      & \coHom( X \otimes X^{\vee 3}, \unitobj ) \otimes X \otimes X^{\vee 3} \otimes X^{\vee 4}
      \arrow[l, "i_{X \otimes X^{\vee 3}}(\unitobj) \otimes \id_{X^{\vee 4}}"']
    \end{tikzcd}
  \end{equation*}
  \caption{Proof of Theorem~\ref{thm:Radford-iso-semisimple}}
  \label{fig:proof-Radford-iso-semisimple}
\end{figure}

\subsection{Spherical tensor categories}

A {\em pivotal structure} of a rigid monoidal category $\mathcal{C}$ is an isomorphism $\id_{\mathcal{C}} \to (-)^{\vee\vee}$ of monoidal functors, where $(-)^{\vee\vee}$ is the double left dual of $\mathcal{C}$.
The sphericity of a pivotal structure was first introduced in \cite{MR1686423} in relation with Turaev-Viro invariants of 3-manifolds.
Later, a modified definition of a spherical pivotal structure for a finite tensor category was introduced in \cite{MR4254952} in view of applications to topological field theories arising from non-semisimple categories. Here, we extend the definition of \cite{MR4254952} to unimodular Frobenius tensor categories as follows:

\begin{definition}
  Let $\mathcal{C}$ be a unimodular Frobenius tensor category, and let $f : \unitobj \to \modobj_{\mathcal{C}}$ be an isomorphism in $\mathcal{C}$. We say that a pivotal structure $\PivIso$ of $\mathcal{C}$ is {\em spherical} if the diagram
  \begin{equation*}
    \begin{tikzcd}[column sep = 80pt]
      X
      \arrow[d, "\id \otimes f"']
      \arrow[r, "\PivIso_{X}"]
      & X^{\vee\vee} \arrow[r, "\PivIso_{X^{\vee\vee}}"]
      & X^{\vee 4}
      \arrow[d, "f \otimes \id"]\\
      X \otimes \modobj_{\mathcal{C}}
      \arrow[rr, "{\RadIso_{X}}"]
      & & \modobj_{\mathcal{C}} \otimes X^{\vee 4}
    \end{tikzcd}
  \end{equation*}
  commutes for all objects $X \in \mathcal{C}$ (since $\Hom_{\mathcal{C}}(\unitobj, \modobj_{\mathcal{C}})$ is one-dimensional by Schur's lemma, this definition does not depend on the choice of $f$).
  A {\em spherical tensor category} is a unimodular Frobenius tensor category equipped with a spherical pivotal structure.
\end{definition}

Let $\mathcal{C}$ be a pivotal monoidal category with pivotal structure $\PivIso$. Then the {\em left pivotal trace} $\ptrace^{\L}(f)$ and the {\em right pivotal trace} $\ptrace^{\R}(f)$ of an endomorphism $f: X \to X$ in $\mathcal{C}$ are defined by
\begin{equation*}
  \ptrace^{\L}(f) = \trace(\PivIso_{X^{\vee}} \circ f^{\vee})
  \quad \text{and} \quad
  \ptrace^{\R}(f) = \trace(\PivIso_X \circ f),
\end{equation*}
respectively. According to \cite{MR4254952}, we say that $\PivIso$ is {\em trace spherical} if $\ptrace^{\L}(f) = \ptrace^{\R}(f)$ for all endomorphism $f$ in $\mathcal{C}$ (or, equivalently, $\PivIso$ is spherical in the sense of \cite{MR1686423}).

As noted in \cite[Example 3.5.5]{MR4254952}, there are no logical implications between the sphericity and the trace sphericity in the general setting. However, for the case of a fusion category ($=$ a semisimple finite tensor category), these two conditions are equivalent \cite[Proposition 3.5.4]{MR4254952}.
The proof given there relies on the fact that the equation $\trace(f^{\vee}) = \trace(f^{\vee\vee} \circ \tilde{\RadIso}_X)$ holds for all morphism $f: X^{\vee\vee} \to X$ in $\mathcal{C}$, where $\tilde{\RadIso}$ is the natural isomorphism given by \eqref{eq:Radford-isomorphism-tilde}.
Since we now know that the same equation holds without the finiteness of $\mathcal{C}$ by Theorem~\ref{thm:Radford-iso-semisimple}, we obtain the following theorem:

\begin{theorem}
  A pivotal structure of a semisimple tensor category is spherical if and only if it is trace spherical.
\end{theorem}

A {\em ribbon category} is a braided rigid monoidal category $\mathcal{C}$ equipped with a {\em twist}, that is, a natural isomorphism $\theta : \id_{\mathcal{C}} \to \id_{\mathcal{C}}$ satisfying
\begin{gather}
  \label{eq:twist-def-1}
  \theta_{X \otimes Y} = \sigma_{Y,X} \circ \sigma_{X,Y} \circ (\theta_X \otimes \theta_Y), \\
  \label{eq:twist-def-2}
  \theta_{X^{\vee}} = (\theta_X)^{\vee}
\end{gather}
for all objects $X, Y \in \mathcal{C}$, where $\sigma$ is the braiding of $\mathcal{C}$. A braided spherical tensor category is naturally a ribbon category. Indeed, we have:

\begin{theorem}
  \label{thm:spherical-ribbon}
  Let $\mathcal{C}$ be a unimodular braided Frobenius tensor category.
  For a pivotal structure $\PivIso$ of $\mathcal{C}$, the following are equivalent:
  \begin{enumerate}
  \item $\PivIso$ is spherical.
  \item $\theta = \DriIso^{-1} \circ \PivIso$ is a twist, where $\DriIso$ is the Drinfeld isomorphism~\eqref{eq:Drinfeld-iso-def}.
  \end{enumerate}
\end{theorem}
\begin{proof}
  The assignment $\PivIso \mapsto \DriIso^{-1} \circ \PivIso$ gives a bijective correspondence between the set of pivotal structures of $\mathcal{C}$ and the set of natural transformations $\theta : \id_{\mathcal{C}} \to \id_{\mathcal{C}}$ satisfying \eqref{eq:twist-def-1}. Now let $\PivIso$ be a pivotal structure of $\mathcal{C}$. Then $\theta := \DriIso^{-1} \circ \PivIso$ is a twist ({\it i.e.}, it satisfies \eqref{eq:twist-def-2}) if and only if the equation
  \begin{equation*}
    \PivIso_{X^{\vee\vee}} \circ \PivIso_{X} = (\DriIso_{X^{\vee}}^{-1})^{\vee} \circ \DriIso_{X}
  \end{equation*}
  holds for all objects $X \in \mathcal{C}$ \cite[Lemma 5.2]{2017arXiv170709691S}. Thus, by Theorem~\ref{thm:Radford-iso-braiding}, the pivotal structure $\PivIso$ is spherical if and only if $\theta$ is a twist. The proof is done.
\end{proof}

\section{Applications to Hopf algebras}
\label{sec:hopf-algebras}

\subsection{Semiperfectness of Hopf algebras}
\label{subsec:semiperfect-Hopf}

In this section, we apply our results to Hopf algebras.
No new significant results on Hopf algebras will be given.
The main purpose of this section is, rather, to give new insights on the theory of Hopf algebras by explaining how our categorical results can be applied to Hopf algebras.

In what follows, the unadorned symbol $\otimes$ means the tensor product over $\bfk$.
Let $H$ be a Hopf algebra with comultiplication $\Delta$, counit $\varepsilon$ and antipode $S$. Then the category $\fdmod^H$ is a left rigid monoidal category. We recall that a left dual object of $X \in \fdmod^H$, denoted by $X^{\vee}$ as in the last section, is the vector space $X^*$ equipped with the right $H$-coaction characterized by
$\langle x^*_{(0)}, x \rangle x^*_{(1)} = \langle x^*, x_{(0)} \rangle S(x_{(1)})$
for $x \in X$ and $x^* \in X^*$. If $S$ is bijective, then a right dual object of $X \in \fdmod^H$ is given by the same way as $X^{\vee}$ but by using $S^{-1}$ instead of $S$. It is known that $\fdmod^H$ is rigid if and only if $S$ is bijective \cite{MR1098991}.
Now we apply our results to obtain the following theorem, which is a part of \cite[Theorem 5.3.2]{MR1786197}.

\begin{theorem}
  \label{thm:Hopf-semiperfect}
  For a Hopf algebra $H$, the following assertions are equivalent:
  \begin{enumerate}
  \item $H$ is co-Frobenius (as a coalgebra; the same applies below).
  \item $H$ is QcF.
  \item $H$ is semiperfect.
  \item $H$ is either left or right QcF.
  \item $H$ is either left or right semiperfect.
  \end{enumerate}
  If these equivalent assertions are satisfied, then the antipode $S$ of $H$ is bijective.
\end{theorem}
\begin{proof}
  For simplicity, we write $\mathcal{C} = \fdmod^H$.
  By Theorems~\ref{thm:one-sided-rigidity} and \ref{thm:Frob-tensor-cat-def}, (2), (3), (4) and (5) are equivalent and these equivalent conditions imply that $\mathcal{C}$ is rigid, or, equivalently, $S$ is bijective. It is trivial that (1) implies (2). Thus it remains to show that (2) implies (1). Suppose that (2) holds.
  The modular object $\modobj_{\mathcal{C}}$ is of dimension one (as a vector space) since it is an invertible object of $\mathcal{C}$.
  By Theorem~\ref{thm:Frob-tensor-cat-Nakayama}, we have
  \begin{equation*}
    \dim_{\bfk}(\Nakr_{\mathcal{C}}(X))
    = \dim_{\bfk}(\modobj_{\mathcal{C}}) \cdot \dim_{\bfk}(X^{**})
    = \dim_{\bfk}(X)
  \end{equation*}
  for all $X \in \mathcal{C}$. Thus, by Theorem~\ref{thm:QcF-to-be-co-Frob}, $H$ is co-Frobenius. The proof is done.
\end{proof}

\subsection{Frobenius pairings on co-Frobenius Hopf algebras}

In the previous subsection, we have given equivalent conditions for a Hopf algebra to be co-Frobenius.
Here we prove that a non-zero cointegral on a co-Frobenius Hopf algebra exists and it gives a Frobenius pairing on that Hopf algebra.
This fact has been known (see, {\it e.g.}, the proof of \cite[Theorem 5.3.2]{MR1786197}), however, we include a proof since the discussion below gives a helpful viewpoint for later subsections.

Let $H$ be a Hopf algebra.
Given an $H$-bicomodule algebra $A$ and an $H$-bimodule coalgebra $C$, we introduce two categories ${}^C_A \Mod^H_H$ and ${}^C_A \YD$ as follows:
An object of the former category is a $C$-$H$-bicomodule $M$ equipped with an $A$-$H$-bimodule structure such that the equations
\begin{align*}
  \delta^{\L}_M(a \cdot m \cdot h)
  & = (a_{(-1)} \actl m_{(-1)} \actr h_{(1)}) \otimes a_{(0)} m_{(0)} h_{(2)}, \\
  \delta^{\R}_M(a \cdot m \cdot h)
  & = a_{(0)} m_{(0)} h_{(1)} \otimes (a_{(1)} \actl m_{(1)} \actr h_{(2)})
\end{align*}
hold for all $a \in A$, $m \in M$ and $h \in H$, where $\actl$ and $\actr$ denote the left and the right action of $H$ on $C$, respectively.

An object of the latter one, ${}^C_A \YD$, is a left $C$-comodule $V$ equipped with a left $A$-module structure such that the Yetter-Drinfeld (YD) condition
\begin{equation}
  \label{eq:YD-condition}
  (a_{(-1)} \actl v_{(-1)}) \otimes a_{(0)} v_{(0)}
  = ((a_{(0)} v)_{(-1)} \actr a_{(1)}) \otimes (a_{(0)} v)_{(0)}
\end{equation}
is satisfied for all elements $a \in A$ and $v \in V$.
We remark that the YD condition is equivalent to that the equation
\begin{equation}
  \label{eq:YD-condition-2}
  \delta_V(a v) = (a_{(-1)} \actl v_{(-1)} \actr S(a_{(1)})) \otimes a_{(0)} v_{(0)}
\end{equation}
holds for all $a \in A$ and $v \in V$.
In both categories, a morphism is a linear map respecting all the actions and coactions.

\begin{lemma}
  \label{lem:Hopf-modules}
  The categories ${}^C_A\YD$ and ${}^C_A \Mod^H_H$ are equivalent.  
\end{lemma}
\begin{proof}
  This is only a slight variant of \cite[Theorem 3.1 (i)]{MR1629389}.
  For later use, we include how to establish an equivalence between these categories.
  Given an object $V \in {}^C_A \YD$, we define $F(V) = V \otimes H$ as a vector space and make it an object of ${}^C_A \Mod^H_H$ by the following structure maps:
  \begin{gather*}
    \delta^{\L}_{F(V)}(v \otimes h) = (v_{(-1)} \actr h_{(1)}) \otimes v_{(0)} \otimes h_{(2)}, \quad 
    \delta^{\R}_{F(V)} = \id_V \otimes \Delta, \\
    a \cdot (v \otimes h) \cdot h' = a_{(0)} v \otimes a_{(1)} h h'
    \quad (v \in V, h, h' \in H, a \in A).
  \end{gather*}
  Given an object $M \in {}^C_A \Mod^H_H$, we set $I(M) = \{ m \in M \mid \delta^{\R}_M(m) = m \otimes 1_H \}$. Then $I(M)$ is a subcomodule of the left $C$-comodule $M$. Furthermore, it becomes an object of ${}^C_A \YD$ together with the left $A$-action given by
  \begin{equation}
    \label{eq:Hopf-modules-invariant-action}
    a \triangleright m = a_{(0)} m S(a_{(1)})
    \quad (a \in A, m \in I(M)).
  \end{equation}
  The assignments $V \mapsto F(V)$ and $M \mapsto I(M)$ give rise to equivalences between ${}^C_A \YD$ and ${}^C_A \Mod^H_H$. Indeed, by the same way as \cite{MR1629389}, one can verify that the following are natural isomorphisms:
  \begin{gather*}
    F I(M) \to M,
    \quad m \otimes h \mapsto m h
    \quad (m \in M \in {}^C_A \Mod^H_H, h \in H), \\
    V \to I F(V),
    \quad v \mapsto v \otimes 1_H
    \quad (v \in V \in {}^C_A\YD). \qedhere
  \end{gather*}
\end{proof}

We suppose that $H$ satisfies the equivalent conditions of Theorem \ref{thm:Hopf-semiperfect}.
Then, by that theorem, the antipode of $H$ is bijective.
We define $\tilde{H}$ to be the $H$-bimodule coalgebra whose underlying coalgebra is same as that of $H$ but whose $H$-bimodule structure is given by
\begin{equation*}
h \actl x \actr h' = S^{-2}(h) x S^2(h')
\quad (x \in \tilde{H}, h, h' \in H).
\end{equation*}
We suppose that $H$ satisfies the equivalent conditions of Theorem \ref{thm:Hopf-semiperfect}.
By the semiperfectness of $H$, the subspace $H^{*\rat} \subset H^*$ is rational both as a left and a right $H^*$-module, and thus it has an $H$-bicomodule structure determined by
\begin{equation}
  \label{eq:Hopf-module-H*rat-coactions}
  f_{(0)} \langle \xi, f_{(1)} \rangle = \xi * f
  \quad \text{and} \quad
  \langle \xi, f_{(-1)} \rangle f_{(0)} = f * \xi
\end{equation}
for $f \in H^{*\rat}$ and $\xi \in H^*$.
One can verify that the subspace $H^{*\rat}$ is closed under the actions $\rightharpoonup$ and $\leftharpoonup$ of $H$ on $H^*$ by the same way as \cite[Lemma 2.8]{MR242840}. Now we define the left action $\rightharpoondown$ and the right action $\leftharpoondown$ of $H$ on $H^{*\rat}$ by
\begin{equation}
  \label{eq:Hopf-module-H*rat-actions}
  h \rightharpoondown f \leftharpoondown h' = S(h') \rightharpoonup f \leftharpoonup S^{-1}(h)
\end{equation}
for $h, h' \in H$ and $f \in H^{*\rat}$. It is straightforward to check that the equations
\begin{align*}
  \delta^{\R}_{H^{*\rat}}(h \rightharpoondown f \leftharpoondown h')
  & = (h_{(1)} \rightharpoondown f_{(0)} \leftharpoondown h'_{(1)}) \otimes h_{(2)} f_{(1)} h'_{(2)}, \\
  \delta^{\L}_{H^{*\rat}}(h \rightharpoondown f \leftharpoondown h')
  & = S^{-2}(h_{(1)}) f_{(-1)} S^2(h'_{(1)}) \otimes (h_{(2)} \rightharpoondown f_{(0)} \leftharpoondown h'_{(2)})
\end{align*}
hold for all $f \in H^{*\rat}$ and $h, h' \in H$. This means that the vector space $H^{*\rat}$ is an object of the category ${}^{\tilde{H}}_H \Mod^H_H$ by the actions \eqref{eq:Hopf-module-H*rat-actions} and the coactions \eqref{eq:Hopf-module-H*rat-coactions}.

We recall that a left cointegral on $H$ is a linear map $\lambda : H \to \bfk$ satisfying $h_{(1)} \langle \lambda, h_{(2)} \rangle = \langle \lambda, h \rangle 1_H$ for all $h \in H$. As has been observed in \cite{MR242840}, the space $I(H^{*\rat})$ coincides with the space of left cointegrals on $H$.
Hence, by the proof of Lemma~\ref{lem:Hopf-modules}, we have an isomorphism
\begin{equation}
  \label{eq:Hopf-module-iso}
  \theta: I(H^{*\rat}) \otimes H \to H^{*\rat},
  \quad \lambda \otimes h \mapsto \lambda \leftharpoondown h
\end{equation}
in the category ${}^{\tilde{H}}_H \Mod^H_H$. Now we are ready to prove:

\begin{theorem}
  \label{thm:co-Fb-cointegral}
  The space $I(H^{*\rat})$ of left cointegrals on $H$ is one-dimensional.
  We fix a non-zero left cointegral $\lambda$ on $H$.
  Then the bilinear map
  \begin{equation}
    \label{eq:co-Fb-Hopf-pairing}
    \beta : H \times H \to \bfk,
    \quad \beta(a, b) := \langle \theta(\lambda \otimes b), a \rangle = \langle \lambda, a S(b) \rangle
    \quad (a, b \in H)
  \end{equation}
  is a Frobenius pairing on $H$.
\end{theorem}
\begin{proof}
  There is a natural isomorphism $M \otimes_{H^*} H^{*\rat} \cong M$ ($M \in {}^C\Mod$) of vector spaces by Lemma \ref{lem:C-star-rat-M-rat}. Thus, by \eqref{eq:Hopf-module-iso}, we have isomorphisms
  \begin{equation*}
    I(H^{*\rat}) \otimes (\bfk \otimes_{H^*} H)
    \cong \bfk \otimes_{H^*} (I(H^{*\rat}) \otimes H)
    \cong \bfk \otimes_{H^*} H^{*\rat}
    \cong \bfk
  \end{equation*}
  of vector spaces. Therefore $I(H^{*\rat})$ is one-dimensional.

  The non-degeneracy of $\beta$ follows from that \eqref{eq:Hopf-module-iso} is bijective and that $H^{*\rat}$ is dense in $H^*$.
  Since \eqref{eq:Hopf-module-iso} is right $H$-colinear, we have
  \begin{gather*}
      \beta(a, h^* \rightharpoonup b)
      = \langle \theta(\lambda \otimes b_{(1)}), a \rangle
      \langle h^*, b_{(2)} \rangle
      = \langle \theta(\lambda \otimes b)_{(0)}, a \rangle
      \langle h^*, \theta(\lambda \otimes b)_{(1)} \rangle \\
      \mathop{=}^{\eqref{eq:Hopf-module-H*rat-coactions}}
      \langle h^* * \theta(\lambda \otimes b), a \rangle
      = \langle h^*, a_{(1)} \rangle
      \langle \theta(\lambda \otimes b), a_{(2)} \rangle
      = \beta(a \leftharpoonup h^*, b)
  \end{gather*}
  for all $a, b \in H$ and $h^* \in H^*$.
  Namely, $\beta$ is $H^*$-balanced.
  The proof is done.
\end{proof}

The bijectivity of the map \eqref{eq:Hopf-module-iso} have been proved by Sweedler \cite{MR242840} by using the fundamental theorem for Hopf modules, {\it i.e.}, Lemma \ref{lem:Hopf-modules} for $A = C = \bfk$.
Unlike \cite{MR242840}, we have used the bijectivity of the antipode of $H$.
As a trade-off, the map \eqref{eq:Hopf-module-iso} is not only a morphism of Hopf modules, but also preserves the left $\tilde{H}$-coaction and the left $H$-action. Below, we see that some useful formulas related to cointegrals can be derived by this observation.

\subsection{Nakayama automorphisms of co-Frobenius Hopf algebras}

Let $H$ be a Hopf algebra satisfying the equivalent conditions of Theorem~\ref{thm:Hopf-semiperfect}. By Theorem~\ref{thm:co-Fb-cointegral}, there is a non-zero cointegral $\lambda$ on $H$ and, moreover, the bilinear map $\beta : H \times H \to \bfk$ given by \eqref{eq:co-Fb-Hopf-pairing} is a Frobenius pairing on $H$. We denote by $\nu_H$ the Nakayama automorphism associated to $\beta$. Since the antipode of $H$ is bijective, the non-degeneracy of $\beta$ implies that the pairing $(a, b) \mapsto \lambda(a b)$ ($a, b \in H$) is also non-degenerate. According to \cite{MR2278058}, we define $\chi_H : H \to H$ by
\begin{equation}
  \label{eq:co-Fb-another-Nakayama-def}
  h \rightharpoonup \lambda = \lambda \leftharpoonup \chi_H(h)
\end{equation}
for $h \in H$ (the map $\chi_H$ is called the {\it generalized Nakayama automorphism} of $H$ in \cite{MR2278058}). We note that, unlike $\nu_H$, the map $\chi_H$ is not a coalgebra automorphism in general but an algebra automorphism of $H$.

By Theorem~\ref{thm:co-Fb-cointegral}, there is a grouplike element $g_H \in H$ and an algebra map $\alpha_H : H \to \bfk$ such that the left $\tilde{H}$-coaction and the left $H$-action on $I(H^{*\rat}) \in {}^{\tilde{H}}_H \YD$ are given by $\delta_{I(H^{*\rat})}(\lambda) = g_H \otimes \lambda$ and $h \triangleright \lambda = \alpha_H(h) \lambda$ for all $h \in H$, respectively.
The former equation is equivalent to that
\begin{equation}
  \label{eq:co-Fb-Hopf-g}
  \langle \lambda, h_{(1)} \rangle h_{(2)} = \lambda(h) g_H
\end{equation}
holds for all $h \in H$. Namely, $g_H$ is the {\em distinguished grouplike element} of $H$ introduced in \cite{MR1482980}. The Nakayama automorphism of $H$ is expressed as follows:

\begin{lemma}
  \label{lem:co-Fb-Hopf-Nakayama-2}
  $\nu_H(h) = g_H S^2(h)$ for $h \in H$.
\end{lemma}
\begin{proof}
  We use the isomorphism $\theta: I(H^{*\rat}) \otimes H \to H^{*\rat}$ of \eqref{eq:Hopf-module-iso}.
  In view of the definition of the right action $\actr$ of $H$ on $\tilde{H}$, the left $\tilde{H}$-comodule structure of the source of $\theta$ is given by the following formula:
  \begin{equation*}
    I(H^{*\rat}) \otimes H \to \tilde{H} \otimes I(H^{*\rat}) \otimes H,
    \quad \lambda \otimes h \mapsto g_H S^2(h_{(1)}) \otimes \lambda \otimes h_{(2)}.
  \end{equation*}
  We fix elements $x, y \in H$ and $\xi \in H^*$, and write $z = \lambda \otimes y$ for simplicity.
  Since the map $\theta$ is left $\tilde{H}$-colinear, we have
  \begin{gather*}
    \langle \xi, g_H S^2(y_{(1)}) \rangle \beta(x, y_{(2)})
    = \langle \xi, g_H S^2(y_{(1)}) \rangle \langle \theta(\lambda \otimes y_{(2)}), x \rangle
    = \langle \xi, z_{(-1)} \rangle \langle \theta(z_{(0)}), x \rangle \\
    = \langle \xi, \theta(z)_{(-1)} \rangle \langle \theta(z)_{(0)}, x \rangle
    \mathop{=}^{\eqref{eq:Hopf-module-H*rat-coactions}} \langle \theta(z) * \xi, x \rangle
    = \beta(x_{(1)}, y) \langle \xi, x_{(2)} \rangle.
  \end{gather*}
  Now the formula follows from Lemma~\ref{lem:Nakayama-auto-lemma-2}.
\end{proof}

The story may be summarized that the $\tilde{H}$-colinearity of $\theta$ implies a formula of the Nakayama automorphism $\nu_H$. Another property of $\theta$ yields the following formula of the algebra automorphism $\chi_H$.

\begin{lemma}
  \label{lem:co-Fb-Hopf-Nakayama-1}
  $\chi_H(h) = S^{-2}(\alpha_H \rightharpoonup h)$ for $h \in H$.
\end{lemma}
\begin{proof}
  Since $\theta$ is left $H$-linear, we have
  \begin{gather*}
    \langle \lambda, h h' \rangle
    = \langle S(h) \rightharpoondown \theta(\lambda \otimes 1), h' \rangle
    = \langle \theta(S(h)_{(1)} \triangleright \lambda \otimes S(h)_{(2)}), h' \rangle \\
    = \langle \alpha_H, S(h_{(2)}) \rangle \langle \theta(\lambda \otimes S(h_{(1)})), h' \rangle
    = \langle \lambda, h' S^2(\alpha_H^{-1} \rightharpoonup h) \rangle
  \end{gather*}
  for $h, h' \in H$, where $\alpha_H^{-1} := \alpha_H \circ S$ is the convolution inverse of $\alpha_H$. By comparing this result with \eqref{eq:co-Fb-another-Nakayama-def}, we obtain $\chi_H^{-1}(h) = S^2(\alpha_H^{-1} \rightharpoonup h)$ for $h \in H$, which implies the desired formula.
\end{proof}

Lemma~\ref{lem:co-Fb-Hopf-Nakayama-1} implies $\alpha_H = \varepsilon \circ \chi_H$. Thus, $\alpha_H$ is the distinguished grouplike element of $H^*$ in the sense of \cite[Definition 2.3]{MR2278058} and the formula of $\chi_H$ given by Lemma~\ref{lem:co-Fb-Hopf-Nakayama-1} is the same as that given in \cite[Lemma 2.5]{MR2278058}.

\subsection{Radford $S^4$-formula for co-Frobenius Hopf algebras}
\label{subsec:Hopf-Radford-S4}

We will discuss a relation between the Radford $S^4$-formula for co-Frobenius Hopf algebras \cite[Theorem 2.8]{MR2278058} and our categorical Radford $S^4$-formula (Corollary~\ref{cor:categorical-Radford-S4}) in the next subsection.
Although it is not directly related to this purpose, we include a new proof of \cite[Theorem 2.8]{MR2278058}.
Let $H$ be a co-Frobenius Hopf algebra with non-zero left cointegral $\lambda$, and define $\alpha_H$ and $g_H$ as in the previous subsection (we note that $g_H$ is the inverse of what written as $g$ in \cite[Theorem 2.8]{MR2278058}).

\begin{theorem}[{\cite[Theorem 2.8]{MR2278058}}]
  \label{thm:co-Fb-Radford-S4}
  For all $h \in H$, we have
  \begin{equation*}
    S^4(h) = g_H^{-1} (\alpha_H \rightharpoonup h \leftharpoonup \alpha_H^{-1}) g_H.
  \end{equation*}
\end{theorem}
\begin{proof}
  Since $I(H^{*\rat})$ is an object of the category ${}^{\tilde{H}}_H\YD$, we have
  \begin{equation*}
    S^{-2}(h_{(1)}) g_H \otimes \alpha_H(h_{(2)}) \lambda
    = \alpha_H(h_{(1)}) g_H S^2(h_{(2)}) \otimes \lambda
    \quad (h \in H)
  \end{equation*}
  by the YD condition \eqref{eq:YD-condition}. The formula easily follows from this.
\end{proof}

We write $\mathcal{C} = \fdmod^H$ for simplicity.
In a similar way as the Drinfeld center of a monoidal category \cite{MR3242743}, the {\em center} of a $\mathcal{C}$-bimodule category $\mathcal{M}$ is defined as the category whose object is a pair $(M, \sigma)$ consisting of an object $M \in \mathcal{M}$ and a natural isomorphism $\sigma(X): X \rhd M \to M \lhd X$ ($X \in \mathcal{C}$) satisfying the `half' of the hexagon axiom for a braiding, where $\lhd$ and $\rhd$ are the left and the right action of $\mathcal{C}$ on $\mathcal{M}$, respectively.

We denote by $\mathfrak{Z}_{2a, 2b}$ ($a, b \in \mathbb{Z}$) the center of the $\mathcal{C}$-bimodule category ${}_{(2a)}\mathcal{C}_{(2b)}$ introduced to formulate Theorem~\ref{thm:Frob-tensor-cat-Nakayama}.
This category may be called the {\em twisted Drinfeld center} of $\mathcal{C}$ as $\mathfrak{Z}_{0,0}$ is precisely the Drinfeld center of $\mathcal{C}$.
Behind the proof of Theorem \ref{thm:co-Fb-Radford-S4} lies an interpretation of ${}^{\tilde{H}}_H\YD_{\fd}$ as the twisted Drinfeld center.
For full generality, we fix integers $a$ and $b$ and let $C$ be the $H$-bicomodule coalgebra whose underlying coalgebra is $H$ but whose $H$-bimodule structure is given by
\begin{equation*}
  h \actl x \actr h' = S^{2b}(h) x S^{2a}(h')
\end{equation*}
for $x \in C$ and $h, h' \in H$ (we note that $\tilde{H}$ is the case where $(a, b) = (1, -1)$). Given objects $V \in {}^{C}_H\YD_{\fd}$ and $X \in \mathcal{C}$, we define the linear map
\begin{equation*}
  \xi_V(X): X^{(S^{2a})} \otimes V \to V \otimes X^{(S^{2b})},
  \quad x \otimes v \mapsto x_{(1)} v \otimes x_{(0)}
\end{equation*}
for $v \in V$ and $x \in X$. We view $V$ as a right $H$-comodule by the coaction
\begin{equation*}
  \delta^{\R}_V: V \to V \otimes H,
  \quad v \mapsto v_{[0]} \otimes v_{[1]} := v_{(0)} \otimes S^{-1}(v_{(-1)})
  \quad (v \in V).
\end{equation*}
Then the YD condition \eqref{eq:YD-condition-2} implies the equation
\begin{equation*}
  (h_{(2)} v)_{[0]} \otimes (h_{(2)} v)_{[1]} S^{2b}(h_{(1)})
  = h_{(1)} v_{[0]} \otimes S^{2a}(h_{(2)}) v_{[1]}
  \quad (h \in H, v \in V),
\end{equation*}
which is equivalent to that $\xi_V(X)$ is right $H$-colinear for all $X \in \mathcal{C}$.
One can verify that the pair $(V, \xi_V)$ becomes an object of $\mathfrak{Z}_{2a,2b}$ if we identify the functors $(-)^{\vee\vee}$ and ${}^{\vee\vee}(-)$ with $(-)^{(S^2)}$ and $(-)^{(S^{-2})}$, respectively, through the canonical isomorphism $\phi_X: X \to X^{**}$ for $X \in \Vect_{\fd}$.
This construction gives rise to a category isomorphism between $\mathfrak{Z}_{2a,2b}$ and ${}^C_H\YD_{\fd}$.
Letting $(a, b) = (1, -1)$, we see that ${}^{\tilde{H}}_H\YD_{\fd}$ is identified with $\mathfrak{Z}_{2,-2}$.

\begin{remark}
There are some other important variations:
The well-known isomorphism ${}^H_H \YD_{\fd} \cong \mathcal{Z}(\mathcal{C})$ is the case where $(a, b) = (0, 0)$. The category of left-left {\em anti-Yetter-Drinfeld modules} \cite[Definition 2.1]{MR2056464} is the case where $(a, b) = (-1, 0)$, and thus its full subcategory of finite-dimensional objects is isomorphic to $\mathfrak{Z}_{-2, 0}$.
\end{remark}

The object of $\mathfrak{Z}_{2,-2}$ corresponding to $I(H^{*\rat}) \in {}^{\tilde{H}}_{H} \YD_{\fd}$ is the right $H$-comodule $\bfk g_H^{-1}$ together with the natural isomorphism $\sigma$ given by
\begin{equation*}
  \sigma_X: X^{\vee\vee} \otimes \bfk g_H^{-1}
  \to \bfk g_H^{-1} \otimes {}^{\vee\vee}\!X,
  \quad \phi_X(x) \otimes g_H^{-1}
  \mapsto \alpha_H(x_{(1)}) g_H^{-1} \otimes \phi_X(x_{(0)})
\end{equation*}
for $x \in X \in \mathcal{C}$.
By applying the functor $\bfk g_H \otimes (-) \otimes \bfk g_H$ to the inverse of $\sigma_{X^{\vee\vee}}$ and composing some obvious isomorphisms, we obtain the natural isomorphism
\begin{equation*}
  \RadIso'_X: X \otimes \bfk g_H \to \bfk g_H \otimes X^{\vee\vee\vee\vee},
  \quad x \otimes g_H \mapsto \langle \alpha^{-1}_H, x_{(1)} \rangle g_H \otimes \phi_{X^{**}}\phi_X(x_{(0)})
\end{equation*}
for $x \in X \in \mathcal{C}$. This gives the object
\begin{equation}
  \label{eq:Radford-iso-from-YD-2}
  (\bfk g_H, \RadIso') \in \mathfrak{Z}_{0,4},
\end{equation}
which is, in a sense, equivalent to our categorical Radford $S^4$-formula as we will see in the next subsection.

\subsection{Equivalence to the categorical version}

We keep the same notation as in the previous subsection.
Here we give an explicit formula of the Radford isomorphism $\RadIso$ of $\mathcal{C} = \fdmod^H$ and explain its relation to the object \eqref{eq:Radford-iso-from-YD-2}.
We recall that the isomorphism $\RadIso_X$ is given in terms of the `twisted' $\mathcal{C}$-bimodule structure of the right exact Nakayama functor $\Nakr = \Nakr_H$, which we denote by
\begin{equation*}
  \Psi^{\L}_{X,Y}: \Nakr(X \otimes Y) \to {}^{\vee\vee}\!X \otimes \Nakr(Y), \quad
  \Psi^{\R}_{X,Y}: \Nakr(X \otimes Y) \to \Nakr(X) \otimes Y^{\vee\vee}
\end{equation*}
for $X, Y \in \mathcal{C}$. By Definition~\ref{def:Radford-iso}, we have
\begin{equation}
  \RadIso_X = \Psi^{\R}_{\unitobj,X^{\vee\vee}} \circ (\Psi^{\L}_{X^{\vee\vee},\unitobj})^{-1}
  \circ (\phi_{X^{**}} \phi_X \otimes \id_{\modobj})
\end{equation}
for $X \in \mathcal{C}$, where $\modobj = \Nakr(\unitobj)$. We shall compute $\Psi^{\L}_{X,Y}$ and $\Psi^{\R}_{X,Y}$.

\begin{lemma}
  For $x \in X$, $h \in H$ and $y \in Y$, we have
  \begin{align}
    \label{eq:co-Fb-Hopf-Nakayama-structure-1}
    \Psi^{\L}_{X,Y}(h \otimes_{H^*} (x \otimes y))
    & = \phi_X(x_{(0)}) \otimes (S^{-1}(x_{(1)}) h \otimes_{H^*} y), \\
    \label{eq:co-Fb-Hopf-Nakayama-structure-2}
    \Psi^{\R}_{X,Y}(h \otimes_{H^*} (x \otimes y))
    & = (h S(y_{(1)}) \otimes_{H^*} x) \otimes \phi_Y(y_{(0)}).
  \end{align}
\end{lemma}
\begin{proof}
  We only show that the map $\Psi^{\L}_{X,Y}$ is expressed as stated, since the expression for $\Psi^{\R}_{X,Y}$ can be verified in a similar way.
  In view of Lemma~\ref{lem:coHom-finite-2}, we may assume that the coHom functor for $\Ind(\mathcal{C}) = \Mod^H$ is given by $\coHom(V, W) := V^* \otimes_{H^*} W$ for $V, W \in \mathcal{C}$. The adjunction isomorphism
  \begin{equation*}
    \Theta_{V, W, X}: \coHom(V, X \otimes W)
    \xrightarrow{\ \eqref{eq:adjunction-coHom} \ }
    \coHom({}^{\vee}\!X \otimes V, W)
    \quad (V, W, X \in \mathcal{C})
  \end{equation*}
  associated to the adjunction $X \otimes (-) \dashv {}^{\vee}\!X \otimes (-)$ is given by
  \begin{equation*}
    \Theta_{V, W, X}(v^{*} \otimes_{H^*} (x \otimes w))
    = (\phi_X(x) \otimes v^*) \otimes_{H^*} w
    \quad (v^* \in V^*, x \in X, w \in W)
  \end{equation*}
  under the identification $(X^* \otimes V)^* \cong X^{**} \otimes V^*$ of vector spaces.
  Let $M \in \mathcal{C}$. By the proof of Lemma~\ref{thm:Nakayama-(co)end}, the universal dinatural transformation for $\Nakr(M)$ as a coend is given by
  \begin{gather*}
    i_V(M) : \coHom(V, M) \otimes V = (V^* \otimes_{H^*} M) \otimes V \to H \otimes_{H^*} M = \Nakr(M), \\
    (v^{*} \otimes_{H^*} m) \otimes v \mapsto \langle v^*, v_{(0)} \rangle v_{(1)} \otimes_{H^*} m
    \quad (v^* \in V^*, v \in V, m \in M)
  \end{gather*}
  for $V \in \mathcal{C}$. Since the forgetful functor $\Mod^H \to \Vect$ preserves small colimits \cite{MR414567}, the morphism $\Psi^{\L}_{X,Y}$ is a unique linear map making the diagram
  \begin{equation*}
    \begin{tikzcd}[column sep = 80pt]
      \Nakr(X \otimes Y)
      \arrow[ddd, "{\Psi^{\L}_{X,Y}}"']
      & \coHom(V, X \otimes Y) \otimes V
      \arrow[d, "(1)"', "\Theta_{V, Y, X} \otimes \id"]
      \arrow[l, "i_{V}(X \otimes Y)"'] \\
      & \coHom({}^{\vee}X \otimes V, Y) \otimes V
      \arrow[d, "(2)"', "\id \otimes \coev_{{}^{\vee\vee} \! X} \otimes \id"] \\
      & \coHom({}^{\vee}X \otimes V, Y) \otimes {}^{\vee\vee}X \otimes {}^{\vee}X \otimes V
      \arrow[d, "(3)"', "{\tau_{\coHom({}^{\vee} \! X \otimes V, Y), {}^{\vee\vee} \! X }\otimes \id \otimes \id}"] \\
      {}^{\vee\vee}\!X \otimes \Nakr(Y)
      & {}^{\vee\vee} \! X \otimes \coHom({}^{\vee}X \otimes V, Y) \otimes {}^{\vee}X \otimes V
      \arrow[l, "(4)"', "\id \otimes i_{{}^{\vee}\!X \otimes V}(Y)"]
    \end{tikzcd}
  \end{equation*}
  commute for all $V \in \mathcal{C}$, where $\tau$ is the natural isomorphism of Lemma~\ref{lem:NS-Lemma-7-1} (which just permutes tensorands in this case).
  We fix a basis $\{ x_j \}$ of $X$, and let $\{ x^j \}$ be the basis of ${}^{\vee}\!X$ ($=X^*$) dual to $\{ x_j \}$.
  An element
  \begin{gather*}
    t := (v^* \otimes_{H^*} (x \otimes y)) \otimes v
    \in \coHom(V, X \otimes Y) \otimes V \\
    (v^* \in V^*, x \in X, y \in Y, v \in V)
  \end{gather*}
  is sent along the path through the lower right corner as follows:
  \begin{align*}
    t & \xmapsto{\ (1) \ } ((\phi_X(x) \otimes v^*) \otimes_{H^*} y) \otimes v \\
    & \xmapsto{\ (2) \ } ((\phi_X(x) \otimes v^*) \otimes_{H^*} y) \otimes \phi_X(x_j) \otimes x^j \otimes v \\
    & \xmapsto{\ (3) \ } \underbrace{\phi_X(x_j)}_{X^{\vee\vee}} \mbox{}
    \otimes ((\underbrace{\phi_X(x) \otimes v^*}_{({}^{\vee\!}X \otimes V)^*}) \otimes_{H^*} y) \otimes \underbrace{x^j \otimes v_{}}_{{}^{\vee\!}X \otimes V} \\
    & \xmapsto{\ (4) \ } \phi_X(x_j) \otimes \langle \phi_X(x) \otimes v^*, (x^j \otimes v)_{(0)} \rangle (x^j \otimes v)_{(1)} \otimes_{H^*} y  \\
    & = \phi_X(x_j) \otimes \langle \phi_X(x), x^j_{(0)} \rangle \langle v^*, v_{(0)} \rangle x^j_{(1)} v_{(1)} \otimes_{H^*} y  \\
    & = \phi_X(x_j) \otimes \langle x^j, x_{(0)} \rangle \langle v^*, v_{(0)} \rangle S^{-1}(x_{(1)}) v_{(1)} \otimes_{H^*} y  \\
    & = \phi_X(x_{(0)}) \otimes \langle v^*, v_{(0)} \rangle S^{-1}(x_{(1)}) v_{(1)} \otimes_{H^*} y,
  \end{align*}
  where the summation over $j$ is suppressed and the second equality follows from the definition of the coaction of $H$ on ${}^{\vee\!}X$. We also have
  \begin{align*}
    t
    & \xmapsto{\makebox[5em]{\scriptsize $i_V(X \otimes Y)$}}
      \langle v^*, v_{(0)} \rangle v_{(1)} \otimes_{H^*} (x \otimes y) \\
    & \xmapsto{\makebox[5em]{\scriptsize \eqref{eq:co-Fb-Hopf-Nakayama-structure-1}}}
      \phi_X(x_{(0)}) \otimes \langle v^*, v_{(0)} \rangle S^{-1}(x_{(1)}) v_{(1)} \otimes_{H^*} y.
  \end{align*}
  Thus \eqref{eq:co-Fb-Hopf-Nakayama-structure-1} makes the diagram commute. The proof is done.
\end{proof}

\begin{lemma}
  The Radford isomorphism in $\mathcal{C}$ is given by
  \begin{equation}
    \label{eq:Radford-iso-H-comod}
    \RadIso_X(x \otimes (h \otimes_{H^*} 1))
    = (x_{(2)} h S^{3}(x_{(1)}) \otimes_{H^*} 1) \otimes \phi_{X^{**}}\phi_X(x_{(0)})
  \end{equation}
  for $x \in X \in \mathcal{C}$ and $h \in H$.
\end{lemma}
\begin{proof}
  We first note that the inverse of \eqref{eq:co-Fb-Hopf-Nakayama-structure-1} is given by
  \begin{equation*}
    (\Psi^{\L}_{X,Y})^{-1}(\phi_X(x) \otimes (h \otimes_{H^*} y))
    = S^{-2}(x_{(1)}) h \otimes_{H^*} (x_{(0)} \otimes y)
  \end{equation*}
  for $X, Y \in \fdmod^{H}$, $x \in X$, $y \in Y$ and $h \in H$. Thus we have
  \begin{align*}
    \Psi^{\R}_{\unitobj,X} (\Psi^{\L}_{X,\unitobj})^{-1}
    (\phi_X(x) \otimes (h \otimes_{H^*} 1))
    & = \Psi^{\R}_{\unitobj, X}(S^{-2}(x_{(1)}) h \otimes_{H^*} (x_{(0)} \otimes 1)) \\
    & = \Psi^{\R}_{\unitobj, X}(S^{-2}(x_{(1)}) h \otimes_{H^*} (1 \otimes x_{(0)})) \\
    & = (S^{-2}(x_{(2)}) h S(x_{(1)}) \otimes_{H^*} 1) \otimes \phi_X(x_{(0)})
  \end{align*}
  for all $x \in X$ and $h \in H$, where we have identified $X \otimes \unitobj$ with $\unitobj \otimes X$ at the second equality. The desired formula is obtained by replacing $X$ with $X^{\vee\vee}$.
\end{proof}

By the proof of Lemma~\ref{lem:Nakayama-co-Fro}, there is a natural isomorphism
\begin{equation*}
  \beta^{\natural}_{M} : \Nakr(M) \to M^{(\nu_H)},
  \quad h \otimes_{H^*} m \mapsto \langle \lambda, h S(m_{(1)}) \rangle m_{(0)}
  \quad (M \in \Mod^H)
\end{equation*}
associated to the Frobenius pairing \eqref{eq:co-Fb-Hopf-pairing}.
By Lemma \ref{lem:co-Fb-Hopf-Nakayama-2}, $\unitobj^{(\nu_H)}$ is isomorphic to $\bfk g_H$ as a right $H$-comodule. Thus we have an isomorphism
\begin{equation*}
  \kappa: \Nakr(\unitobj) \to \bfk g_H,
  \quad h \otimes_{H^*} 1_{\bfk} \mapsto \langle \lambda, h \rangle g_H
  \quad (h \in H)
\end{equation*}
of right $H$-comodules.

\begin{lemma}
  \label{lem:Radford-iso-H-comod-2}
  For all $X \in \mathcal{C}$, the following diagram commutes:
  \begin{equation*}
    \begin{tikzcd}[column sep = 80pt]
      X \otimes \Nakr(\unitobj)
      \arrow[d, "{\RadIso_X}"']
      \arrow[r, "{\id_X \otimes \kappa}"]
      & X \otimes \bfk g_H
      \arrow[d, "{\RadIso'_X}"] \\
      \Nakr(\unitobj) \otimes X^{\vee 4}
      \arrow[r, "{\kappa \otimes \id_{X^{\vee 4}}}"']
      & \bfk g_H \otimes X^{\vee 4}
    \end{tikzcd}
  \end{equation*}
\end{lemma}

Thus the object \eqref{eq:Radford-iso-from-YD-2} is isomorphic to $(\Nakr(\unitobj), \RadIso)$ in $\mathfrak{Z}_{0,4}$. In other words, our categorical Radford $S^4$-formula is equivalent to \cite[Theorem 2.8]{MR2278058} through the interpretation of the category ${}^{\tilde{H}}_H \YD$ as the twisted Drinfeld center.

\begin{proof}
  For $x \in X \in \mathcal{C}$ and $h \in H$, we compute
  \begin{align*}
    & (\kappa \otimes \id_{X^{\vee4}}) \RadIso_X (x \otimes (h \otimes_{H^*} 1_{\bfk})) \\
    {}^{\eqref{eq:Radford-iso-H-comod}}\!
    & = \langle \lambda, x_{(2)} h S^3(x_{(1)}) \rangle g_H \otimes \phi_{X^{**}}\phi_X(x_{(0)}) \\
    {}^{\eqref{eq:co-Fb-another-Nakayama-def}}\!
    & = \langle \lambda, h S^{3}(x_{(1)}) \chi_H^{-1}(x_{(2)}) \rangle g_H \otimes \phi_{X^{**}}\phi_{X}(x_{(0)}) \\
    {}^{\text{(Lemma~\ref{lem:co-Fb-Hopf-Nakayama-1})}}\!
    & = \langle \lambda, h S^{3}(x_{(1)}) S^2(x_{(2)}) \rangle \langle \alpha_H^{-1}, S^2(x_{(3)}) \rangle g_H \otimes \phi_{X^{**}}\phi_{X}(x_{(0)}) \\
    & = \langle \lambda, h \rangle \langle \alpha_H^{-1}, x_{(1)} \rangle g_H \otimes \phi_{X^{**}}\phi_{X}(x_{(0)}) \\
    & = \RadIso'_X (\id_X \otimes \kappa) (x \otimes (h \otimes_{H^*} 1_{\bfk})). \qedhere
  \end{align*}
\end{proof}

Thanks to this lemma, one can interpret Theorem~\ref{thm:Radford-iso-braiding} in the context of Hopf algebras. We suppose that $H$ possesses a universal R-form $r: H \times H \to \bfk$ so that the monoidal category $\mathcal{C}$ has the braiding given by
\begin{equation*}
  \sigma_{X,Y} : X \otimes Y \to Y \otimes X,
  \quad x \otimes y \mapsto r(x_{(1)}, y_{(1)}) y_{(0)} \otimes x_{(0)}
\end{equation*}
for $X, Y \in \mathcal{C}$. We define $u, v, b \in H^*$ by
\begin{equation*}
  u(h) = r(h_{(2)}, S(h_{(1)})),
  \quad v = u \circ S,
  \quad b(h) = r(g_H, h)
  \quad (h \in H).
\end{equation*}
For reader's convenience, we note that our $g_H$ and $b$ correspond to $a^{-1}$ and $\beta_a^{-1}$ of \cite{MR2554186}, respectively.
The Drinfeld isomorphism~\eqref{eq:Drinfeld-iso-def} is given by
\begin{equation*}
  \DriIso_X : X \to X^{\vee\vee},
  \quad x \mapsto \phi_X(u \rightharpoonup x)
  \quad (X \in \mathcal{C}).
\end{equation*}
Thus, by Theorem~\ref{thm:Radford-iso-braiding} and Lemma~\ref{lem:Radford-iso-H-comod-2}, we have
\begin{align*}
  (\RadIso'_X)^{-1}(g_H \otimes x)
  & = \sigma_{\bfk g_H, X} (\id_{\bfk g_H} \otimes \DriIso_X^{-1} \DriIso_{X^{\vee}}^{\vee} \phi_{X^{**}}\phi_X)(g_H \otimes x) \\
  & = \langle b * u^{-1} * v, x_{(1)} \rangle x_{(0)} \otimes g_H
\end{align*}
for all $x \in X \in \mathcal{C}$. This implies $\alpha_H = b * u^{-1} * v$ \cite[Theorem 3.3]{MR2554186}.


\begin{thebibliography}{BDRVO98}

\bibitem[ACE15]{MR3410615}
N. Andruskiewitsch, J. Cuadra, and P. Etingof.
\newblock On two finiteness conditions for {H}opf algebras with nonzero
  integral.
\newblock {\em Ann. Sc. Norm. Super. Pisa Cl. Sci. (5)}, 14(2):401--440, 2015.

\bibitem[BB09]{MR2554186}
M. Beattie and D. Bulacu.
\newblock On the antipode of a co-{F}robenius (co)quasitriangular {H}opf
  algebra.
\newblock {\em Comm. Algebra}, 37(9):2981--2993, 2009.

\bibitem[BBT07]{MR2278058}
M. Beattie, D. Bulacu, and B. Torrecillas.
\newblock Radford's {$S^4$} formula for co-{F}robenius {H}opf algebras.
\newblock {\em J. Algebra}, 307(1):330--342, 2007.

\bibitem[BDR97]{MR1482980}
M.~Beattie, S.~D\u{a}sc\u{a}lescu, and \c{S}. Raianu.
\newblock Galois extensions for co-{F}robenius {H}opf algebras.
\newblock {\em J. Algebra}, 198(1):164--183, 1997.

\bibitem[BDRVO98]{MR1629389}
M.~Beattie, S.~D\u{a}sc\u{a}lescu, \c{S}. Raianu, and F.~Van~Oystaeyen.
\newblock The categories of {Y}etter-{D}rinfel\cprime d modules, {D}oi-{H}opf
  modules and two-sided two-cosided {H}opf modules.
\newblock {\em Appl. Categ. Structures}, 6(2):223--237, 1998.

\bibitem[BV12]{MR2869176}
A. Brugui\`eres and A. Virelizier.
\newblock Quantum double of {H}opf monads and categorical centers.
\newblock {\em Trans. Amer. Math. Soc.}, 364(3):1225--1279, 2012.

\bibitem[BW99]{MR1686423}
J.~W. Barrett and B.~W. Westbury.
\newblock Spherical categories.
\newblock {\em Adv. Math.}, 143(2):357--375, 1999.

\bibitem[CGT02]{MR1904645}
J.~Cuadra and J.~G\'{o}mez-Torrecillas.
\newblock Idempotents and {M}orita-{T}akeuchi theory.
\newblock {\em Comm. Algebra}, 30(5):2405--2426, 2002.

\bibitem[CnIDN04]{MR2078404}
F.~Casta\~{n}o Iglesias, S.~D\u{a}sc\u{a}lescu, and C.~N\u{a}st\u{a}sescu.
\newblock Symmetric coalgebras.
\newblock {\em J. Algebra}, 279(1):326--344, 2004.

\bibitem[CS10]{MR2684139}
W. Chin and D. Simson.
\newblock Coxeter transformation and inverses of {C}artan matrices for
  coalgebras.
\newblock {\em J. Algebra}, 324(9):2223--2248, 2010.

\bibitem[Cua06]{MR2236104}
J. Cuadra.
\newblock On Hopf algebras with nonzero integral.
\newblock {\em Comm. Algebra}, 34(6):2143–2156, 2006.

\bibitem[DNR01]{MR1786197}
S. D\u{a}sc\u{a}lescu, C. N\u{a}st\u{a}sescu, and \c{S}. Raianu.
\newblock {\em Hopf algebras}, volume 235 of {\em Monographs and Textbooks in
  Pure and Applied Mathematics}.
\newblock Marcel Dekker, Inc., New York, 2001.
\newblock An introduction.

\bibitem[DNT99]{MR1717358}
S.~D\u{a}sc\u{a}lescu, C.~N\u{a}st\u{a}sescu, and B.~Torrecillas.
\newblock Co-{F}robenius {H}opf algebras: integrals, {D}oi-{K}oppinen modules
  and injective objects.
\newblock {\em J. Algebra}, 220(2):542--560, 1999.

\bibitem[DSPS20]{MR4254952}
C.~L. Douglas, Christopher Schommer-Pries, and Noah Snyder.
\newblock Dualizable tensor categories.
\newblock {\em Mem. Amer. Math. Soc.}, 268(1308):vii+88, 2020.

\bibitem[EGNO15]{MR3242743}
P. Etingof, S. Gelaki, D. Nikshych, and V. Ostrik.
\newblock {\em Tensor categories}, volume 205 of {\em Mathematical Surveys and
  Monographs}.
\newblock American Mathematical Society, Providence, RI, 2015.

\bibitem[ENO04]{MR2097289}
P. Etingof, D. Nikshych, and V. Ostrik.
\newblock An analogue of {R}adford's {$S^4$} formula for finite tensor
  categories.
\newblock {\em Int. Math. Res. Not.}, (54):2915--2933, 2004.

\bibitem[FSS20]{MR4042867}
J. Fuchs, G. Schaumann, and C. Schweigert.
\newblock Eilenberg-{W}atts calculus for finite categories and a bimodule
  {R}adford {$S^4$} theorem.
\newblock {\em Trans. Amer. Math. Soc.}, 373(1):1--40, 2020.

\bibitem[GTMN03]{MR1998048}
J.~G\'{o}mez-Torrecillas, C.~Manu, and C.~N\u{a}st\u{a}sescu.
\newblock Quasi-co-{F}robenius coalgebras. {II}.
\newblock {\em Comm. Algebra}, 31(10):5169--5177, 2003.

\bibitem[HKRS04]{MR2056464}
P.~M.~Hajac, M.~Khalkhali, B.~Rangipour, and Y.~Sommerh\"{a}user.
\newblock Stable anti-{Y}etter--{D}rinfeld modules.
\newblock {\em C. R. Math. Acad. Sci. Paris}, 338(8):587--590, 2004.

\bibitem[Iov06]{MR2253657}
M. C. Iovanov.
\newblock Co-{F}robenius coalgebras.
\newblock {\em J. Algebra}, 303(1):146--153, 2006.

\bibitem[Iov13]{MR3125851}
M. C. Iovanov.
\newblock Abstract algebraic integrals and {F}robenius categories.
\newblock {\em Internat. J. Math.}, 24(10):1350081, 36, 2013.

\bibitem[Iov14]{MR3150709}
M. C. Iovanov.
\newblock Generalized {F}robenius algebras and {H}opf algebras.
\newblock {\em Canad. J. Math.}, 66(1):205--240, 2014.

\bibitem[KS06]{MR2182076}
M. Kashiwara and P. Schapira.
\newblock {\em Categories and sheaves}, volume 332 of {\em Grundlehren der
  Mathematischen Wissenschaften [Fundamental Principles of Mathematical
  Sciences]}.
\newblock Springer-Verlag, Berlin, 2006.

\bibitem[ML98]{MR1712872}
S.~Mac~Lane.
\newblock {\em Categories for the working mathematician}, volume~5 of {\em
  Graduate Texts in Mathematics}.
\newblock Springer-Verlag, New York, second edition, 1998.

\bibitem[NS07]{MR2381536}
S.-H.~Ng and P.~Schauenburg.
\newblock Higher {F}robenius-{S}chur indicators for pivotal categories.
\newblock In {\em Hopf algebras and generalizations}, volume 441 of {\em
  Contemp. Math.}, pages 63--90. Amer. Math. Soc., Providence, RI, 2007.

\bibitem[Rad76]{MR407069}
D. E. Radford.
\newblock The order of the antipode of a finite dimensional {H}opf algebra is finite,
\newblock {\em Amer. J. Math.}, 98(2):333--355, 1976.

\bibitem[\c{S}te95]{MR1323693}
D. \c{S}tefan.
\newblock The uniqueness of integrals (a homological approach).
\newblock {\em Comm. Algebra}, 23:1657-–1662, 1995.

\bibitem[Sch92]{MR1623637}
P.~Schauenburg.
\newblock {\em Tannaka duality for arbitrary {H}opf algebras}, volume~66 of
  {\em Algebra Berichte [Algebra Reports]}.
\newblock Verlag Reinhard Fischer, Munich, 1992.

\bibitem[Shi17a]{MR3569179}
K.~Shimizu.
\newblock The relative modular object and {F}robenius extensions of finite
  {H}opf algebras.
\newblock {\em J. Algebra}, 471:75--112, 2017.

\bibitem[{Shi}17b]{2017arXiv170709691S}
K.~Shimizu.
\newblock {Ribbon structures of the Drinfeld center of a finite tensor
  category}.
\newblock {\em arXiv e-prints}, page arXiv:1707.09691, July 2017.

\bibitem[Shi19]{MR3996323}
K.~Shimizu.
\newblock Non-degeneracy conditions for braided finite tensor categories.
\newblock {\em Adv. Math.}, 355:106778, 36, 2019.

\bibitem[Ste75]{MR0389953}
B.~Stenstr\"om.
Rings of quotients.
Die Grundlehren der mathematischen Wissenschaften, Band 217.
An introduction to methods of ring theory. Springer-Verlag, New York-Heidelberg, 1975.

\bibitem[SS21]{2021arXiv210313702S}
T.~Shibata and K.~Shimizu.
\newblock {Modified traces and the Nakayama functor}.
\newblock {\em arXiv e-prints}, page arXiv:2103.13702, March 2021.

\bibitem[SW21]{2021arXiv210315772S}
C.~Schweigert and L.~Woike.
\newblock {The Trace Field Theory of a Finite Tensor Category}.
\newblock {\em arXiv e-prints}, page arXiv:2103.15772, March 2021.

\bibitem[Swe69]{MR242840}
M.~E.~Sweedler.
\newblock Integrals for {H}opf algebras.
\newblock {\em Ann. of Math. (2)}, 89:323--335, 1969.

\bibitem[Tak71]{MR292876}
M.~Takeuchi.
\newblock Free {H}opf algebras generated by coalgebras.
\newblock {\em J. Math. Soc. Japan}, 23:561--582, 1971.

\bibitem[Tak77]{MR472967}
M.~Takeuchi.
\newblock Morita theorems for categories of comodules.
\newblock {\em J. Fac. Sci. Univ. Tokyo Sect. IA Math.}, 24(3):629--644, 1977.

\bibitem[Ulb90]{MR1098991}
K.-H.~Ulbrich.
\newblock On {H}opf algebras and rigid monoidal categories.
\newblock volume~72, pages 252--256. 1990.
\newblock Hopf algebras.

\bibitem[Wis75]{MR414567}
M.~B.~Wischnewsky.
\newblock On linear representations of affine groups. {I}.
\newblock {\em Pacific J. Math.}, 61(2):551--572, 1975.

\end{thebibliography}

\def\cprime{$'$}

\end{document}